\documentclass[twoside,11pt]{article}

\usepackage{graphicx, subfigure}
\usepackage[top=1in,bottom=1in,left=1in,right=1in]{geometry}
\usepackage[numbers,sort&compress]{natbib} 
\usepackage{amsfonts, amsmath, amssymb, amsthm, bbm}
\usepackage[normalem]{ulem}
\usepackage{comment}
\usepackage{eufrak}

\usepackage{booktabs}

\usepackage{color}
\definecolor{darkred}{RGB}{100,0,0}
\definecolor{darkgreen}{RGB}{0,100,0}
\definecolor{darkblue}{RGB}{0,0,150}

\usepackage{hyperref}
\hypersetup{colorlinks=true, linkcolor=darkred, citecolor=darkgreen, urlcolor=darkblue}
\usepackage{url}

\newtheorem{thm}{Theorem}
\newtheorem{prp}{Proposition}
\newtheorem{lem}{Lemma}
\newtheorem{cor}{Corollary}

\def\beq{\begin{equation}}
\def\eeq{\end{equation}}
\def\beqn{\begin{eqnarray*}}
\def\eeqn{\end{eqnarray*}}
\def\bitem{\begin{itemize}}
\def\eitem{\end{itemize}}
\def\benum{\begin{enumerate}}
\def\eenum{\end{enumerate}}
\def\bmult{\begin{multline*}}
\def\emult{\end{multline*}}
\def\bcenter{\begin{center}}
\def\ecenter{\end{center}}

\def\cA{\mathcal{A}}
\def\cB{\mathcal{B}}

\def\cE{\mathcal{E}}

\def\cK{\mathcal{K}}
\def\cL{\mathcal{L}}

\def\cN{\mathcal{N}}

\def\cP{\mathcal{P}}

\def\cR{\mathcal{R}}
\def\cS{\mathcal{S}}
\def\cT{\mathcal{T}}
\def\cU{\mathcal{U}}

\def\bX{\mathbf{X}}

\def\1{{\mathbf 1}}

\def\bbB{\mathbb{B}}

\def\bbP{\mathbb{P}}
\def\bbQ{\mathbb{Q}}
\def\bbR{\mathbb{R}}

\newcommand{\E}{\operatorname{\mathbb{E}}}
\renewcommand{\P}{\operatorname{\mathbb{P}}}
\newcommand{\Var}{\operatorname{Var}}

\newcommand{\var}[1]{\operatorname{Var}\left(#1\right)}

\title{Adaptive estimation of the sparsity in the Gaussian vector model}
\author{Alexandra Carpentier\footnote{Universit\"at Potsdam, Institut f\"ur Mathematik, Karl-Liebknecht-Straße 24-25, 14476 Potsdam, GERMANY, \url{carpentier@math.uni-potsdam.de}}~\ and Nicolas Verzelen\footnote{INRA, UMR 729 MISTEA, F-34060 Montpellier, FRANCE, \url{nicolas.verzelen@inra.fr}}}

\begin{document}

\maketitle
\begin{abstract}
Consider the Gaussian vector model with mean value $\theta$. We study the twin problems of estimating the number $\|\theta\|_0$ of non-zero components of $\theta$ and testing whether $\|\theta\|_0$ is  smaller than some value. For testing, we establish the minimax separation distances for this model and introduce a minimax adaptive test. 
Extensions to the case of unknown variance are also discussed. Rewriting the estimation of $\|\theta\|_0$ as a multiple testing problem of all hypotheses $\{\|\theta\|_0\leq q\}$, we  both derive a new way of assessing the optimality of a sparsity estimator and we exhibit such an optimal procedure. This general approach provides a roadmap for estimating the complexity of the signal in various statistical models. 
\end{abstract}

\section{Introduction}

Many estimation methods in high or infinite-dimensional statistics rely on the assumption that the parameter of interest belongs to some smaller parameter space. Depending on the problem at hand, the assumptions on the structure of the unknown parameter 
take various forms. In high-dimensional linear regression, it is usually assumed that the regression parameter is sparse~\cite{buhlmann2011statistics}. In matrix completion, the underlying matrix may be supposed to be low-rank~\cite{keshavan2010matrix}. In density estimation, 
many nonparametric methods are based on the assumption that the function satisfies some smoothness properties~\cite{hardle2012wavelets}. Many Model-based clustering methods require the data to follow a mixture distribution with several Gaussian components~\cite{ESL}.
In practice, the exact complexity of the parameter (e.g.\  the rank of the matrix, the smoothness of the function) is unknown. 
Although a lot of work has been devoted to the construction of statistical procedures performing as well as if the model complexity was known~(e.g.\ ~\cite{buhlmann2011statistics,massart2007concentration,gine2015mathematical}), the literature on the estimation of the complexity of the parameter is scarcer.

In this paper, we are interested in the twin  problems of estimating the complexity of the parameter and testing whether the parameter belongs to some complexity class. There are several motivations for these problems. First, complexity estimation allows to assess the relevance of specific parameter estimation approaches. For instance, inferring the smoothness of a function allows to justify the use of regularity-based procedures. Second, the construction of adaptive confidence regions is closely connected to the model testing problem since the size of a good confidence region should depend on the complexity of the unknown parameter~\cite{MR2906872}. Finally, in some practical applications, the primary objective is rather to evaluate the complexity of the parameter than the parameter itself. This is for instance the case in some heritability studies where the goal is to decipher whether a trait is multigenic or ``highly polygenic'' which amounts to inferring whether a high-dimensional regression parameter is sparse or dense \cite{maher:2008,toro:2014}.

\medskip

In this paper, we 
focus on a comparatively simple, yet emblematic setting, namely the Gaussian vector model, that we define as follows :
\beq\label{eq:model} 
Y_i= \theta_i + \epsilon_i , \quad i=1,\ldots, n\ ,
\eeq
where $\theta=(\theta_i) \in \mathbb R^n$ is unknown and the noise components $\epsilon_i$ are independent and follow a centered normal distributions with variance $\sigma^2$. We are interested  in (i) estimating the number $\|\theta\|_0$ of non-zero components of $\theta$ and (ii) given some non-negative integer $k_0$, testing whether $\|\theta\|_0\leq k_0$ or $\|\theta\|_0> k_0$. The former problem is called sparsity estimation and the latter sparsity testing. 

\subsection{Sparsity testing and separation distances}

As the sparsity testing problem is easier to formalize than the sparsity estimation problem, let us be first more specific about it. 
Given a non-negative integer $k_0\in [0,n]$, we write 
\beq\label{eq:def_B0}
\bbB_0[k_0] := \{\theta \in \mathbb R^n : \|\theta\|_0 \leq k_0\}\ ,
 \eeq
for the set of $k_0$-sparse vectors $\theta$, that is to say the set of vectors $\theta$ with less than $k_0$ non-zero coefficients. Our goal is to test whether $\theta$ belongs to $\bbB_0[k_0]$ or not. In order to assess the performances of a test, we need to specify a rejection region and a risk. Before describing our results and the literature, we shall first define the notion  of minimax separation distance of a test.

Let $\|.\|_2$ stand for the Euclidean norm in $\bbR^n$. 
For any $\theta\in \mathbb{R}^n$, we write $d_2(\theta,\bbB_0[k_0]) := \inf_{u \in \bbB_0[k_0]} \|\theta - u\|_2$ for the distance of $\theta$ to the set of $k_0$-sparse vectors. Intuitively, any $\alpha$-level test $T$ of the null hypothesis 
$\{\theta \in \mathbb B_0[k_0]\}$ cannot reject the null with high probability when the true parameter is  arbitrarily close (in the $d_2(\theta,\bbB_0[k_0])$ sense)  to  $\bbB_0[k_0]$. Conversely, any reasonable test should reject the null hypothesis with high probability for parameters $\theta$ that are really distant  to $\bbB_0[k_0]$. In order to quantify the performances of a given test $T$, it is then classical~\cite{baraud02,ingster2012nonparametric} to rely on the notion of separation distance. Given positive integers $k_1>k_0$ and a real number $\rho>0$, define 
\beq\label{eq:defB1}
\bbB_0[k_1,k_0,\rho] := \{ \theta \in \bbB_0[k_1] : d_2(\theta,\bbB_0[k_0]) \geq \rho\}\ ,
\eeq
as the set of $k_1$-sparse vectors that lie at distance larger than $\rho$ from the null. Then, for a fixed $\Delta>0$ and $\rho>0$, we consider the testing problem 
\beq\label{eq:hypotheses}
 H_{k_0}: \ \theta\in \bbB_0[k_0]\quad \text{ versus }\quad  H_{\Delta,k_0,\rho}:\ \theta \in \bbB_0[k_0+\Delta,k_0,\rho]\ .
\eeq
The purpose of this definition is to remove from the alternative hypothesis parameters $\theta$ that are too close to the null hypothesis. 
Given a  test $T$, its risk $R(T;k_0,\Delta,\rho)$ for the above problem~\eqref{eq:hypotheses} is defined as the sum of the type I and type II error probabilities 
\beq\label{eq:risk}
R(T;k_0,\Delta,\rho):= \sup_{\theta \in \bbB_0[k_0]}\P_{\theta,\sigma}[T=1] +  \sup_{\theta \in \bbB_0[k_0+\Delta,k_0,\rho]}\P_{\theta,\sigma}[T=0]\ . 
\eeq
Here, $\P_{\theta,\sigma}$ stands for the distribution of $Y$. The function $\rho\mapsto R(T;k_0,\Delta,\rho)$ is non-increasing and equals at least one for $\rho=0$. 
Fixing some $\gamma\in (0,1)$, the separation distance 
$\rho_\gamma(T;k_0,\Delta)$ is the largest $\rho$ such that the hypotheses 
\beq\label{eq:separation}
\rho_{\gamma}(T;k_0,\Delta):= \sup \left\{\rho>0\ |R(T;k_0,\Delta,\rho)>\gamma\right\}\ . 
\eeq
The separation distance of a good test $T$ should be the smallest possible. 
Finally, the minimax separation distance is
\beq\label{eq:separationkvminmax}
\rho^*_{\gamma}[k_0,\Delta]:= \inf_{T}\rho_{\gamma}(T;k_0,\Delta )\ ,
\eeq
where the infimum is taken over all tests $T$. 
In other words, $\rho^*_{\gamma}[k_0,\Delta]$ is the minimal distance to $\bbB_0[k_0]$ such some  test is able to  reliably distinguish parameters in $\bbB_0[k_0]$ from parameters in $\bbB_0[k_0+\Delta,k_0,\rho]$. Hence, it characterizes the difficulty of the testing problem. 
A test $T$ whose separation distance $\rho_{\gamma}(T;k_0,\Delta)$ is (up to a multiplicative constant) smaller than $\rho^*_{\gamma}[k_0,\Delta]$ is said to be minimax.

\subsection{Our contribution}\label{sec:contrib}

Our contribution is threefold:
\begin{enumerate}
\item[(i)] We first establish the minimax separation distance $\rho_{\gamma}^*[k_0,\Delta]$ for all integers $k_0$ and all $\Delta>0$. Besides, we introduce a new test which is minimax adaptive for all $\Delta$.  
\item[(ii)] In the more realistic setting where the noise level $\sigma$ is unknown, the minimax separation distance  $\rho_{\gamma,\mathrm{var}}^*[k_0,\Delta]$ (defined in Section \ref{sec:testUV}) is established and minimax adaptive tests are exhibited. Interestingly, it is proved that the sparsity testing problem under unknown noise level is no more difficult than under known noise level for small $\Delta$. For large $\Delta$, the knowledge of $\sigma$ plays an important role.
\item[(iii)] We reformulate the sparsity estimation problem as a multiple testing problem where we simultaneously consider all nested hypotheses $H_{q}$ for $q\in [0,n]$. Introducing a multiple testing procedure which is simultaneously optimal over all $q$, we derive an estimator $\widehat{k}$ which is smaller or equal to $\|\theta\|_0$ with high probability and is also closest to $\|\theta\|_0$ in a minimax sense. Interestingly, this property will be valid for all possible $\theta\in \mathbb{R}^n$ and avoids us to rely on any particular assumption on the parameter. 
More generally, this perspective also provides a general roadmap to handle the problem of complexity estimation using simultaneous separation distances. 
\end{enumerate}

Before discussing more specifically these three points, let us review the literature. 

\subsection{Related literature}

Although the twin problems of sparsity testing and sparsity estimation are closely connected, we start by discussing the literature mostly related to the test version of our problem and then turn to the estimation version.

\paragraph{Signal detection.} The signal detection problem which amounts to testing whether $\theta=0$ is a special instance of the sparsity testing problem (corresponding to $k_0 = 0$). Signal detection in the Gaussian vector model has been extensively studied ~\cite{ingster2012nonparametric,baraud02,jin2004, collier2015minimax} in the last fifteen years and is now well understood. For instance, it has been established in~\cite{collier2015minimax} that the minimax separation distance $\rho_{\gamma}^{*}[0,\Delta]$ satisfies
$$\rho_{\gamma}^{*2}[0,\Delta] \asymp_{\gamma} \sigma^2 \Delta\log\big(1+\frac{\sqrt{n}}{\Delta}\big)\ ,$$
where $f(\Delta,n) \asymp_{\gamma} g(\Delta,n)$ means that there exist  positive  constants $c_{\gamma}$ and $c'_{\gamma}$ (possibly depending on $\gamma$) such that 
$f(\Delta,n)\leq c_{\gamma}  g(\Delta,n)\leq c'_{\gamma}f(\Delta,n)$. Besides, some tests are able to simultaneously achieve the above separation distances for all positive $\Delta$.

Looking more closely at the above equation, one can distinguish two main regimes for this problem depending on the sparsity $\Delta$ of the alternative: the sparse case ($\Delta\leq \sqrt{n}$) and the dense case ($\Delta>\sqrt{n}$). In the sparse case, $\rho_{\gamma}^{*2}[0,\Delta]$ is of order $\Delta \log(1+ n/\Delta^2)$. This entails that it is possible to detect sparse vectors $\theta$ whose non-zero values are of order $\sqrt{\log(n/\Delta^2)}$. Known optimal tests such as the higher criticism test~\cite{jin2004} or the one proposed in~\cite{collier2015minimax} amount to counting the number of values $|Y_i|$ that are larger than $t$ and to compare this number to what is expected under the null hypothesis. Doing this simultaneously for a wide range of $t$ leads to near-optimal performances simultaneously for all $\Delta\in [1,\sqrt{n}]$. In the dense case ($\Delta\geq \sqrt{n}$), the situation is qualitatively different as the square minimax separation distance $\rho_{\gamma}^{*2}[0,\Delta]$ is of order $\sqrt{n}$. A near-optimal test, proposed in e.g.\ ~\cite{baraud02}, is based on the statistic $\|Y\|_2^2/\sigma^2$, which, under the null, follows a $\chi^2$ distribution with $n$ degrees of freedom and, under the alternative, follows a non-central $\chi^2$ distribution with non-centrality parameter $\|\theta\|_2^2$.

\paragraph{Composite-composite testing problems and functional estimation.} For the signal detection problem ($k_0=0$), the null hypothesis is simple whereas for the general case $k_0>0$, the null hypothesis is composite, thereby making the analysis of the problem more challenging. Although we are not aware of any general treatment of this kind of problem in the literature (and we are also not aware of the treatment of our specific problem in the literature), some partial results and methods may be derived in our setting from prior approach on related problems. 

Minimax analysis of composite-composite testing problems has, up to our knowledge, been tackled in a few work~\cite{Juditsky_convexity,baraud2005testing,comminges2013minimax,carpentier2015testing}. Some functional estimation problems, whose goal is to infer $f(\theta)$ for a given function $f$, are also related to some composite-composite testing problems. In fact, some work on functional problems~\cite{lepski1999estimation, gayraud2005adaptive,cailow2011,MR2382653,MR2589318} and adaptive confidence regions (e.g.\ ~\cite{cai2004adaptation,cai2006adaptive, MR2906872,nickl_vandegeer}) have lead to progress in the understanding of such  testing problems.

To be more specific on the challenge of composite-composite problems, let us describe a natural  approach  called ''infimum testing" \cite{gine2015mathematical}. Consider a signal detection test based on the statistic $S(.)$, that rejects the null hypothesis $H_{k_0}$ with $k_0=0$ for large values of $S(Y)$. The corresponding infimum test for the composite-composite testing problem~\eqref{eq:hypotheses}, is a test rejecting $H_{k_0}$ for large values of $S_{\inf}:=\inf_{u\in \bbB_0[k_0]}S(Y-u)$. Indeed, there exists, under the null hypothesis $H_{k_0}$, some $u$ such that the expectation of $Y-u$ is zero. As one may expect, considering this infimum over all possible parameters in the null hypothesis is not priceless and the separation distance $\rho_{\gamma}[T_{k_0};k_0,\Delta]$ of the corresponding infimum test $T_{k_0}$ may depend on the complexity of the null hypothesis. 
Conversely, simple inclusion arguments that will be recalled in our proofs entail that the composite problem is at least as difficult as the signal detection problem, that is $\rho^{*}_{\gamma}[k_0,\Delta]$ is at least of the  order of $\rho^{*}_{\gamma}[0,\Delta]$. The main challenge is therefore to decipher whether $\rho^{*}_{\gamma}[k_0,\Delta]$ is indeed of order $\rho^{*}_{\gamma}[0,\Delta]$ or if it is larger than that and really depends on $k_0$. In other words, we seek to understand how the complexity of the null hypothesis influences the difficulty of the testing problem.

\paragraph{Sparsity estimation.} Closer to our setting, Cai, Jin and Low~\cite{MR2382653} study the problem of estimating $\|\theta\|_0$ for sparse vectors $\theta$ such that $\|\theta\|_0\leq \sqrt{n}$. They consider a Bayesian framework, where each component $\theta_i$ is drawn independently from a two points mixture distribution $(1-\eta)\delta_0+ \eta\delta_a$ for some unknown $a>0$ ($\delta_x$ denotes the Dirac measure at $x$). The goal is then to estimate $\eta= \E[\|\theta\|_0]/n$. Relying on the tail distribution of $Y$, they introduce an estimator $\widehat{\eta}$ that satisfies $\widehat{\eta}\leq \eta$ with high probability and such that the risk $\E[|1- \widehat{\eta}/\eta|]$ is as small as possible. In~\cite{MR2420411}, Jin introduced a class of estimators of $\theta$ based on the empirical characteristic function of $Y$ to handle the denser case $\|\theta\|_0\geq \sqrt{n}$. Later, these procedures have been extended~\cite{MR2325113,MR2589318} to allow for unknown noise level $\sigma$ and even unknown mean in the more general model $Y_i=u+\theta_i+\epsilon_i$, where $u$ is unknown. Again, in a Bayesian framework where all $\theta_i$'s follow the same mixture distribution $(1-\eta)\delta_0+ \eta \pi$ for some smooth density $\pi$, their estimator $\widehat{\eta}$ is proved to achieve an optimal minimax rate. 

In multiple testing, estimating the number of false hypotheses has a longer history. Rephrased in the Gaussian vector model, multiple hypotheses testing amounts to test simultaneously whether each $\theta_i$ is zero or not. Hence, estimating the number of false hypotheses is equivalent to sparsity estimation. Nevertheless, most work on this field (e.g.\  \cite{patra2015estimation,meinshausen2006estimating,storey2002direct,celisse2010cross,langaas2005estimating}) consider a more general setting where each $Y_i$ follows a mixture of a normal distribution and some unknown distribution that stochastically dominates the normal distribution. Hence, the methods and results are not directly comparable to ours.

\subsection{Further description of our results}
We now discuss in more details our three contributions mentionned in Section \ref{sec:contrib}. 

\paragraph{Sparsity testing for known $\sigma$.} Table~\ref{Tab:KV} summarizes the squared minimax separation distances $\rho^{*2}_\gamma[k_0,\Delta]$. Interestingly, for $k_0\leq \sqrt{n}$, the minimax separation distance is the same as for signal detection ($k_0=0$). In contrast, for more complex null hypotheses ($k_0\geq \sqrt{n}$), the complexity of the null hypothesis comes into play. For instance, when $\Delta\geq k_0 \geq \sqrt{n}$, then  $\rho^{*2}_\gamma [k_0,\Delta]$ is of order $k_0/[\log(1+\tfrac{k_0}{\sqrt{n}})]$. This is smaller by a polylog multiplicative term than what can be obtained by infimum tests and we have to rely on really different statistics. In fact, our minimax adaptive procedure is a combination of three tests. The first one is an adaptation of the the higher criticism test introduced in~\cite{jin2004}. The second one relies on the empirical characteristic function of $Y$ and borrows ideas from~\cite{MR2420411}. The third statistic is novel and relies on deconvolution ideas. As for the lower bounds of the minimax separation distances for large $k_0$, the proof ideas are more involved than for signal detection~\cite{baraud02} and make use of the moment matching techniques introduced in~\cite{lepski1999estimation} and later refined in~\cite{Juditsky_convexity,cailow2011}.

 \begin{table}
\caption{Square minimax separation distances (in the $\asymp_{\gamma}$ sense) when the noise level $\sigma$ is known for all $k_0\in [0,n-1]$ and $\Delta\in [1,n-k_0]$.}
\label{Tab:KV}
 \begin{center}
 \begin{tabular}{c|c|c}
 $k_0$& $\Delta$ &  $\rho^{*2}_\gamma[k_0,\Delta]/\sigma^2$\\ \hline \hline
$k_0 \leq \sqrt{n}$ & $1\leq \Delta \leq n-k_0$ & $\Delta \log\Big(1+\frac{\sqrt{n}}{\Delta}\Big) $ \\ 
& $ \sqrt{n} \leq \Delta \leq n-k_0$  &$\sqrt{n}$ \\ 
\hline
$k_0 > \sqrt{n}$ &  $1 \leq \Delta \leq \sqrt{n^{1/2} k_0}$
&  $\Delta \log\Big(1+\frac{k_0}{\Delta}\Big)$\\
 & 
$\sqrt{n^{1/2} k_0}\leq \Delta \leq k_0$ & $ \Delta \frac{\log^2\big(1+\tfrac{k_0}{\Delta}\big)}{\log\big(1+\tfrac{k_0}{\sqrt{n}}\big)}$\\
 &  $k_0 \leq \Delta \leq n-k_0$ & $ \frac{k_0}{\log\big(1+\tfrac{k_0}{\sqrt{n}}\big)}$\\ 
 \end{tabular}
 \end{center}
\end{table}

\paragraph{Sparsity testing for unknown $\sigma$.} The results discussed above hold under the restrictive assumption that the noise level $\sigma$ is known. For unknown $\sigma$, the situation is qualitatively different (see Table \ref{fig:UV}). As a first step, we study the signal detection problem ($k_0=0$) for which only partial results had been established. For sparse alternatives ($\Delta\leq \sqrt{n}$), one can plug an estimator of $\sigma$ in the signal detection statistic so that the minimax separation distance $\rho^*_{\gamma,\mathrm{var}}(0,\Delta)$ for unknown variance (defined in~\eqref{eq:separationuvminmax}) is the same as $\rho^*_{\gamma}(0,\Delta)$. However, for $\Delta$ larger than $\sqrt{n}$ and much smaller than $n$, one cannot simply plug a variance estimator and new test statistics are required. The squared separation distance $\rho^{*2}_{\gamma,\mathrm{var}}(0,\Delta)$ is of order $\sqrt{\Delta n^{1/2}}$ whereas $\rho^{*2}_{\gamma}(0,\Delta)$ is only of order $\sqrt{n}$. In the really dense case where $\Delta$ is proportional to $n$, we establish that the separation distance $\rho^{*2}_{\gamma,\mathrm{var}}(0,\Delta)$ is even larger. Turning to the general case $k_0>0$, we establish that $\rho^*_{\gamma,\mathrm{var}}(k_0,\Delta)$ is larger than its counterpart for known $\sigma$ for all $\Delta\geq \sqrt{n}\vee k_0$. In comparison to the known variance case, one cannot simply accommodate the adaptive test by estimating the noise level. In fact, the minimax adaptive test in this new setting is based on quite different statistics.

\begin{table}
\caption{Square minimax separation distance $\rho_{\gamma,\mathrm{var}}^{*2}[k_0,\Delta]$ (as defined in Equation~\eqref{eq:separationuvminmax})  when the noise level $\sigma$ is unknown but belongs to some known fixed interval  $[\sigma_-,\sigma_+]$. Here, $c\in (0,1)$ is some fixed universal constant and $\xi\in (0,1)$ can be chosen arbitrarily small.}\label{fig:UV}
 \begin{center} 
 \begin{tabular}{c|c|c}
$k_0$ & $\Delta$ &  $\rho_{\gamma, \mathrm{var}}^{*2}[k_0,\Delta]/\sigma_+^2$\\\hline \hline
$0\leq k_0 \leq \sqrt{n}$ & $0\leq \Delta \leq \sqrt{n}$ & $\Delta \log\Big(1+\frac{\sqrt{n}}{\Delta}\Big)$  \\ 
 & $\sqrt{n}< \Delta \leq cn$ & $\sqrt{\Delta n^{1/2}}$ \\
\hline
$n^{1-\xi} \geq k_0 \geq \sqrt{n}$ & $0\leq \Delta \leq \sqrt{k_0n^{1/2}}$ &  $\Delta \log\Big(1+\frac{k_0}{\Delta}\Big)$\\
 & $\sqrt{k_0n^{1/2}}< \Delta \leq k_0$  & $ \Delta \tfrac{\log^2\big(1+\frac{k_0}{\Delta}\big)}{\log\big(1+\frac{k_0}{\sqrt{n}}\big)}$\\
&  $ k_0 < \Delta \leq cn$ &  $\frac{\sqrt{\Delta k_0}}{\log\Big(1+\frac{k_0}{\sqrt{n}}\Big)}$  
 \end{tabular}
 \end{center}
\end{table}

\paragraph{Sparsity estimation.} 
Let us first verbalize the desirable properties of a good estimator of $\|\theta\|_0$. The functional $\|\theta\|_0$ is not continuous with respect to $\theta$. 
Consider a one-sparse vector $\theta$ (with one large non-zero component) and a perturbation $\theta'$ of $\theta$ whose components are all nonzero but are arbitrarily small. As the distribution $\P_{\theta,\sigma}$ is close to $\P_{\theta',\sigma}$, the estimator $\widehat{k}$ will follow almost the same distribution for both parameters. It is obviously preferable for  $\widehat{k}$ to be concentrated around one under $\P_{\theta',\sigma}$ than  around $n$ under $\P_{\theta,\sigma}$. In other words,  a good estimator $\widehat{k}$ should have  a small overestimation probability. Besides, a good estimator $\widehat{k}$ should be larger than any fixed $q$, as soon as the distance of $\theta$ to the collection $\mathbb B_0[q]$ 
is large enough.

To formalize the above intuition, let us consider the multiple testing problems with all hypotheses $(H_{q})$, for $q=0,\ldots, n$ where $H_{q}$ is defined in \eqref{eq:hypotheses}. Then, the set of true hypotheses is exactly  $\{H_{q},\  q\geq \|\theta\|_0\}$. Similarly, an estimator $\widehat{k}$ of $\|\theta\|_0$ can be interpreted as a multiple testing procedure rejecting all hypotheses $H_{q}$ with $q< \widehat{k}$ and accepting all hypotheses $H_q$ with $q\geq \widehat{k}$. Conversely, one can build an estimator of $\|\theta\|_0$ from any multiple testing procedure. Building on this correspondence between  complexity tests and complexity estimation, we first construct a multiple sparsity testing procedures. Although the minimax optimality of multiple testing procedures is difficult to assess (but see~\cite{2016_fromont}), we are able to prove that our procedure is simultaneously minimax for all single hypotheses $H_{q}$. 
 Then, the corresponding estimator $\widehat{k}$ satisfies, with high probability, the three following properties
\begin{enumerate}
 \item[(a)] $\widehat{k}\leq \|\theta\|_0$, which is equivalent to $\theta_{(\widehat{k})}\neq 0$ (Here $\theta_{(i)}$ stands  for the $i$-th largest entry of $\theta$ in absolute value\footnote{Consequently, we have $|\theta_{(1)}|\geq|\theta_{(2)}|\geq \ldots \geq |\theta_{(n)}|$.} with the convention $\theta_{(0)}=+\infty$).
 \item[(b)] For all $q=1,\ldots, n-\widehat{k}$,  $|\theta_{(\widehat{k}+q)}|\leq c\psi_{\hat{k},q}$, 
where $c$ is a numerical constant and the function $\psi_{\widehat{k},q}$ is defined in ~\eqref{eq:upper_distance_linfty}. In other words, we can certify, that even if 
 $\widehat{k}$ is possibly smaller than $\|\theta\|_0$, each of its remaining $(\|\theta\|_0-\widehat{k})$ non-zero components are small enough. 
 \item [(c)]  $d_2(\theta,\bbB_0[\widehat{k}])\leq c' \rho_{\gamma}^{*}[\widehat{k}, \|\theta\|_0-\widehat{k}]$, where $c'$ is a numerical constant and $\gamma$ is fixed. In other words, $\theta$ is close in $l_2$ distance to the collection of $\widehat{k}$-sparse vectors.
\end{enumerate}
Note that both properties (a) and (b) produce data-driven certificates for all $\theta_{(\widehat{k}+q)}$,  $q\geq 0$ in the sense that corresponding bounds are explicit. 
Besides, the three above properties are valid  for {\it all} $\theta\in \mathbb{R}^n$, whereas previous work~\cite{MR2382653,MR2325113,MR2589318} only considered specific classes $\theta$ by assuming for instance that the $\theta_i$'s are sampled according to a mixture of a Dirac at $0$ and a smooth distribution. For a given $\theta$, one can invert the inequalities in conditions (b) and (c) to obtain a bound for $|\widehat{k}-\|\theta\|_0|$. Finally, both conditions (b) and (c) are optimal from a minimax perspective defined in Section \ref{sec:estimation}.

\subsection{Notation and organization of the paper}

Although some of the notation have already been introduced, we gather them here to ease the reading. 
Given a vector $u\in \mathbb{R}^n$ and $p\geq 1$, we denote $\|u\|_p^p := (\sum_i |u_i|^p)^{1/p}$ its $l_p$ norm. Also, $\|u\|_{\infty} := \max_i |u_i|$  stands for its $l_{\infty}$ norm and $\|u\|_0= \sum_{i}\1_{u_i\neq 0}$ its $l_0$ function. 
In the sequel, $\phi(.)$ stands for the density of a standard normal variable, and $\Phi(.)$ for its survival function. Also $\cN(x,\sigma^2)$ stands for the normal distribution with mean $x$ and variance $\sigma^2$. 
Given $x\in \mathbb{R}$,  we write as usual $\lfloor x\rfloor$ for the integer part of $x$ and $\lceil x\rceil$ for the rounding to the upper integer, and $(x)_+:= \max(x,0)$. Also $[n]$ is short for the set $\{1,\ldots,n\}$. For any $i\in [n]$, $\theta_{(i)}$ stands for the $i$-th largest entry of $\theta$ in absolute value. In other words, one has $|\theta_{(1)}|\geq|\theta_{(2)}|\geq \ldots \geq |\theta_{(n)}|$. 

In the sequel, $c$, $c_1$, $\ldots$ denote positive universal constants that may change from line to line. We also denote $c_{\alpha}$, $c'_{\beta}$,\ldots, denote positive constants whose values may depend on $\alpha$ or $\beta$. 

When $Y$ is distributed according to the model \eqref{eq:model}, we write $\P_{\theta,\sigma}$ for the distribution of $Y$. As $\sigma$ is fixed and supposed to be known in Sections \ref{sec:testKV} and \ref{sec:estimation}, we drop the dependency on $\sigma$ in these two sections and simply write $\P_{\theta}$.

\medskip 

In Section~\ref{sec:testKV}, we describe our model testing results when the variance of the noise is known, presenting both upper and lower bounds. In Section~\ref{sec:estimation}, we detail how these testing results can be applied to the relevant problem of sparsity estimation. Section~\ref{sec:testUV} is devoted to the unknown variance case.  Finally, remaining results and all the proofs are postponed the Appendix.

\section{Sparsity testing with known variance}\label{sec:testKV}

\subsection{Minimax lower bound}\label{sec:lbkv}

In this section, we consider the  the sparsity testing problem \eqref{eq:hypotheses}  in a setting when the noise variance $\sigma^2$ is known. 
The following theorem states a lower bound on the minimax separation distance $\rho^*_{\gamma}[k_0,\Delta]$.

\begin{thm}\label{prp:lower}
There exists a numerical constant $c>0$ such that the following holds.
Consider any $\gamma\leq 0.5$.  For any  $k_0 \leq \sqrt{n}$ and $\Delta\leq n-k_0$, we have
 \beq\label{eq:lower_bound_main}
 \rho_{\gamma}^{*2}[k_0,\Delta] \geq  \sigma^2 \Delta \log\Big[1+ \frac{ \sqrt{n}}{8\Delta} \Big]\ .
 \eeq
 For any $k_0>\sqrt{n}$, we have
 \beq\label{eq:lower_bound_main2}
 \rho_{\gamma}^{*2}[k_0,\Delta] \geq c\sigma^2 \left\{ \begin{array}{cc}                 
 \Delta \left[\frac{\log^2\big[1+  \frac{k_0}{\Delta}\big]}{\log\big[1+ \frac{k_0}{\sqrt{n}}\big]}\wedge \log\big[1+ \frac{k_0}{\Delta}\big] \right] & \text{ if $\Delta\leq k_0\wedge  (n-k_0)$}\\
 \frac{k_0}{\log\big[1+ \frac{k_0}{\sqrt{n}}\big]} & \text{ if } k_0< \Delta \leq n-k_0
 \end{array} \right.
 \eeq

\end{thm}
As proved in the next subsection, this lower bound turns out to be sharp. We shall precisely discuss these quantities later. Before this, we only give a glimpse of the different regimes unveiled by the above theorem.

Whenever  $k_0 \leq \sqrt{n}$,  
the lower bound on the minimax separation distance is the same as the signal detection minimax separation distance $\rho^*_{\gamma}[0,\Delta]$, see~\cite{baraud02,collier2015minimax}. In this regime, the size $k_0$ of the null hypothesis does not play a role in the separation distance. In fact, the proof of \eqref{eq:lower_bound_main} is a consequence of known results for the signal detection problem. More precisely, we follow Le Cam's method and choose  a particular $\theta_0\in \bbB_0[k_0]$ and  a  prior distribution $\nu$ on the collection $\bbB_0[k_0+\Delta,k_0,\rho]$. Let us write $\bbQ_1:= \int \mathbb{P}_{\theta}\nu(d\theta)$ the marginal distribution of $Y$ when $\theta$ is sampled according to $\nu$. Then, the risk $R(T;k_0,\Delta,\rho)$ \eqref{eq:risk} of any test $T$ is larger than $1-\|\mathbb{P}_{\theta_0}-\bbQ_1\|_{TV}$ ($\|.\|_{TV}$ is the total variation distance). Since the total variation distance is dominated by the $\chi^2$ distance between probability distributions, it suffices to bound from above this $\chi^2$ distance.

For  $k_0$ much larger than  $\sqrt{n}$ and for $\Delta\geq k_0$, the lower bound \eqref{eq:lower_bound_main2} is of order $k_0/\log\big[ \tfrac{k_0}{\sqrt{n}}\big]$  - which is significantly larger than the signal detection rate $\rho^*_{\gamma}[0,\Delta]$. In this regime, the complexity of the null hypothesis $H_{k_0}$ has to be taken into account to obtain the right lower bound. Following an approach pioneered in \cite{lepski1999estimation}, we build two product prior distributions $\mu_0^{\otimes n}$ and $\mu_1^{\otimes n}$ (almost) supported by $\bbB_0[k_0]$ and $\bbB_0[k_0+\Delta,k_0,\rho]$ in such a way that the first moments of $\mu_0$ and $\mu_1$ are matching. Writing $\bbQ_{0}:= \int \mathbb{P}_{\theta}\mu_0^{\otimes n}(d\theta)$ and $\bbQ_{1}:= \int \mathbb{P}_{\theta}\mu_1^{\otimes n}(d\theta)$, we need to upper bound the $\chi^2$ distance between $\bbQ_0$ and $\bbQ_1$. It turns out that matching the moments of $\mu_0$ and $\mu_1$ enforces the $\chi^2$ distribution between $\bbQ_0$ and $\bbQ_1$ to be small enough. The main technical hurdle in the proof is the construction of the two measures $\mu_0$ and $\mu_1$ that maximize the the number of matching moments, while being supported respectively on the null and alternative hypothesis with $\rho$ as large as possible.

\subsection{Minimax upper bound}\label{sec:ubkv}

In this subsection, we construct three tests that are most effective in three different situations: the Higher Criticism regime (large but few non-zero components), the Bulk regime (many but small non-zero components) and the Intermediary regime. Then, a combination of these three procedures is proved to achieve the minimax lower bounds of Theorem~\ref{prp:lower} and is even adaptive to the sparsity $k_1$. 
Throughout this subsection, we consider some fixed $\alpha$ and $\beta$ in $(0,1)$. 

\subsubsection{Higher Criticism Statistic}\label{sec:HCKV}

Let us adapt the Higher Criticism statistic introduced in \cite{jin2004} for signal detection. Recall that, for $t>0$, $\Phi(t)$ is the survival function of the standard normal distribution For any $t>0$, define 
\beq\label{eq:definition_Nt}
N_t := \#\{i\ ,\, |Y_i|\geq t\}\ ,
\eeq
the number of components larger (in absolute value) than $t$, 
  $t^{HC}_{*,\alpha}:= \lceil \sqrt{2\log[4n/\alpha]}\rceil $ and the collection $\cT_{\alpha}:=[ t^{HC}_{*,\alpha}]$. Then, the test $T^{HC}_{\alpha,k_0}$ rejects the null hypothesis $H_{k_0}$, if either $N_{\sigma t^{HC}_{*,\alpha}}\geq k_0+1$ or for some $t\in \cT_{\alpha}$, 
\beq \label{eq:definition_rejection_HC}
N_{\sigma t}\geq  k_0 +2(n-k_0)\Phi(t)+ u^{HC}_{t,\alpha}\ ,
\eeq
where 
\beq\label{eq:rejection_Nt}
   u^{HC}_{t,\alpha} := 2\sqrt{n\Phi(t)\log\left(\frac{t^2\pi^2}{3\alpha}\right)} + \frac{2}{3}\log\left(\frac{t^2\pi^2}{3\alpha}\right)\ . 
\eeq
Under the the null hypothesis $H_{k_0}$, $\theta$ contains at most $k_0$ non zero coefficients and $N_{\sigma t}-k_0$ is therefore stochastically dominated by a Binomial random variable with parameters $(n-k_0, 2\Phi(t))$. It then follows from Chebychev inequality that $N_{\sigma t}\leq k_0 + 2(n-k_0)\Phi(t)+ O_p(\sqrt{(n-k_0)\Phi(t)})$. 
The specific choice of the tuning parameter $u^{HC}_{t,\alpha}$ allows to handle the multiplicity of the tests. In the specific case $k_0=0$ (signal detection), $T^{HC}_{\alpha,k_0}$
is analogous to the vanilla Higher Criticism test~\cite{jin2004}.

\begin{prp}\label{cor:T_HC} 
The size of the test $T^{HC}_{\alpha,k_0}$ is smaller of equal to $\alpha$. Besides, any $\theta\in \mathbb{R}^n$ such that 
 \beq\label{eq:separation_simple_HC}
 |\theta_{(k_0+q)}| \geq c_{\alpha,\beta}\sigma \sqrt{\log\Big(2+ \frac{\sqrt{n}\vee k_0}{q}\Big)}\ , \quad \text{ for some  }q\in [1,n-k_0]
 \eeq
 belongs to the high probability rejection region of $T^{HC}_{\alpha,k_0}$, that is $\P_{\theta}[T^{HC}_{\alpha,k_0}=1]\geq 1-\beta$.
 \end{prp}
In the specific case $k_0=0$, we recover the known behavior or the Higher Criticism statistic in the signal detection setting. The test $T^{HC}_{\alpha,k_0}$ is powerful when, for a given integer $q$, there are least $(k_0+q)$ coefficients  larger than some threshold depending on $q$. For $q=1$, the threshold is of order $\sigma \sqrt{\log(n)}$, whereas for $q\geq \sqrt{n}\vee k_0$, the threshold is of order one. It will turn out that $T^{HC}_{\alpha,k_0}$ achieves the optimal separation $\rho^{*}_{\alpha+\beta}[k_0,\Delta]$ when $\Delta\leq \sqrt{n^{1/2}k_0 \vee n}$. However, the test $T^{HC}_{\alpha,k_0}$ does not manage to detect vectors $\theta$ containing many coefficients that are small in front of one. This is why we follow another approach in this regime.

\subsubsection{Detecting the signal in the bulk distribution}\label{sec:BulkKV}

When there are many small coefficients, we rely on the empirical characteristic functions of $Y$ following an approach introduced in \cite{MR2420411}. 
Given $s>0$, define the function
\beq \label{eq:defintion_kappa}
\kappa_{s}(x):= \int_{-1}^1 (1-|\xi|) e^{s^2\xi^2/2}\cos(s \xi x)d\xi\ ,
\eeq
and the test statistic $Z(s)$ 
\beq\label{eq:definition_Z}
Z(s) := \sum_{i=1}^n \big(1-  \kappa_{s}(Y_i/\sigma)\big)\ .
\eeq
Let us describe the intuition behind this statistic using a population approach. 
Denoting $\overline{\varphi}_n(s)$ the empirical characteristic function and $\overline{\varphi}(s)$ its expectation
\beq  \label{eq:definition_empirical_fourier}
\overline{\varphi}_n(s):= n^{-1}\sum_{i=1}^n \cos(s Y_i),\quad \quad \overline{\varphi}(s) := n^{-1}\sum_{i \leq n} \cos(s \theta_i) e^{-\frac{s^2\sigma^2}{2}}\ ,
\eeq
one can derive the expectation of $Z(s)$
\[\E_{\theta}[Z(s)]= \sum_{i=1}^n 1- \int_{-1}^1 (1-|\xi|)  \cos(s\xi\theta_i/\sigma)d\xi= \sum_{i=1}^n 1- 2\frac{1 - \cos(s\theta_i/\sigma)}{(s\theta_i/\sigma)^2} \ , \]
with the convention $(1-\cos(x))/x^2=1/2$ for $x=0$. Since, for all $x$, $\cos(x)\in [1-x^2/2,1]$, one may easily show (see the proof of Proposition \ref{cor:TB_power} for details) that $\E_{\theta}[Z(s)]\leq \|\theta\|_0$. Under the null, this expectation is therefore smaller or equal to $k_0$. Besides, a Taylor development of the $\cos$ function around $0$ ensures that $1- 2\frac{1 - \cos(sx)}{(sx)^2}= \frac{1}{12}(sx)^2+ o(s^2x^2)$. If, under the alternative, there are so many small coefficients $|\theta_i|$ that the corresponding sum 
$\sum_{i}\theta_i^2 s^2/\sigma^2$  is large in front of $k_0$, then, at least in expectation,  $Z(s)$ is larger than under the null. 

\medskip

\noindent
{\bf Remark}: 
Rewriting the statistic $Z(s)/n=1- \int_{-1}^1(1-|\xi|)   e^{s^2\xi^2/2}\overline{\varphi}_n(s\xi/\sigma)d\xi $, one observes that the empirical characteristic function is multiplied by the function $(1-|\xi|)$ before integration. In   \cite{MR2420411}, Jin also suggests other statistics such as $\int_{-1}^1   e^{s^2\xi^2/2}\overline{\varphi}_n(s\xi/\sigma)d\xi$ or the deconvolution estimator $e^{s^2/2}\overline{\varphi}_n(s/\sigma)$. However, these two statistics turn out to be suboptimal in our setting.

\medskip

In practice, we set $s_{k_0}:= \sqrt{\log(ek^2_0/n)}\vee 1$ and we define the test $T^{B}_{\alpha,k_0}$ rejecting the null hypothesis when
\beq\label{eq:definition_Tb_alpha}
Z(s_{k_0})\geq k_0 + u^{B}_{k_0,\alpha}\ , \quad \text{where}\quad  u^{B}_{k_0,\alpha}:= \frac{e^{s_{k_0}^2/2}}{s_{k_0}} \sqrt{8n\log(2/\alpha)}\ .
\eeq

 \begin{prp}\label{cor:TB_power}
 There exist three positive constants $c_{\alpha,\beta},c'_{\alpha,\beta}, c^{''}_{\alpha,\beta}$ such that the following holds. 
  The type I error  probability of $T^{B}_{\alpha,k_0}$ is smaller or equal to $\alpha$. Besides, any $\theta\in \mathbb{R}^n$ satisfying any of the two following conditions
\begin{eqnarray}
 |\theta_{(k_0+q)}| &\geq&  c_{\alpha,\beta}\sigma \sqrt{\frac{k_0}{q\log(1+k_0/\sqrt{n})}}\ ,\, \text{ for some } q\geq  \frac{c^{'}_{\alpha,\beta}k_0}{\sqrt{\log(1+ \frac{k^2_0}{n}})}
\ , \label{eq:separation_log}
\\
 \label{eq:separation_log_l2}
 \sum_{i=1}^n \Big[\theta_i^2\wedge s^{-2}_{k_0}\Big]&\geq& c^{''}_{\alpha,\beta}\sigma^2 \frac{k_0}{\log(1+ k_0/\sqrt{n})}\ ,
\end{eqnarray}
belongs to the high probability rejection region of $T^{B}_{\alpha,k_0}$, that is $\P_{\theta}[T^{B}_{\alpha,k_0}=1]\geq 1-\beta$.

 \end{prp}

 The above proposition provides two sufficient condition for  $T^{B}_{\alpha,k_0}$ to be powerful. The second condition \eqref{eq:separation_log_l2} formalizes the above discussion for the population version of the statistic: when the squared $l_2$ norm of the restriction of $\theta$ to its small coefficients is larger in front of $\sigma \frac{k_0}{\log(1+ k_0/\sqrt{n})}$, then the test is powerful. Condition~\eqref{eq:separation_log} ensures that the test is also powerful when there are more than $k_0+q$ coefficients larger than some threshold depending on $q$. In comparison to the Higher Cristicism test, Condition~\eqref{eq:separation_log} is effective for large $q$ (many non-zero coefficients), but these coefficients can be much smaller than one.

\subsubsection{Intermediary regimes}\label{sec:InterKV}

A combination of the two previous tests covers the extreme regimes for the sparsity testing problem: a few large coefficients (Higher Criticism) and many small coefficients (Bulk). Unfortunately, they turn out to be suboptimal in intermediate regimes ie.\ ~for any parameters in between. This is why we have to devise a third test. In this subsection we aim at discovering intermediary signals whose signature is neither in the bulk of the empirical distribution of $(Y_{i})$ nor in its extreme values. This problem will only reveal to be relevant for large $k_0$ and  we assume henceforth that $k_0\geq 20\sqrt{n}$.

Given two tuning parameters $r$ and $l$, define the function 
\beq\label{eq:def_eta}
\eta_{r,w}(x):= \frac{r}{(1-2\Phi(r))}\int_{-1}^1 \frac{e^{-r^2\xi^2/2}}{\sqrt{2\pi}} e^{\xi^2 w^2/2 }\cos(\xi  w x)d\xi\ .
\eeq
and the statistic
\[V(r,w):= \sum_{i=1}^n 1 - \eta_{r,w}(Y_i/\sigma)\ .\]
In order to get a grasp this statistic let us consider the expectation of $\eta_{r,w}(X)$ for $X\sim \cN(x,1)$. Simple computations (see \eqref{eq:definition_psiq} in the proof of Proposition \ref{cor:power_TI}) lead to 
\[
  \E[\eta_{r,w} (X)]= \frac{1}{1-2\Phi(r)}\int_{-r}^{r} \phi(\xi)\cos(\xi x \frac{w}{r} )d\xi\ , 
\]
which, for large $r$, is of order $\int_{\bbR} \phi(\xi) \cos(\xi x \frac{w}{r})d\xi= \exp(-x^2\tfrac{w^2}{2r^2})$. As a consequence, $\E_{\theta}[V(r,w)]$ approximates the function $\|\theta\|_0$ at an exponential rate. 
In contrast, the population version of $Z(s)$ \eqref{eq:definition_Z} only approximates the function $\|\theta\|_0$ at a quadratic rate. Unfortunately, the variance $V(r,w)$ is quite large which prevents us to take 
 $w/r$ as large as $s_{k_0}$ as in the previous test.

The test $T^{I}_{\alpha,k_0}$ is an aggregation of multiple tests based on the statistics $V(r,w)$ for different tuning parameters $r$ and $w$. Define
 $l_{k_0}:= \lceil (k_0\sqrt{n})^{1/2}\rceil$ and the dyadic collection $\cL_{k_0}=\{l_{k_0}, 2l_{k_0},4l_{k_0}, \ldots, l_{\max}\}$ where $l_{\max}:=  2^{\lfloor \log_2 (k_0/l_{k_0})\rfloor} l_{k_0}/4 \leq k_0/4$. Note that $\cL_{k_0}$ is not empty if $k_0\geq 20 \sqrt{n}$ and $n$ is large enough.
Given any $l\in \cL_{k_0}$, define 
\beq \label{eq:param}
 r_{k_0,l} :=  \sqrt{2\log(\tfrac{k_0}{l})} \ , \quad \quad w_l := \sqrt{\log(\tfrac{l}{\sqrt{n}})}\ .
\eeq
 Then, the test $T^{I}_{\alpha,k_0}$ rejects the null hypothesis if, for some $l \in \cL_{k_0}$, 
\beq\label{eq:rejection_intermediary}
  V(r_{k_0,l},w_l) \geq k_0+ l+ u^{I}_{k_0,l,\alpha} \, \quad \quad\text{ where } \quad u^{I}_{k_0,l,\alpha}:=\sqrt{2 ln^{1/2}\log\Big(\frac{\pi^2 [1+\log_2(l/l_{k_0})]^2}{6\alpha}\Big)}\ ,
\eeq
where $\log_2$ is the binary logarithm.

 \begin{prp}\label{cor:power_TI}
 There exists four positive constants $c,c_{\alpha,\beta},c^{'}_{\alpha,\beta}, c^{''}_{\alpha,\beta}$ such that the following holds. Assume that $k_0\geq 20\sqrt{n}$ and $n\geq c$.
  The type I error  probability of  $T^{I}_{\alpha,k_0}$  is smaller of equal to $\alpha$.  If $k_0\geq c_{\alpha,\beta}\sqrt{n}$,  any $\theta\in \mathbb{R}^n$ satisfying
 \[
 |\theta_{(k_0+q)}| \geq  c^{'}_{\alpha,\beta}\sigma \frac{1+ \log(1+\frac{k_0}{q})}{\sqrt{\log(1+\frac{k_0}{\sqrt{n}})}}\ ,\quad \ \text{ for some }q\geq c^{''}_{\alpha,\beta} \sqrt{k_0n^{1/2}}\ ,
 \]
belongs to the high probability rejection region of $T^{I}_{\alpha,k_0}$, that is $\P_{\theta}[T^{I}_{\alpha,k_0}=1]\geq 1-\beta$.
 \end{prp}

 \subsection{Combination of the tests}

For any integer $q\in [n-k_0]$, define $\psi_{k_0,q}>0$ by 
\beq \label{eq:upper_distance_linfty}
\psi^2_{k_0,q}:= \left\{\begin{array}{cc}
\log\Big[1+ \frac{\sqrt{n}}{q}\Big] & \text{if }k_0\leq \sqrt{n}\ ,\\
\frac{\log^2\big(1+  \tfrac{k_0}{q}\big)}{\log\big(1+\frac{ k_0}{\sqrt{n}}\big)} \bigwedge \log\big(1+ \frac{k_0}{q}\big) & \text{if }k_0 > \sqrt{n} \text{ and }q\leq k_0\ ,\\
\frac{k_0}{q\log\big(1+\frac{ k_0}{\sqrt{n}}\big)} & \text{if }k_0 > \sqrt{n} \text{ and }q> k_0\ .
                          \end{array}\right.
\eeq
Let $T^{C}_{\alpha,k_0}$ denote the aggregation of the three previous tests. We take $T^{C}_{\alpha,k_0}:=\max\big(T^{HC}_{\alpha/3,k_0},T^{B}_{\alpha/3,k_0},T^{I}_{\alpha/3,k_0}\Big)$,
if $k_0\geq 20\sqrt{n}$ and $T^{C}_{\alpha,k_0}:=\max(T^{HC}_{\alpha/2,k_0},T^{B}_{\alpha/2,k_0})$ else. The following result holds.

\begin{cor}\label{cor:power_combined2}There exist three constants $c$, $c_{\alpha,\beta}$, and $c'_{\alpha,\beta}$ such that the following holds for $n\geq c$. 
 The type I error  probability of $T^{C}_{\alpha,k_0}$ is smaller than $\alpha$. Besides, 
 $\P_{\theta}[T^{C}_{\alpha,k_0}=1]\geq 1-\beta$ for any vector $\theta$ such that 
 \beq\label{eq:upper_adaptatif_linfini}
 |\theta_{(k_0+q)}| \geq c_{\alpha,\beta}\sigma \psi_{k_0,q}\ , \text{ for some } q\in [n-k_0]\ .
 \eeq
  Also,  $\P_{\theta}[T^{C}_{\alpha,k_0}=1]\geq 1-\beta$ for any vector $\theta$ satisfying,
 \beq\label{eq:upper_adaptatif_l2}
  \theta\in \mathbb{B}_0(k_0+\Delta)\quad  \text{ and }\quad  d^2[\theta,\mathbb{B}_0(k_0)] \geq c'_{\alpha,\beta}\sigma^2 \Delta\psi^2_{k_0,\Delta}\ ,\, \text{for some $\Delta\in [n-k_0]$.}
 \eeq
\end{cor}
 
In view of Theorem \ref{prp:lower} and \eqref{eq:upper_adaptatif_l2} in  Corollary \ref{cor:power_combined2}, it holds that $\rho^{*}_{\alpha+\beta}[k_0,\Delta] \asymp_{\gamma} \sigma^2 \Delta \psi^2_{k_0,\delta}$. Besides, the test $T^{C}_{\alpha,k_0}$ simultaneously achieves (up to multiplicative constants) these minimax separation distances over all $\Delta\in [n-k_0]$. Condition \eqref{eq:upper_adaptatif_linfini} provides a complementary characterization of $T^{C}_{\alpha,k_0}$ power function. This bound will be central for sparsity estimation in the next section.

To conclude this section, we summarize the results on the  testing separation distance $\rho_{\gamma}^{*2}[k_0,\Delta]$ as depicted in Table \ref{Tab:KV} in the introduction. 
 For $k_0\leq \sqrt{n}$, $\rho_{\gamma}^{*}[k_0,\Delta]$ is of same order as the signal detection separation distance $\rho_{\gamma}^{*}[0,\Delta]$. 
 For $k_0 > \sqrt{n}$, the minimax-optimal separation distance $\rho_{\gamma}^{*}[k_0,\Delta]$ becomes significantly larger than the signal detection separation distance. The complexity of the null hypothesis
plays an important role in $\rho_{\gamma}^{*}[k_0,\Delta]$. For instance, when $k_0=n^{\zeta}$ with $\zeta>1/2$ and for $\Delta\geq k_0$, $\rho_{\gamma}^{*2}[k_0,\Delta]$ is of order $k_0/\log(n)$. Besides, for $k_0$ between $\sqrt{n^{1/2}k_0}$ and $k_0$, there is smooth transition from squared separation distances of order $\Delta\log(n)$ to $\Delta/\log(n)$.

\section{Sparsity estimation}\label{sec:estimation}

Given an observation $Y$, our goal is now to estimate the number $\|\theta\|_0$ of non-zero components of $\theta$. As explained in the introduction, we rephrase this estimation problem as a multiple testing problem. Let $\mathcal{H}= (H_{k})_{k=0,\ldots, n}$ denote the nested collection of all hypotheses $H_{k}$~\eqref{eq:hypotheses}. For a parameter $\theta$, the set of true hypotheses $\mathcal{T}(\theta)$ is the collection $\{H_k, k\geq \|\theta\|_0\}$ and the set of false hypotheses $\mathcal{R}(\theta)$ is the collection $\{H_k, k< \|\theta\|_0\}$. A multiple hypothesis test is a measurable collection $\widehat{\mathcal{R}}\subset \mathcal{R}$.

Let us make explicit the connection between these two problems. Given an estimator $\widehat{k}$ of $\|\theta\|_0$, taking $\widehat{\mathcal{R}}= \{H_k, k <\widehat{k}\}$ defines a multiple test. Conversely, consider a multiple test $\widehat{\mathcal{R}}$. Then, one may define the estimator $\widehat{k}=1+ \max \{k: H_k \in \widehat{\mathcal{R}}\}$. In our framework, a closed test $\widehat{\mathcal{R}}$ is a test that satisfies the property ``$H'\subset H$ and $H\subset \widehat{\mathcal{R}}$ implies $H'\subset \widehat{\mathcal{R}}$" (see e.g.\ ~\cite{2016_fromont}). It follows from the above constructions that sparsity estimators $\widehat{k}$ are in one to one correspondence with closed testing procedures.

The above correspondence leads us (i) to build estimators $\widehat{k}$ that rely on the test statistics introduced in the previous section and (ii) to evaluate the performances of $\widehat{k}$ in terms of separation distances of a multiples testing procedure.

\subsection{From single tests to multiple tests }

Fix some $\alpha\in (0,1)$.  As in the previous section, our estimator $\widehat{k}$ defined by 
\beq \label{eq:estimation_sparsity}
\widehat{k}:= \lceil\widehat{k}_{HC}\rceil \vee \lceil\widehat{k}_{B}\rceil \vee \lceil\widehat{k}_{I}\rceil
\eeq
is based on a combination of three statistics respectively corresponding to tests of the form  $T^{HC}_{\alpha,k_0}$, $T^{B}_{\alpha,k_0}$ and $T^{I}_{\alpha,k_0}$. However, contrary to these tests, we have to deal with  many null hypotheses.

\paragraph{Construction of $\widehat{k}_{HC}$}
Let $t_{*}:= t_{*,\alpha/3}^{HC}$ where $t_{*,\alpha/3}^{HC}$ is defined in Section \ref{sec:HCKV} and write $\cT=[t_*]$. Define the Higher-Criticism estimator of $\|\theta\|_0$ by 
\beq\label{eq:def_est_HC}
\widehat{k}_{HC} :=  N_{\sigma t_{*}} \bigvee\sup_{t\in \cT}  \frac{N_{\sigma t}- 2n\Phi(t)- u^{HC}_{t,\alpha/3}}{1-2\Phi(t)} \ ,
\eeq 
where $N_t$ and $u^{HC}_{t,\alpha}$ are introduced in Section \ref{sec:HCKV}. Note that $\widehat{k}_{HC}$ is quite similar to the estimator of Meinshausen and Rice~\cite{MR2275246} developed in a mixture model setting . Let us explain the rationale between this estimator. First, $N_{\sigma t_*}$ is the number of coordinates of $Y$ larger than $t_*$ (in absolute value). Deviation inequalities of the normal distribution enforce that, with high probability, each of these coordinates corresponds to a non-zero component of $\theta$. For $t\in \cT$,  Bernstein's inequality enforces that, with high probability, there are less than $2(n-\|\theta\|_0)\Phi(t)+ u^{HC}_{t,\alpha/3}$ components of $Y$ larger than $\sigma t$ in absolute values that correspond to null components $\theta_i$. As a consequence, $N_{\sigma t}- 2n\Phi(t)- u^{HC}_{t,\alpha/3}$ is, with high probability, a lower bound of the number of non-zero coordinates of $\theta$.

\paragraph{Construction of $\widehat{k}_B$ and $\widehat{k}_I$}

Following  the intuition explained in the introduction, it would be tempting to define $\widehat{k}_B-1$ as the largest $q\in [n]$ such that the test $T^B_{\alpha_{q},q}$ (with some suitable tuning parameters $\alpha_q$) rejects the null. However, this simple strategy leads to a logarithmic loss in comparison to the optimal testing separation rate. 
 As explained in  Sections  \ref{sec:BulkKV} and \ref{sec:InterKV}, the the  statistics $Z(s)$ and $V(r,w)$ involved in the tests $T^{B}_{\alpha,k_0}$ and $T^{I}_{\alpha,k_0}$ can be interpreted as (possibly biased) estimators of $\|\theta\|_0$. The bias and the variance of these estimators depends on choice of the tuning parameters $s$, $r$ and $w$. For instance, for a large value of $s$, the variance $Z(s)$ is higher but $\E_{\theta}[Z(s)]$ is close to $\|\theta\|_0$ (see Section \ref{sec:BulkKV}). This is why we shall compute these statistics for a large collection of tuning parameters.
 
 Introducing $k_{\min}:=\lceil \sqrt{n}\rceil$, we shall consider the dyadic collection $\cK_0:= \{k_{\min},2k_{\min},\ldots, k_{\max}\}$, where $k_{\max}\in (n/2;n]$. In order to calibrate this large collection of statistics, we have to adjust the thresholds $u^B_{k_0,\alpha}$ and $u^{I}_{k_0,l,\alpha}$ of the statistics. For any $k_0\in \cK_0$, denote $\alpha_{k_0}:= 2\alpha([1+\log_2(\tfrac{k_0}{k_{\min}})]^2 \pi^2)^{-1}$ so that $\sum_{k_0\in \cK_0}\alpha_{k_0}\leq \alpha/3$. 
Equipped with this notation, we define the Bulk and Intermediary estimators of $\|\theta\|_0$ as follows
\begin{eqnarray}\label{eq:hat_kb}
\widehat{k}_B&:=& \sup_{k_0\in \cK_0} Z(s_{k_0}) - u_{k_0,\alpha_{k_0}}^B\ , \\
 \widehat{k}_I&:=& \sup_{k_0\in \cK_0,\ k_0\geq 20\sqrt{n}}\,\,\sup_{l\in\cL_{k_0}}\frac{V(r_{k_0,l},w_l) - u^I_{k_0,l,\alpha_{k_0}}}{1+l/k_0} \label{eq:hat_kI}\ ,
 \end{eqnarray}
 where $Z(s)$, $V(r,w)$, $u_{k_0,\alpha}^{B}$ and $u_{k_0,l,\alpha}^I$ are introduced in Sections  \ref{sec:BulkKV} and \ref{sec:InterKV}.

 \medskip 
 
 \noindent 
 {\bf Remark}. The number of statistics required to compute $\widehat{k}$ is of order $\log^2(n)$.

 \subsection{Optimal sparsity estimation rates}

\begin{thm}\label{thm:estimation}
Fix any $\beta\in (0,1)$. There exists two positive constants $c_{\alpha,\beta}$ and $c'_{\alpha,\beta}$ such that the following hold for any $\theta\in\mathbb{R}^n$. With high probability, $\widehat{k}$ does not overestimate the number of non-zero components,
\beq\label{eq:level_estimator}
 \P_{\theta}\big[\widehat{k}>\|\theta\|_0\big.
 ]\leq \alpha\ .
\eeq
With probability larger than $1-\beta$, the vector  $\theta$ contains no more than $\widehat{k}$ large coefficients in the sense that 
\beq\label{eq:puissance_estimateur}
 \big|\theta_{(\widehat{k}+q)}\big|\leq c_{\alpha,\beta}\sigma \psi_{\hat{k},q}\ , \quad \quad\quad \quad \forall q=1,\ldots , n-\widehat{k}\ .
\eeq
and 
\beq\label{eq:puissance_estimateur2}
 d^2\big[\theta, \mathbb{B}_0(\widehat{k})\big]\leq c'_{\alpha,\beta}\sigma^2[\|\theta\|_0-\hat{k}]_+\psi^2_{\hat{k},(\|\theta\|_0-\hat{k})_+}\ ,
\eeq
where the sequence $\psi$ is defined in Equation~\eqref{eq:upper_distance_linfty}.
\end{thm}

As a consequence, outside an event of probability smaller than $\alpha+ \beta$, we have $\widehat{k}\leq \|\theta\|_0$ and $\theta$ is so close to $\mathbb{B}_0[\widehat{k}]$ that is is impossible to reliably decipher whether $\theta\in \mathbb{B}_0[\widehat{k}]$ or not. Alternatively, Theorem \ref{thm:estimation} provides the following data-driven certificate: with high probability and simultaneously for all $q\geq 1$, there are no more than $\widehat{k}+q$ coefficients larger (up to constants) than  $\psi_{\widehat{k},q}$.

Below, we state two straightforward corollaries of Theorem \ref{thm:estimation} providing alternative interpretations of the result. Recall the multiple testing procedure $\widehat{\cR}$ derived from $\widehat{k}$. 
\begin{cor}\label{cor:test_estimation}
The Family-wise error rate (FWER) of the procedure $\widehat{\cR}$ is controlled at level $\alpha$:
\[
\inf_{\theta\in \mathbb{R}^n}\mathbb{P}_{\theta}[\widehat{\mathcal{R}}\cap \mathcal{T}(\theta)\neq \emptyset ]\leq \alpha.
\]
Given $\beta\in(0,1)$, there exists a constant $c_{\alpha,\beta}$ such that the following holds for all $\theta\in \mathbb{R}^n$. With probability larger than $1-\beta$, $\widehat{\cR}$ contains all hypotheses  $H_k$ such that 
\[
 \sum_{i=1}^{\Delta}\theta_{(k+i)}^2\geq c_{\alpha,\beta }\Delta \psi^2_{k,\Delta}\ \quad  \text{ for some  }\Delta\in [1,n-k]\ .
\]
\end{cor}
In view of Section \ref{sec:testKV}, the multiple testing procedure $\widehat{\cR}$ simultaneously  performs as well as any minimax adaptive single test of the hypothesis $H_{k_0}$ for a given $k_0=0,\ldots, n-1$. In other words, the multiplicity of the hypotheses does not induce any loss.

For a given $\theta$, we can easily ''invert`` the conditions  \eqref{eq:puissance_estimateur} and  \eqref{eq:puissance_estimateur2} to control the error $|\widehat{k}-\|\theta\|_0|$.

\begin{cor}\label{cor:rate_estimation}
 There exists a positive constant $c_{\alpha,\beta}$  such that the following holds.
 For any $\theta\in \mathbb{R}^n$, the sparsity estimator satisfies the three following properties
 \begin{eqnarray}  \label{eq:lower|theta|_0}
 \widehat{k}&\leq& \|\theta\|_0 \ ,\\
 \label{eq:upper_rate_|theta|_0}
 (\|\theta\|_0-\widehat{k})_+ &<& \min \big\{q\, , \quad  \text{such that} \quad d^2_2(\theta, \mathbb{B}_0[\|\theta\|_0-q])\geq c_{\alpha,\beta}\sigma^2 q \psi^2_{\|\theta\|_0-q,q}    \big\}\ , \\
  \label{eq:upper_rate_|theta|_0_l_infty}
 \widehat{k}&\geq& 1+ \max \big\{r\, , \quad  \text{such that }\ \exists q \in [1,n-r ], |\theta_{(r+q)}|\geq c_{\alpha,\beta}\sigma \psi_{r,q}    \big\}\ ,
 \end{eqnarray}
 outside an event of probability smaller than $\alpha+\beta$. In the above equations, we choose the convention $\min\{\emptyset\}= \infty$ and $\max\{\emptyset\}=- \infty$.
\end{cor}

Conversely, it is not possible to improve the bounds \eqref{eq:upper_rate_|theta|_0} and \eqref{eq:upper_rate_|theta|_0_l_infty}.

\begin{cor}\label{cor:rate_estimation_optimality}
 There exists a positive constant $c'_{\alpha,\beta}$  such that the following holds.
 Fix  any integers $q>0$ and $k>0$ such that $k+q\leq n$. No estimator $\tilde{k}$ can satisfy simultaneously $\inf_{\theta\in \mathbb{B}_0[k]}\P_{\theta}[\tilde{k}\leq k]\geq 1-\alpha$ and at least one of the two following properties
 \begin{eqnarray}\label{eq:1_negative}
 \inf_{\theta\in \mathbb{B}_0[k+q,k, c'_{\alpha,\beta}\sigma \sqrt{q} \psi_{k,q}]}
\P_{\theta}[\tilde{k}\geq \|\theta\|_0- q ]\geq 1-\beta\ ,\\  \label{eq:2_negative}
\inf_{\theta\in \mathbb{R}^n,\  |\theta_{(k+q)}|\geq c'_{\alpha,\beta}\sigma \psi_{k,q}}
\P_{\theta}[\tilde{k}>k ]\geq 1-\beta\ .
  \end{eqnarray}
\end{cor}
For any fixed $(r,q)$, if we replace $\psi^2_{r,q}$ in \eqref{eq:upper_rate_|theta|_0} by $\tfrac{c'_{\alpha,\beta}}{c_{\alpha,\beta}}\psi^2_{r,q}$, then \eqref{eq:lower|theta|_0} cannot hold together with \eqref{eq:upper_rate_|theta|_0} on an event of large probability. 
The same optimality results holds for \eqref{eq:upper_rate_|theta|_0_l_infty}.

To better grasp the implication of \eqref{eq:upper_rate_|theta|_0}, let us consider a toy example for which $\|\theta\|_0= n^{\gamma}$ for some $\gamma\in (0,1)$ and given $\Delta\in [1,\ldots,\|\theta\|_0]$, we define $m^2_{\Delta}= \frac{1}{\Delta}\sum_{j=1}^{\Delta}\theta_{(\|\theta\|_0+1-j)}^2$ the mean square of the $\Delta$ smallest non-zero values of $\theta$. Note that $m_{\Delta}$ is a non-decreasing function of $\Delta$. It corresponds to the typical value of the $\Delta$ smallest non-zero components of $\theta$. 
Depending on the behavior of $m_{\Delta}$ we may bound the error of the estimator of $\|\theta\|_0$. First, if $m_1$ is large in front $\sqrt{\log(n)}$, then we have  $\widehat{k}=\|\theta\|_1$ with high probability. Then, we consider two subcases:
\begin{enumerate}
 \item[(i)] $\gamma\in (0,1/2)$. Take $\Delta=n^{\zeta}$ with $\zeta\in (0,\gamma]$. 
 \[\text{If }m_{\Delta}\geq c_{\alpha,\beta} \sigma   \sqrt{(1/2 - \zeta) \log(n)}\ , \quad \text{then } \frac{\|\theta\|_0-\widehat{k}}{\|\theta\|_0}\leq n^{\zeta-\gamma}\ .\] 
 Conversely, if $m_{\|\theta\|_0}\leq c'_{\alpha,\beta}  \sigma   \sqrt{(1/2 - \gamma) \log(n)}$, then it is impossible to distinguish $\theta$ from $0$. As a consequence, the relative estimation precision is mainly driven by the proportion of non-zero components that are large in front of $\sigma  \sqrt{\log(n)}$. 
 
 \item [(ii)] $\gamma\in (1/2,1)$. Here, the situation is more intricate:
 \begin{enumerate}
  \item [(a)] $\Delta=n^{\zeta}$ with $\zeta\in (0,\gamma)$.
  \[\text{ If }\, m_\Delta\geq c_{\alpha,\beta}\sigma  \big[\sqrt{2(\gamma-\zeta)}\wedge \frac{2(\gamma-\zeta)}{\sqrt{\gamma-1/2}} \big] \sqrt{ \log(n)},\quad  \text{ then }  \frac{\|\theta\|_0-\widehat{k}}{\|\theta\|_0}\leq n^{\zeta-\gamma}\ .\] In that case, all non-zero components of $\theta$ except a polynomially small proportion of them are larger than $\sigma  \sqrt{\log(n)}$ and the relative estimation error $\frac{|\|\theta\|_0-\widehat{k}|}{\|\theta\|_0}$ converges polynomially fast to zero.
  \item [(b)]  $\Delta=\frac{\|\theta\|_0}{u_n}$ with $u_n\to \infty$ and $u_nn^{-\zeta}\to 0$ for all $\zeta>0$ .  
  \[\text{If }m_\Delta\geq c_{\alpha,\beta}\sigma  \frac{\log(u_n)}{\sqrt{(\gamma-1/2)\log(n)}}\quad \text{ then }  \frac{\|\theta\|_0-\widehat{k}}{\|\theta\|_0}\leq \frac{1}{u_n}\ . \]
  For concreteness, fix $u_n=\log^{\zeta}(n)$ with $\zeta>0$. the relative convergence rate is of order $\log^{-\zeta}(n)$ if all non-zero components of $\theta$ except a proportion $u^{-1}_n$ of them are larger than $\sigma  \zeta\tfrac{\log\log(n)}{\sqrt{\log(n)}}$.

  \item [(c)] $\Delta= \zeta \|\theta\|_0$ with some $\zeta \in(0,1)$. If $m_\Delta\geq c_{\alpha,\beta}\sigma   \frac{\log(1/\zeta)}{\sqrt{\gamma\log(n)}}$, then $\frac{\|\theta\|_0-\widehat{k}}{\|\theta\|_0}\leq (1-\zeta)$. In that setting, a fixed proportion of non-zero coefficients are larger than $\sigma  \frac{1}{\sqrt{\log(n)}}$. One is able to estimate $\|\theta\|_0$ up to a constant multiplicative factor. 
  
  \item [(d)] $\Delta=  \|\theta\|_0(1- \log^{-\zeta}(n))$ with $\zeta>0$. 
  \[\text{If }m_{\Delta}\geq c_{\alpha,\beta}\sigma  \frac{1}{\sqrt{\gamma-1/2}\log^{(\zeta+1)/2}(n)},\quad  \text{ then }\, \widehat{k}\geq \|\theta\|_0\log^{-\zeta}(n)\ .\] In other words, if most non-zero coefficients of $\theta/\sigma$ are logarithmically small (at some power larger than $1/2$), it is still possible estimate the order of magnitude of $\|\theta\|_0$ up to some polylog multiplicative terms.

  \item [(e)] More generally, consider $\Delta=  \|\theta\|_0(1- \tfrac{1}{u_n})$ with $u_n\to \infty$. 
  \[\text{If }m_{\Delta}\geq c_{\alpha,\beta}\sigma  \frac{1}{\sqrt{u_n\log\big(1+ \frac{\|\theta\|_0}{u_n\sqrt{n}}}\big)},\quad  \text{ then }\, \widehat{k}\geq \frac{\|\theta\|_0}{u_n}\ .\]
  For instance take $u_n=n^{\zeta}$ for $\zeta\in (0,\gamma)$. Even if most non-zero components of $\theta$, are polynomially small, it is still possible to distinguish $\theta$ from zero, but it is just possible to estimate $\log(\|\theta\|_0)$ up to a multiplicative constant.
 \end{enumerate}
\end{enumerate}
Finally, let us emphasize that all these convergence rates are optimal in the sense of Corollaries \ref{cor:rate_estimation} and \ref{cor:rate_estimation_optimality}.

\medskip 

\noindent 
{\bf Comparison with the literature}. In \cite{MR2382653}, Cai et al.~consider an asymptotic framework where $\|\theta\|_0=n^{\gamma}$ with $\gamma\in (0,1/2)$ and $\theta$ only takes the values $0$ and $\sigma  \sqrt{2r\log(n)}$ for some $r>0$. These authors obtain convergence rates similar to Case (i) above but with explicit optimal constant $c(\alpha,\beta)$. In~\cite{MR2589318}, Cai and Jin consider an asymptotic framework where the non zero components of $\theta$  are sampled according to a fixed distribution with a smooth density $h$ in the sense that its characteristic function decays at rate not slower than $t^{-\alpha}$ for some $\alpha>2$. Their estimator $\widetilde{k}$~\cite[Sect.~3.1]{MR2589318} achieves a relative convergence rate of order $\log^{-\alpha/2}(n)$. However, if $h$ does not satisfy an uniform smoothness assumption, then $\widetilde{k}$ can be inconsistent. According to Case (ii,b), when $h$ is continuous at $0$, the relative convergence rate of our estimator $\widehat{k}$ is of order $\frac{\log\log(n)}{\sqrt{\log(n)}}$. This rate is slightly slower than that of Cai and Jin when $h$ is highly smooth, but our estimator is not tailored
to vectors $\theta$ that are sampled according to a smooth distribution and is valid for all $\theta$. This difference in the optimal rates highlights that our problem is qualitatively not the same as theirs in relevant cases.

\section{Sparsity testing with unknown variance}\label{sec:testUV}

In this part, we consider the problem of testing the sparsity of $\theta$ when the noise level $\sigma$ is unknown. For the sake of simplicity, it is assumed that $\sigma$ belongs to some fixed interval $[\sigma_-, \sigma_+]$ where $0< \sigma_-<\sigma_+$ are known. This assumption is not restrictive since, in most interesting settings, one may build a data-driven interval $[\widehat{\sigma}_-, \widehat{\sigma}_+]$ containing $\sigma$ with large probability and such that the ratio $\widehat{\sigma}_+/\widehat{\sigma}_{-}$ remains bounded. See below for further explanations.

In this section and in the corresponding proofs, we denote $\mathbb{P}_{\theta,\sigma}$ the distribution of $Y$. Given two integers $k_0\geq 0$ and $\Delta>0$, we consider the sparsity testing problem with unknown variance
\beq\label{eq:hypothesesuv}
 H_{k_0,\mathrm{var}}: \ \theta\in \bbB_0[k_0],\ \sigma \in [\sigma_-, \sigma_+]\quad \text{ versus }\quad  H_{\Delta,k_0,\rho,\mathrm{var}}:\ \theta \in \bbB_0[k_0+\Delta,k_0,\rho],\  \sigma \in [\sigma_-, \sigma_+]\ .
\eeq
Given a test $T$, let us define its risk $R_{\mathrm{var}}(T;k_0,\Delta, \rho)$ for the problem \eqref{eq:hypothesesuv} by
\beq\label{eq:riskuv}
R_{\mathrm{var}}(T;k_0,\Delta, \rho):= \sup_{\theta \in \bbB_0[k_0],\ \sigma \in [\sigma_-, \sigma_+]}\P_{\theta,\sigma}[T=1] +  \sup_{\theta \in \bbB_0[k_0+\Delta,k_0,\rho],\ \sigma \in [\sigma_-, \sigma_+]}\P_{\theta,\sigma}[T=0]\ ,
\eeq
and its $\gamma$-separation distance $\rho_{\gamma,\mathrm{var}}(T)$ by
\beq\label{eq:separationuv}
\rho_{\gamma,\mathrm{var}}(T;k_0,\Delta):= \sup \left\{\rho>0\ : R_{\mathrm{var}}(T;k_0,\Delta,\rho)>\gamma\right\}
\eeq
Finally, the minimax separation distance for the problem with unknown variance is defined by  
\beq\label{eq:separationuvminmax}
\rho^{*}_{\gamma,\mathrm{var}}[k_0,\Delta]:= \inf_{T}\rho_{\gamma,\mathrm{var}}(T;k_0,\Delta).
\eeq

\subsection{Detection problem ($k_0=0$)}\label{ss:sd}

Before turning to the general case, let us first restrict ourselves to the signal detection problem. To the best of our knowledge, the minimax separation distances for unknown variance have not been derived yet. Besides, this provides an introduction to the general case. Obviously, the problem with unknown variance is at least as difficult as the initial problem~\eqref{eq:hypotheses} so that, for all $\Delta$,  $\rho^{*}_{\gamma,\mathrm{var}}[k_0,\Delta]\geq \sigma_+ \rho^{*}_{\gamma}[k_0,\Delta]$. Our purpose is to pinpoint the range of $\Delta$ such that $\rho^{*}_{\gamma,\mathrm{var}}[k_0,\Delta]$ is of order $\rho^{*}_{\gamma}[k_0,\Delta]$ so that the the knowledge of the variance is not critical and the range of $\Delta$ such that $\rho^{*}_{\gamma,\mathrm{var}}[k_0,\Delta]$ is much larger than $\rho^{*}_{\gamma}[k_0,\Delta]$ so that the knowledge of the variance effectively makes the testing problem easier.

\begin{prp}\label{prp:signal_detection_uv}
Fix any $\gamma<0.25$. There exists two positive constants $c_{\gamma}$ and $c'_{\gamma}$ such that the following holds
For any  $\Delta \leq \sqrt{n}$, we have
\beq\label{eq:separation_distance_sduv_S}
c_{\gamma}\sigma_+^2 \Delta \log(1 + \frac{\sqrt{n}}{\Delta}) \leq \rho_{\gamma,\mathrm{var}}^{*2}[0,\Delta] \leq c'_{\gamma} \sigma_+^2 \Delta \log(1 + \frac{\sqrt{n}}{\Delta})\ .
\eeq 
For any $\eta<1/3$ and any $\Delta \in [\sqrt{n}, (\tfrac{1}{3}-\eta)n]$,
\beq\label{eq:separation_distance_sduv_L}
c_{\gamma}\sigma_+^2 \sqrt{\Delta n^{1/2}} \leq \rho_{\gamma,\mathrm{var}}^{*2}[0,\Delta] \leq c'_{\gamma,\eta} \sigma_+^2 \sqrt{\Delta n^{1/2}}  \ ,
\eeq 
where the constant $c_{\gamma,\eta}$ and $c'_{\gamma,\eta}$ only depend on $\gamma$ and $\eta$. 
\end{prp}

For $\Delta\leq \sqrt{n}$, the minimax separation distance is the same as for known variance. This can be achieved, for instance, by a generalization of the Higher Criticism to the  unknown variance setting as explained in Section \ref{sec:ubuv}.

For $\Delta$ between $\sqrt{n}$ and $n/3$, $\rho_{\gamma,\mathrm{var}}^{*2}[0,\Delta]$ is of order $\sqrt{\Delta n^{1/2}}$ which is much larger than the squared separation distance $\sqrt{n}$ for known variance. 
When $\sigma$ is known, a near optimal test amounts to reject the null hypothesis when $S_2=\|Y\|_2^2/\sigma^2-n$ is large in front of $\sqrt{n}$. Under the null, $S_2+n$ follows a $\chi^2$ distribution with $n$ degrees of freedom whereas, under the alternative, $S_2+n$ follows a non-central $\chi^2$ distribution with non-centrality parameter $\|\theta\|_2^2/\sigma^2$ so that the test is powerful when $\|\theta\|_2^2$ is large in front of $\sigma^2 \sqrt{n}$. When $\sigma$ is unknown, one cannot simply rely on the second moment of $Y$ and higher order moments are needed. For instance, a test achieving the separation distance \eqref{eq:separation_distance_sduv_L} is based on the statistic
\beq\label{eq:definition_S4}
 S_4 = \frac{n\|Y\|_4^4}{\|Y\|_2^2}-3
\eeq
Under the null, it follows from Chebychev inequality that $S_4=  O_P(n^{-1/2})$. Under the alternative, $\E_{\theta,\sigma}[\|Y\|_2^2]= \|\theta\|_2^2 +  n\sigma^2 $ and $\E_{\theta,\sigma}[\|Y\|_4^4]= \|\theta\|_4^4 + 6 \sigma^2 \|\theta\|_2^2+ 3n\sigma^2$ so that, one may expect that $S_4$ is of order 
\[\frac{n\|\theta\|_4^4 - 3\|\theta\|_2^4}{(\|\theta\|_2^2+ n\sigma^2)^2}\geq  (n-3\|\theta\|_0) \frac{\|\theta\|_4^4}{(\|\theta\|_2^2+ n\sigma^2)^2}\ , \]
by Cauchy-Schwarz inequality. As a consequence, one may expect that $S_4$ takes significantly larger values when $(n- 3\|\theta\|_0)\|\theta\|_4^4$ is large in front of $\sqrt{n}$. When $n-3\Delta$ is of order $n$, this occurs when $\|\theta\|_2^2$ is larger than $\sqrt{\Delta n^{1/2}}$.  See  the proof of Proposition \ref{prp:signal_detection_uv} for further details.

Conversely, the proof of the minimax lower bound \eqref{eq:separation_distance_sduv_L} also proceeds from moments arguments. For known variance $\sigma=1$, one builds a prior probability measure $\nu$ on $\theta$ supported by $\bbB_0[\Delta]$ such that the expectation of $\sum_{i=1}^n Y_i$ is the same under $\int \P_{\theta,\sigma}\nu(d \theta)$ and $\P_{0,\sigma}$. When the variance is unknown, one may choose $\sigma_1\neq \sigma_0$ such that all 
 expectations $\sum_{i=1}^n Y^q_i$ for $q=1,2,3$ are matching under $\int \P_{\theta,\sigma_1}\nu(d \theta)$ and $\P_{0,\sigma_0}$. As explained in the proof of Theorem \ref{thm:lower_bound_mainuv}, these moment matching properties translate into a smaller total variation between $\int \P_{\theta,\sigma_1}\nu(d \theta)$ and $\P_{0,\sigma_0}$ which in turn implies that the separation distance $ \rho^{*}_{\gamma,\mathrm{var}}[0,\Delta]$ is large.

\bigskip

Proposition \ref{prp:signal_detection_uv} above characterizes the signal detection separation distance for all $\Delta$ small in front of $n/3$. For $\Delta= c n$ with $c<1/3$, $\rho_{\gamma,\mathrm{var}}^{*2}[0,\Delta]$ is of order $n^{3/4}$. One may then wonder if $\rho_{\gamma,\mathrm{var}}^{*2}[0,\Delta]$ remains of order $n^{3/4}$ for all $\Delta \in (n/3,n]$. This turns out to be false. In fact, $\rho_{\gamma,\mathrm{var}}^{*2}[0,n]$ is of order $(\sigma^2_+-\sigma^2_-)n$. Indeed, let $\nu$ denote the centered normal distribution with  variance $(\sigma^2_+-\sigma_-^2)I_n$. When $\theta$ is sampled according to $\nu$ and for $\sigma=\sigma_-$, the marginal distribution of $Y$ is $\bbP_{0,\sigma^+}$. As a consequence, it is impossible to distinguish $\theta=0$ from $\theta\sim \nu$ for which $\|\theta\|_2^2$  is of order $(\sigma^2_+-\sigma^2_-)n$. This entails that $\rho_{\gamma,\mathrm{var}}^{*2}[0,n]$ is at least of order $(\sigma^2_+-\sigma^2_-)n$.

In fact, the squared minimax separation distance  $\rho_{\gamma,\mathrm{var}}^{*2}[0,\Delta]$ jumps above $n^{3/4}$ well before $\Delta=n$ as stated by the next proposition.
\begin{prp}\label{prp:n_div_3}
Consider any $0\leq \gamma\leq 0.25$. 
Fix any $\eta>0$ arbitrarily small and take  $\Delta=  \lfloor (\frac{1}{3}+ \eta)n \rfloor$. For $n$ large enough, we have 
\[
\rho_{\gamma,\mathrm{var}}^{*2}[k_0,\Delta] \geq c_{\eta} \sigma_+^2 n^{5/6}\ , 
\]
for some constant $c_\eta>0$ only depending on $\eta$. 
\end{prp}
As a consequence, the detection problem become much more difficult when $\Delta$ is above $n/3$ and the condition on $\Delta$ in Proposition \ref{prp:signal_detection_uv} is  tight. In comparison to the proof of the lower bound \eqref{eq:separation_distance_sduv_L}, for $\Delta$ larger than $n/3$, it is possible to define a prior measure $\nu$ supported on $\bbB_0[\Delta]$, $\sigma_0$ and $\sigma_1$ such that all expectations $\sum_{i=1}^n Y^q_i$ for $q=1,\ldots,5$ are matching under $\int \P_{\theta,\sigma_1}\nu(d \theta)$ and $\P_{0,\sigma_0}$. Matching these five moments then allows to recover the  $n^{5/6}$ rate. See the proof of Proposition \ref{prp:n_div_3} for details.

To summarize, for $\Delta\leq \sqrt{n}$ the minimax detection distance is the same as for known variance. For $\Delta\in [\sqrt{n}, cn]$ with $c<1/3$ the square minimax detection distance is of order $\sqrt{\Delta n^{1/2}}$ which is larger than its counterpart for known variance.  For $\Delta > cn$ with $c>1/3$, the difficulty of the testing problem greatly increases.

In view of this phenomenon, we shall restrict ourselves, for the general sparsity testing problems, to values  $(k_0,\Delta)$ such that $k_0+\Delta\leq cn$ where $c$ is some constant small enough.

\subsection{Lower bounds}\label{sec:lbuv}

For $\Delta \leq \sqrt{n}\vee k_0$ we simply use the lower bound $\rho_{\gamma,\mathrm{var}}^{*2}[k_0,\Delta]\geq \rho_{\gamma}^{*2}[k_0,\Delta]$ (where $\rho_{\gamma}^{*2}[k_0,\Delta]$ is defined for known $\sigma=\sigma_+$). The following corollary is then a direct consequence of Theorem~\ref{prp:lower}. 
\begin{cor}
Consider any $\gamma\leq 0.5$.  For any  $k_0\leq \sqrt{n}$ and $\Delta\leq n-k_0$, we have
 \beq\label{eq:lower_bound_main_ad}
 \rho_{\gamma,\mathrm{var}}^{*2}[k_0,\Delta] \geq  \sigma_+\Delta \log\Big[1+ \frac{ \sqrt{n}}{8\Delta} \Big]\ .
 \eeq
 There exists a numerical constant $c >0$ such that the following holds. For any $k_0>\sqrt{n}$ and $\Delta\leq k_0\wedge  (n-k_0)$, we have
 \beq\label{eq:lower_bound_main2_ad}
 \rho_{\gamma,\mathrm{var}}^{*2}[k_0,\Delta] \geq c \sigma_+
 \Delta \left[\frac{\log^2\big[1+  \frac{k_0}{\Delta}\big]}{\log\big[1+ \frac{k_0}{\sqrt{n}}\big]}\wedge \log\big[1+ \frac{k_0}{\Delta}\big] \right].
\eeq
\end{cor}

Additional work is needed to pinpoint the minimax separation distance $\rho_{\gamma,\mathrm{var}}^{*}[k_0,\Delta]$ for $\Delta \geq \sqrt{n}\vee k_0$. As for known variance, there are two different regimes  depending whether $k_0\leq \sqrt{n}$ or $k_0> \sqrt{n}$. 
\begin{thm}\label{thm:lower_bound_mainuv}
Consider any $0\leq \gamma\leq 0.25$. For any  $0\leq k_0 \leq \sqrt{n}$ and $\max(\sqrt{n}, 48) \leq \Delta\leq n-k_0$, we have 
 \beqn
 \rho_{\gamma,\mathrm{var}}^{*2}[k_0,\Delta] \geq c\sigma_+^2 \sqrt{\Delta n^{1/2}}\ ,
 \eeqn
 where $c$ is a numerical constant.
\end{thm}
For $k_0\leq \sqrt{n}$ and $\Delta\geq \sqrt{n}$, the separation distance $\rho_{\gamma,\mathrm{var}}^{*2}[k_0,\Delta]$ is the same as in the signal detection setting $\rho_{\gamma,\mathrm{var}}^{*2}[0,\Delta]$. In comparison to $\rho_{\gamma}^{*2}[k_0,\Delta]$, the squared distance $\sqrt{n}$ has increased up to $\sqrt{\Delta n^{1/2}}$. The intuition behind Theorem \ref{thm:lower_bound_mainuv} has been already described below Proposition \ref{prp:signal_detection_uv}.

\begin{thm}\label{thm:lower_bound_mainuv2} There exist three positive constants $c_1$, $c_2$, and $c_3$ such that the following holds. 
Assume that $n/c_1 \geq \Delta \geq c_1 k_0 \geq c_1\sqrt{n}$ and that $n\geq c_2$. Then, we have 
 \beqn
 \rho_{\gamma,\mathrm{var}}^{*2}[k_0,\Delta] \geq c_{3} \sigma^2_+ \frac{\sqrt{\Delta k_0}}{\log(1 + k_0/\sqrt{n})}.
 \eeqn
\end{thm}

In  the known variance setting, the squared separation distance is of order $\frac{k_0}{\log(1 + k_0/\sqrt{n})}$. The price to pay for not knowing the variance is a multiplicative factor of order $\sqrt{\Delta/k_0}$.

Contrary to the proof of Theorem \ref{prp:lower} for known variance, it is difficult to follow here a moment matching approach. Given two suitable  prior distributions $\mu_0^{\otimes n}$ and $\mu_1^{\otimes n}$ on $\theta$ and variances $\sigma_0^2$ and $\sigma_1^2$ in such a way that $\mu_0^{\otimes n}$ is almost supported in $\bbB_0[k_0]$ and $\mu_1^{\otimes n}$ is almost supported in $\bbB_0[k_0+\Delta,k_0,\rho]$, the goal is to prove that the two marginal distribution of $Y$, $\int \P_{\theta,\sigma_0}\mu^{\otimes n}_0(d\theta)$ and $\int \P_{\theta,\sigma_1}\mu^{\otimes n}_1(d\theta)$ are close to each other in total variation distance. Since the two last measures are product measures, this is equivalent to proving that the densities  $\pi_0(x):= \int \phi(\frac{t-x}{\sigma_0})\mu_0(dx)$ and $\pi_1(x):= \int \phi(\frac{t-x}{\sigma_1})\mu_1(dx)$ are close in $l_1$ distance
(recall that $\phi(.)$ denotes the density of the  standard normal distribution). It is difficult to obtain an analytic expression of the $l_1$ distance between two mixture distribution and hence one cannot directly choose 
the measure $\mu_0$ and $\mu_1$ minimizing this $l_1$ distance. As performed earlier in e.g.\  \cite{cailow2011,kalai2012disentangling}, we choose instead  $\mu_0$ and $\mu_1$ in such a way that the Fourier transforms $\widehat{\pi}_0$ and $\widehat{\pi}_1$ are matching for all frequencies small enough. Afterwards, we prove that this particular choice of $\mu_0$ and $\mu_1$ makes the $l_1$ distance between $\pi_0$ and $\pi_1$ small. 
Although the general approach is not new, the control of the $l_1$ distance is more delicate than in previous work, especially in the regime where $k_0$ is close to $\sqrt{n}$. In the proof, our implicit construction of the prior distributions $\mu_0$ may be of independent interest.

\subsection{Upper bounds}\label{sec:ubuv}

In this subsection, we build matching upper bounds for all $(k_0,\Delta)$ such that  $k_0+\Delta\leq cn$ where $c$ a numerical constant small enough. Indeed, when $\Delta$ is of order $n$, it has been proved in Proposition \ref{prp:n_div_3} that the detection problem becomes much more difficult, so that there is no hope to find tests matching Theorem \ref{thm:lower_bound_mainuv} and Theorem \ref{thm:lower_bound_mainuv2} when $k_0+\Delta$ is too large.  Note that, in the  regime $k_0+\Delta\leq cn$, one may construct a data-driven confidence interval of $\sigma$ so that the knowledge of the fixed interval $[\sigma_+,\sigma_-]$ is not really critical. In Appendix \ref{sec:appendix_estimation_sigma}, we provide such a confidence interval and we briefly explain how to how to extend the testing procedures to completely unknown variances $\sigma\in \mathbb{R}^+$. 

 Throughout this subsection, we consider some fixed $\alpha$ and $\beta$ in $(0,1)$.

\subsubsection{Adaptive Higher Criticism Statistic}\label{sec:HCUV}

The principle underlying the Higher Criticism is to compare the number $N_t$ of components of $Y$ larger than $t$ in absolute value to an upper bound of their expectation under the null, namely $k_0+ (n-k_0)\Phi(t/\sigma)$. This is why we adapt this test by plugging a suitable estimator of $\sigma$ and adding some correcting terms accounting for the variance estimation error. Let 
\beq\label{eq:def:siguv}
\widehat{\sigma}= \widehat{\sigma}^2(v):= -\frac{2}{v^2}\log\big[\overline{\varphi}_n(v)\big]\ ,\quad \text{ where }\quad \quad  v^2 := \frac{2}{\sigma^2_+}[\log(1+\frac{k_0}{\sqrt{n}})\lor 1]\ ,
\eeq
where we recall that $\overline{\varphi}_n$ is the empirical characteristic function~\eqref{eq:definition_empirical_fourier} of $Y$. Let us briefly explain the idea behind this definition  by replacing $\overline{\varphi}_n(v)$ by its expectation $\overline{\varphi}(v)$~\eqref{eq:definition_empirical_fourier}. Intuitively, $\widehat{\sigma}^2$ is expected to be of order 
\beq\label{eq:pop_analysis_sigub}
 -\frac{2}{v^2}\log\big[e^{-v^2\sigma^2/2}\frac{1}{n}\sum_i\cos(v\theta_i)\big]= \sigma^2  - \frac{2}{v^2}\log\big[\frac{1}{n}\sum_{i}\cos(v\theta_i)\big]\ , 
\eeq
so that when $\tfrac{1}{n}\sum_{i}\cos(v\theta_i)$ is close to one, $\widehat{\sigma}^2$ should be close to $\sigma^2$. Estimation of $\sigma$ based on the empirical characteristic function has been first tackled   by Cai and Jin~\cite{MR2325113,MR2589318}. Nevertheless, our estimator \eqref{eq:def:siguv} differs from theirs, as we do not assume that the non-zero components of $\theta$ are sampled from a smooth distribution.

Defining $t_{*,\alpha}^{HC,\mathrm{var}}:= \lceil 2 \sqrt{2\log(\frac{4n}{\alpha})} \rceil$, we 
consider the test $T^{HC,\mathrm{var}}_{\alpha,k_0}$ that rejects the null hypothesis, if either $N_{\sigma_+ t_{*,\alpha}^{HC,\mathrm{var}}}\geq k_0+1$ or if for some integer $t \geq  1$, 
\beq\label{eq:rejection_HC_ad}
N_{\sigma_+ t}\geq  k_0 +2(n-k_0)\Phi(\frac{t \sigma_+}{\hat \sigma})+ u_{t,\alpha}^{HC,\mathrm{var}},
\eeq
where 
\beq\label{eq:rejection_Ntuv}
   u_{t,\alpha}^{HC,\mathrm{var}} :=  \sqrt{4n\Phi\big(t\big)\log\left(\frac{t^2\pi^2}{\alpha}\right)} + \frac{2}{3}\log\left(\frac{t^2\pi^2}{\alpha}\right) +  8 t\frac{\sigma_+^3}{\sigma_-^3}\frac{k_0}{\log(1+\frac{k_0}{\sqrt{n}})}\phi\big(t\big) \sqrt{\log\big(\tfrac{6}{\alpha}\big)}\ .
\eeq
In comparison to the original calibration parameter $u_{t,\alpha}^{HC}$, the third term is new and accounts for the estimation error of $\sigma^2$. 

\begin{thm}\label{thm:HCuv}
Let $C$ be any constant larger than $1$. There exist constants $c$, $c'_{\alpha}$, $c''_{\beta,\sigma_+/\sigma_{-},C}$, and 
$c'''_{\alpha,\beta}$ such that the following holds. 
If $n\geq c'_{\alpha}$ and $k_0\leq n/c$, the type I error probability of $T^{B,\mathrm{var}}_{\alpha,k_0}$ is smaller than $\alpha$, that is 
\[
 \P_{\theta,\sigma}[T^{HC,\mathrm{var}}_{\alpha,k_0}=1]\leq \alpha \ , \quad \quad \forall \theta\in\mathbb{B}_0[k_0]\ . 
\]
Now assume that $n\geq c''_{\beta,\sigma_+/\sigma_{-},C}$.
Any $\theta\in \mathbb{R}^n$ satisfying $\|\theta\|_0\leq n/c$,
\begin{align}\label{the:cond2}
|\theta_{(k_0+q)}| \geq c'''_{\alpha,\beta}\sigma_+  \Big[\sqrt{\log(C)}+ \sqrt{\log\big(\tfrac{\sigma_+}{\sigma_{-}}\big)}+ \sqrt{\log\Big(2+ \tfrac{k_0\vee \sqrt{n}}{q}\Big)_+}\Big]\ ,
\end{align}
 for  some $q\in [1,n-k_0]$ and 
\begin{align}\label{eq:assu:cond2def}
\sum_{i=1}^n \big[(v\theta_i)^4\wedge 1\big]\leq  C (k_0\lor \sqrt{n})\ ,
\end{align}
belongs to the high probability rejection region of $T^{B,\mathrm{var}}_{\alpha,k_0}$, that is 
$\P_{\theta,\sigma}[T^{HC,\mathrm{var}}_{\alpha,k_0}=0]\leq \beta$.
\end{thm}

Condition \eqref{eq:assu:cond2def} aside, the behavior of $T^{HC,\mathrm{var}}_{\alpha,k_0}$ is similar to the one of $T^{HC}_{\alpha,k_0}$ as stated in Proposition \ref{prp:T_HC_levelpower}. In fact, Condition \eqref{eq:assu:cond2def} allows to bound the term $\frac{1}{n}\sum_{i}\cos(v\theta_i)$ in \eqref{eq:pop_analysis_sigub} and  ensures that $|\widehat{\sigma}^2-\sigma^2|$ is, with high probability, at most of order $\tfrac{k_0}{n \log(1+k_0/\sqrt{n})}$. When this condition \eqref{eq:assu:cond2def} is not met, we are unable to control the behavior of the adaptive Higher Criticism test. Nevertheless, it turns out that parameters $\theta$ not satisfying \eqref{eq:assu:cond2def} belong to the high-probability rejection region of the test $T^{B,\mathrm{var}}_{\alpha,k_0}$ described below so that a combination of $T^{HC,\mathrm{var}}_{\alpha,k_0}$ and $T^{B,\mathrm{var}}_{\alpha,k_0}$ achieves similar performances to the original Higher Criticism test $T^{HC,\mathrm{var}}_{\alpha,k_0}$. At the end of the section, the constant $C$ in Theorem \ref{thm:HCuv} will be carefully chosen to put the three tests $T^{HC,\mathrm{var}}$, $T^{B,\mathrm{var}}$ and $T^{I,\mathrm{var}}$ together.


\subsubsection{Detecting the signal in the bulk distribution}

Analogously to the above extension of the Higher-Criticism test, it would be natural to plug a variance estimator $\widehat{\sigma}^2$ in the statistic $Z(s)$~\eqref{eq:definition_Z} and then to build a test based on this data-driven statistic. Unfortunately, it turns out that the estimation error for such $\widehat{\sigma}$ is not negligible in our setting. Such a phenomenon is not unexpected as we have proved in Theorem \ref{thm:lower_bound_mainuv2} that no test in the unknown variance setting can perform as well as $T_{\alpha,k_0}^B$ for known $\sigma$. 

This is why we define a new statistic which is almost invariant with respect to the noise variance. Denoting  $P_B$ the linear polynom $P_B(\xi):= 4\xi-3$, we define, for $s>0$,  the statistic $Z^{\mathrm{var}}(s)$
\beq\label{definition_statistic}
Z^{\mathrm{var}}(s):=  n\int_0^1 P_B(\xi) \log\big[\big(\overline{\varphi}_n(\frac{s\xi}{\sigma_+})\big)_+\big]d\xi \ .
\eeq
The polynom $P_B$ has been defined in such a way that $\int_0^1 P_B(\xi)\xi^2d\xi=0$. To understand the rationale behind $Z^{\mathrm{var}}(s)$, let us assume that $\overline{\varphi}_n(s\xi)$ is close to its expectation $\overline{\varphi}(s\xi)$. Since for $x$ close to $1$, $\log(x)$ is approximately $x-1$, we obtain 
\beqn 
 Z^{\mathrm{var}}(s)&\approx& n\int_0^1 P_B(\xi) \Big[- \frac{\xi^2s^2\sigma^2}{2\sigma_+^2}+ \log\big(\frac{1}{n}\sum_{i=1}^n \cos(\frac{s\xi\theta_i}{\sigma_+}) \big)\Big]d\xi\\
 &\approx &  \sum_{i=1}^n  \int_0^1 P_B(\xi) \big(\cos(\frac{s\xi\theta_i}{\sigma_+})-1 \big)d\xi = \sum_{i=1}^n g(\frac{s\theta_i}{\sigma_+}) \ ,
\eeqn 
where  $g(x)=\int_0^1 P_B(\xi) \big(\cos(\xi x)-1 \big)d\xi$. For small $x$, a Taylor expansion of the $\cos$ function enforces that $g(x)\approx \int_0^1 P_B(\xi) [-\xi^2\tfrac{x^2}{2}+ \xi^4 \tfrac{x^4}{12}]d\xi= x^4 \int_{0}^1P_B(\xi)  \tfrac{\xi^4}{12}d\xi>0$. For larger $x$ (in absolute value), one can prove that $g(x)$ is positive and bounded away from zero. As a consequence, $\sum_{i=1}^n g(s\theta_i/\sigma_+)$ behaves like  $\sum_{i=1}^n (s\theta_i/\sigma_+)^4\wedge 1$ and approximates $\|\theta\|_0$. This informal discussion is made rigorous in the proof of Theorem \ref{thm:bulk_unknown} below.
In practice, we set 
\beq\label{eq:def_sbad}
s_{k_0}^{\mathrm{var}}=  \big[\sqrt{1+ \log\big(\tfrac{k_0}{n^{1/2}}\big)}\vee 1\big]\ , 
\eeq
and we define $T^{B,\mathrm{var}}_{\alpha,k_0}$ as the test rejecting the null hypothesis for large values of $Z^{\mathrm{var}}(s_{k_0}^{\mathrm{var}})$, that is when
\beq\label{eq:definition_Tb_alpha_ad}
Z^{\mathrm{var}}(s_{k_0}^{\mathrm{var}})\geq 1.09 k_0 + 16 \frac{k_0^2}{n }+ 4\sqrt{e} ( \sqrt{k_0n^{1/2}}\vee \sqrt{n}) \sqrt{\log(2/\alpha)}\ .
\eeq

\begin{thm}\label{thm:bulk_unknown}
There exist numerical constants $c$, $c'$, and $c''_{\alpha,\beta}$ such that the following holds. Assume that $n\geq c$ and that $k_0\leq c' n$.  
For any $k_0$-sparse vector $\theta$, the type I error probability of $T^{B,\mathrm{var}}_{\alpha,k_0}$ is small, that is 
\beq\label{eq:typeI_bulk_unknown}
 \P_{\theta,\sigma}[T^{B,\mathrm{var}}_{\alpha,k_0}=1]\leq \alpha + \frac{2(\|\theta\|_1/\sigma_++ n)}{n^4}\ . 
\eeq
Any $\theta\in \mathbb{R}^n$ such that $\|\theta\|_0\leq c'n$,  and 
\beq\label{eq:separation_test_intermediary_unknown_varianceBULK}
\sum_{i=k_0+1}^{n} \big[\big(\frac{s_{k_0}^{\mathrm{var}} \theta_{(i)}}{\sigma_+}\big)^4\wedge 1\big] \geq  c{''}_{\alpha,\beta} (k_0\vee \sqrt{n})
\eeq 
belongs to the high probability rejection region of $T^{B,\mathrm{var}}_{\alpha,k_0}$, that is 
\[
\P_{\theta,\sigma}[T^{B,\mathrm{var}}_{\alpha,k_0}=0]\leq \beta+   \frac{2(\|\theta\|_1/\sigma_++ n)}{n^4} \ .
\]
\end{thm}

The sufficient condition \eqref{eq:separation_test_intermediary_unknown_varianceBULK} for $T^{B,\mathrm{var}}_{\alpha,k_0}=1$ to be powerful corresponds to the heuristics described above. This condition will be the main ingredients towards matching the $ \sigma_+^2\tfrac{\sqrt{\Delta k_0}}{\log(1+ k_0/\sqrt{n})}$ separation distance of Theorem \ref{thm:lower_bound_mainuv2}.

The main downside to the above theorem is the presence of the small term $\|\theta\|_1/(\sigma_+n^4)$ in the type I and type II error probabilities. Although for typical parameters $\theta$ this term will be negligible, this makes the supremum of the type I error bound~\eqref{eq:typeI_bulk_unknown} over all $\theta\in \bbB_0[k_0]$. In Section \ref{sec:combination_uv}, we sketch a trimming approach which amounts to first discard components large components $Y$ 
and then apply the test to the trimmed vector  $\tilde{Y}$. The $l_1$ norm of the corresponding trimmed parameter $\tilde{\theta}$ is then small enough so that the type I and type II error probabilities are uniformly controlled.


\subsubsection{Intermediary regimes}

As for $T^{B}_{\alpha,k_0}$, one cannot easily adapt $T^{I}_{\alpha,k_0}$ by plugging an estimator of $\sigma$. Following the same approach as above we modify the statistic by considering the logarithm of the empirical characteristic function and multiplying it by some suitable polynom.

As the following test aims at discovering intermediary signals whose signature is neither in the bulk of the empirical distribution of $(Y_{i})$ nor in its extreme values, we restrict ourselves to the case  $k_0\geq 20\sqrt{n}$ (as for $T^{I}_{\alpha,k_0}$). Consider the dyadic collection $\mathcal{L}_{k_0}$ defined in Section \ref{sec:InterKV}. For $l\in \cL_{k_0}$, let 
 \beq \label{eq:paramUV}
  r_{k_0,l}:=\sqrt{16\log(\tfrac{k_0}{l})}\ ,\quad \quad w_l := \sqrt{\log(\tfrac{l}{\sqrt{n}})}\ .
\eeq
Note that, if $w_l$ is defined  as  in \eqref{eq:param} for $T^{I}_{\alpha,k_0}$, the definition of $r_{k_0,l}$ is slightly different.
Equipped with this notation, we consider the statistic
\beq\label{eq:def_eta_alt2_log}
V^{\mathrm{var}}(r_{k_0,l},w_l) :=   n r_{k_0,l} \int_{-1}^{1} P_l(r_{k_0,l}\xi) \phi(r_{k_0,l}\xi) \log\big[\overline{\varphi}_n\big(\frac{w_l\xi}{\sigma_+}\big)_+\big] d\xi\ ,
\eeq
where $P_l(t)=\gamma_l \big[\zeta_l t^2  - \kappa_l\big] $ with 
\begin{eqnarray} \label{eq:param_s}
\kappa_l&:=& -2r_{k_0,l}^3\phi(r_{k_0,l}) - 6r \phi(r_{k_0,l})+ 3\big(1- 2\Phi(r_{k_0,l})\big)\ ,\\ \zeta_l&:=& -2r_{k_0,l}\phi(r_{k_0,l}) + 1 - 2 \Phi(r_{k_0,l})\ , \nonumber\\
\quad 
\gamma_l &:= &[\kappa_l - \zeta_l]^{-1} \ , \quad \text{ and }\quad \delta_l := 4\gamma_l (r_{k_0,l} +4r_{k_0,l}^{-1})\phi(r_{k_0,l})\ . \nonumber
\end{eqnarray}
The purpose of this polynom $P_l$ is to cancel the  term $\int_{-1}^{1} P_l(r_{k_0,l}\xi) \phi(r_{k_0,l}\xi)\xi^2 d\xi$. Heuristically, $\log[\overline{\varphi}_n(w_l\xi/\sigma_+)_+]$ should be close to 
\[\log[\overline{\varphi}\big(\frac{w_l\xi}{\sigma_+}\big)_+]= -\frac{\sigma^2w_l^2\xi^2}{2\sigma_+^2} + \log\big[\frac{1}{n}\sum_{i}\cos\big(\frac{w_l\xi\theta_i}{\sigma_+}\big)\big]\approx  -\frac{\sigma^2w_l^2\xi^2}{2\sigma_+^2} + \frac{1}{n}\sum_{i}\big[\cos\big(\frac{w_l\xi\theta_i}{\sigma_+}\big)-1\big]\]
Since $P_l(r_{k_0,l}\xi) \phi(r_{k_0,l}\xi)$ is orthogonal to $\xi^2$, we expect that 
\[
 V^{\mathrm{var}}(r_{k_0,l},w_l)\approx \sum_{i=1}^n r_{k_0,l} \int_{-1}^{1} P_l(r_{k_0,l}\xi) \phi(r_{k_0,l}\xi) \big[\cos\big(\frac{w_l\xi\theta_i}{\sigma_+}\big)-1\big] d\xi\ . 
\]
Each term of this sum is zero for $\theta_i=0$. More generally, we show in the proof of Theorem \ref{cor:stat_intermediaire_unknown_variance} that, when $\theta$ does not contain too many large coefficients, this sum approximates the number of coefficient larger than $r_{k_0,l}^2/w_l$.

Finally, let $T^{I,\mathrm{var}}_{\alpha,k_0}$  be the test  rejecting the null hypothesis, if for some $l \in \cL_{k_0}$, $  V^{\mathrm{var}}(r_{k_0,l},w_l) $ is large enough, that is 
\beq\label{eq:rejection_intermediary_unknown}
  V^{\mathrm{var}}(r_{k_0,l},w_l) \geq k_0(1+\delta_l) +32\frac{k_0^2}{n}+ 8\sqrt{ln^{1/2}\log\Big(\frac{\pi^2 [1+\log_2(l/l_0)]^2}{3\alpha}\Big)}\ .
\eeq

\begin{thm}\label{cor:stat_intermediaire_unknown_variance}
There exist numerical constants $c$, $c'$, $c''_{\alpha,\beta}$, and $c'''_{\alpha,\beta}$ such that, for any $C>2$, the following holds. Assume that $n\geq c$ and that $k_0\leq c'n$.
For any $k_0$-sparse vector $\theta$, the type I error probability of $T^{I,\mathrm{var}}_{\alpha,k_0}$ is small, that is 
\[
 \P_{\theta,\sigma}[T^{I,\mathrm{var}}_{\alpha,k_0}=1]\leq \alpha + \frac{2(\|\theta\|_1/\sigma_++ n)}{n^4}\ . 
\]
Recall $s_{k_0}^{\mathrm{var}}$ defined in \eqref{eq:def_sbad}. Any parameter $\theta\in \mathbb{R}^n$ satisfying $\|\theta\|_0\leq c'n$ and the two following properties  
\begin{eqnarray}
 \label{eq:separation_test_intermediary_unknown_variance_simplified2}
\sum_{i=1}^n\mathbf{1}_{s_{k_0}^{\mathrm{var}}|\theta_i|\geq \sigma_+}&\leq& C k_0\ , \\
\label{eq:separation_test_intermediary_unknown_variance_simplified}
|\theta_{(k_0+q)}|&\geq& c''_{\alpha,\beta}\log(C)\sigma_+\frac{1+ \log(1+ \frac{k_0}{q})}{\sqrt{\log(1+ \frac{k_0}{\sqrt{n}})}}  \, \text{ for some } q\geq c^{'''}_{\alpha,\beta} C^2 \big[\sqrt{k_0n^{1/2}}\vee \frac{k_0^2}{n}\big]\ ,
\end{eqnarray}
belongs to the high probability rejection region of $T^{I,\mathrm{var}}_{\alpha,k_0}$, that is 
\[
\P_{\theta,\sigma}[T^{I,\mathrm{var}}_{\alpha,k_0}=0]\leq \beta+   \frac{2(\|\theta\|_1/\sigma_++ n)}{n^4} \ .
\]
\end{thm}

Condition \eqref{eq:separation_test_intermediary_unknown_variance_simplified} for $T^{I,\mathrm{var}}_{\alpha,k_0}$ to be powerful is analogous to Condition \eqref{eq:separation_test_intermediary} for $T^{I}_{\alpha,k_0}$ in the known variance setting except that $q$ is now restricted to be larger than $k_0^2/n$. This restriction will turn out to be benign except when $k_0$ is  too close to $n$. Also, contrary to Proposition \ref{prp:test_intermediary}, $\theta$ is assumed to contain less than $Ck_0$ coefficients larger than $\sigma_+/s_{k_0}^{\mathrm{var}}$ (which is of order $\sigma_+\log(k_0/\sqrt{n})^{-1/2}$). Again, this restriction is not a serious issue as $T^{B,\mathrm{var}}_{\alpha,k_0}$ is powerful for such $\theta$ not satisfying this assumption.

 \subsubsection{Combination of the tests}\label{sec:combination_uv}

For any integers $k_0\geq 0$ and $q>0$, define $\psi_{k_0,q}^{\mathrm{var}}>0$ by 
\beq \label{eq:upper_distance_linftyuv}
(\psi^{\mathrm{var}}_{k_0,q})^2:= \left\{\begin{array}{cc}
\sigma_+^2\log\Big[1+ \frac{\sqrt{n}}{q}\Big] & \text{if }k_0\leq \sqrt{n} \text{ and }q\leq \sqrt{n}\ ,\\
\sigma_+^2\big(\frac{\sqrt{n}}{q}\big)^{1/2} & \text{if }k_0\leq \sqrt{n} \text{ and }q> \sqrt{n}\\
\sigma_+^2\left(\frac{\log^2\big(1+  \tfrac{k_0}{q}\big)}{\log\big(1+\frac{ k_0}{\sqrt{n}}\big)} \bigwedge \log\Big[1+ \frac{k_0}{q}\Big]\right) & \text{if }k_0 > \sqrt{n} \text{ and }q\leq k_0\ ,\\
\sigma_+^2\frac{k^{1/2}_0}{q^{1/2}\log\big(1+\frac{ k_0}{\sqrt{n}}\big)} & \text{if }k_0 > \sqrt{n} \text{ and }q> k_0\ .
                          \end{array}\right.
\eeq

Let $T^{C,\mathrm{var}}_{\alpha,k_0}$ denote the aggregation of the three previous tests, that is $$T^{C,\mathrm{var}}_{\alpha,k_0}:=\max\big(T^{HC,\mathrm{var}}_{\alpha/3,k_0},T^{B,\mathrm{var}}_{\alpha/3,k_0},T^{I,\mathrm{var}}_{\alpha/3,k_0}\Big),~~~\text{if}~~k_0\geq 20\sqrt{n},$$ 
and $$T^{C,\mathrm{var}}_{\alpha,k_0}:=\max(T^{HC,\mathrm{var}}_{\alpha/2,k_0},T^{B,\mathrm{var}}_{\alpha/2,k_0}),~~~\text{else.}$$

As pointed out above, it is not possible to control uniformly the type I error probability of this test as such probabilities depend on the $l_1$ norm of $\theta$. This is why introduce a trimmed version of this test by removing large components of $Y$. Given $z>0$ and $V\in \mathbb{R}^n$, let $\cS(z;V)= \{i\in [n],\ |V_i|> (z+1) \sigma_+n^2 \}$. Let $U \sim \cU[0,1]$ be an uniformly distributed random variable independent of $Y$. We write $\cS(U,Y)= \cS[(U+1)\sigma_+ n^2;Y]$ for the coordinates $i$ such that $|Y_i|>(U+1)\sigma_+ n^2$. Let $\widetilde{Y}(\cS(U,Y)):=(Y_i), i\in ([n]\setminus \cS(U,Y))$ be the sub vector of $Y$ of size $n - |\cS(U,Y)|$. Finally, we define
the trimmed test  $\overline{T}^{C,\mathrm{var}}_{\alpha,k_0}$  rejecting the null hypothesis if either $k_0 - |\cS(U,Y)|$ is negative or if the test
$T^{C,\mathrm{var}}_{\alpha,k_0 - |\cS(U,Y)|}$ applied to the size $ n - |\cS(U,Y)|$ vector $\widetilde{Y}(\cS(U,Y))$  rejects the null hypothesis.

We use a random threshold $(U+1)\sigma_+ n^2$ instead of a deterministic one to make the subset $\cS$ of trimmed variable almost independent from $Y$, which facilitate the analysis of the two-step procedure  $\overline{T}^{C,\mathrm{var}}_{\alpha,k_0}$.

\begin{cor}\label{cor:power_combined2ubv2}
Fix any $\xi\in (0,1)$.
 There exist  positive constants $c$, $c'$, $c''_{\alpha,\beta,\xi}$ and $c'''_{\alpha,\beta,\xi}$ such that the following holds. Consider any $k_0\leq n^{1-\xi}$ and $n\geq c$. 
 Then, for any $\theta\in \bbB_{0}[k_0]$, one has 
 \[
  \P_{\theta,\sigma}[\overline{T}^{C,\mathrm{var}}_{\alpha,k_0}=1]\leq \alpha + \frac{c'\log(n)}{n}\ . 
 \]
Moreover,
 $\P_{\theta,\sigma}[\overline{T}^{C,\mathrm{var}}_{\alpha,k_0}=1]\geq 1-\beta-\frac{c'\log(n)}{n}$ for any vector $\theta$ satisfying $\|\theta\|_0\leq c'n$ and
 \beq\label{eq:upper_adaptatif_linfiniuv}
 |\theta_{(k_0+q)}| \geq c''_{\alpha,\beta,\xi}\sigma_+\psi_{k_0,q}^{\mathrm{var}}\ , \text{ for some } q\in [1,n-k_0]\ .
 \eeq
  Also,  $\P_{\theta,\sigma}[\overline{T}^{C,\mathrm{var}}_{\alpha,k_0}=1]\geq 1-\beta-\frac{c'\log(n)}{n}$ for any vector $\theta$ satisfying  
 \beq\label{eq:upper_adaptatif_l2uv}
  \theta\in \mathbb{B}_0(k_0+\Delta)\quad  \text{ and }\quad  d^2[\theta,\mathbb{B}_0(k_0)] \geq c'''_{\alpha,\beta,\xi} \sigma_+^2\Delta(\psi_{k_0,\Delta}^{\mathrm{var}})^2\ ,\, \text{for some $\Delta\in [1,c'n-k_0]$.}
 \eeq
\end{cor}

 As a consequence, for $k_0\leq n^{1-\xi}$ (and $\xi$ is an arbitrary constant in $(0,1)$), $\overline{T}^{C}_{\alpha,k_0}$ simultaneously achieves the minimax separation distance for all $\Delta$ such that  $k_0+\Delta\leq c n$ where $c$ is constant small enough. 
 
Building on the statistics introduced in this section, one can  then construct an adaptive estimator of the sparsity for unknown variance in the spirit of what has been done in Section \ref{sec:estimation}. For reasons of space, we do not pursue in this direction.

\section{Discussion}

\subsection{Other noise distributions}

Some of our testing procedures heavily rely on the assumption that the noise's distribution is Gaussian. For instance, the behavior of the Bulk and intermediary statistics depends on the exact form of the characteristic function of the noise. The radical change in the rates between the known variance case, and the unknown variance case, is already eloquent enough on the importance of knowing the exact shape of the noise distribution - even a slight deformation of the noise distribution by changing the variance has a strong effect on the minimax separation distances. We may consider two different extensions to non-Gaussian noises:
\begin{enumerate}
 \item The noise distribution is not Gaussian but is explicitly known. For the sake of discussion, let us also assume that it is symmetric. In that case, one could adapt the higher criticism statistic by replacing $\Phi(.)$ by the survival function of this distribution. Also, both the bulk and intermediary statistic could be accommodated by replacing $\exp(-\xi^2 w^2/2)$ in \eqref{eq:defintion_kappa} by the characteristic function of the noise distribution. Nevertheless, some additional work would be needed to adapt the lower bounds
 \item Only an upper bound of the tail distribution of the noise is known. For instance, the noise is only assumed to be sub-Gaussian with a bounded sub-Gaussian norm. In that situation, one cannot rely anymore on its characteristic function. Nevertheless, one could adapt some signal detection tests~\cite{baraud02} to build ``infimum test''~\cite{gayraud2005adaptive,nickl_vandegeer} such as those described in the introduction. From rough calculations, it seems that the corresponding test would achieve the optimal separation distances up to polylogarithmic multiplicative terms. It remains an open problem to understand whether this polylog loss is intrinsic or not. 
\end{enumerate}

\subsection{Other models}

The same general roadmap can be pursued to estimate discrete functionals in many other problems, including  rank estimation in matrix regression and matrix completion models, smoothness estimation in the density framework, number of clusters estimation in model-based clustering,\ldots. A prominent example is sparsity estimation in  the high-dimensional linear regression model. Let $Y\in \mathbb{R}^n$, ${\bf X}\in \mathbb{R}^{n\times p}$ be such that 
\[
 Y = {\bf X}\theta + \epsilon\ , 
\]
where the parameter $\theta\in \mathbb{R}^p$ is unknown and $\epsilon=(\epsilon_i)$ is made of centered independent normal distributions with variance $\sigma^2$. In the specific case where $n=p$ and ${\bf X}$ is the identity matrix, it is is equivalent to Gaussian vector	 model~\eqref{eq:model}. Estimation of $\theta$ under sparsity assumptions has received a lot of attention in the last decade~\cite{buhlmann2011statistics}. In the specific case 
where the entries of $\bX$ are independently sampled according to the standard normal distribution, the minimax separation distances for the detection problem has been derived in~\cite{2010_EJS_Ingster,2011_AS_Arias-Castro}. For the purpose of building adaptive confidence intervals,  Nickl and van de Geer~\cite{nickl_vandegeer} have introduced and analyzed sparsity testing procedures. However, the optimal separation distances for the sparsity testing problem remain unknown (except in some specific regimes). Further work is therefore needed to establish the minimax separation distances and to construct adaptive sparsity tests and sparsity estimators.

\paragraph{Acknowledgements.} The work of A. Carpentier is supported by the Deutsche Forschungsgemeinschaft (DFG) Emmy Noether grant MuSyAD (CA 1488/1-1). The authors thank Christophe Giraud for careful rereading and insightful suggestions on the presentation of the results.

\bibliography{biblio}
\bibliographystyle{plain}

\pagebreak

\appendix 

\begin{center}
\noindent
{\huge Supplementary Material for the paper "Adaptive estimation of the sparsity in the Gaussian vector model"}
\bigskip


\end{center}

\bigskip
\bigskip

\section{Estimation of $\sigma_-$ and $\sigma_+$ and full adaptation to unknown $\sigma$}\label{sec:appendix_estimation_sigma}

The purpose of this section is to exhibit a confidence interval of $\sigma$ that. This allows us to first estimate $[\hat{\sigma}_-,\hat{\sigma}_+]$ and plug this confidence interval in the testing procedures of Section \ref{sec:testUV}.

\begin{lem}\label{lem:calcsig}
There exists some universal constant $c>0$ such that the following holds for any $\theta \in \bbB_0[n/2]$. Define 
$$\bar \sigma^2 := \frac{2}{n} \sum_{i\geq n/2+1} Y_{(i)}^2\ ,\quad \tilde \sigma := 2^{\lfloor \log(\bar \sigma)/\log(2) \rfloor}\ ,$$  $\hat{\sigma}_+ :=  2.2\tilde \sigma ~~\text{and}~~\hat{\sigma}_- := \tilde \sigma /16.$
With probability higher than $1- 2e^{-cn}$, we know that 
$$\sigma \in [\hat{\sigma}_-, \hat{\sigma}_+],~~\text{with}~~\frac{\sigma_+}{\sigma_-}\leq 40,\quad\text{and }\,
 \tilde \sigma \in \Big\{2^{\lfloor \log( \sigma)/\log(2)\rfloor+x}\ ,\, x= -4,-3,\ldots ,2\Big\}.$$
\end{lem}
Outside an event of exponentially small probability, $[\hat{\sigma}_-,\hat{\sigma}_+]$ only takes seven possible values. Then, conditioning on each of these seven events, one analyzes the behavior of the tests $T_{\alpha,k_0}^{HC,\mathrm{var}}$, $T_{\alpha,k_0}^{B,\mathrm{var}}$, and $T_{\alpha,k_0}^{I,\mathrm{var}}$ to control the risk of the corresponding fully data-driven procedures.

\begin{proof}[Proof of Lemma~\ref{lem:calcsig}]
The proof follows closely that of Proposition 1 in \cite{collier2016optimal}.
For the sake of simplicity, we assume that $n$ is even. Let $\mathcal S$ be a set of size $n/2$ that does not intersect with the support of $\theta$. Then, 
$$\frac{n\bar \sigma^2}{2\sigma^2}\leq  \sum_{i\in \mathcal{S}} \frac{\epsilon_i^2}{\sigma^2}\ ,$$
the last random variable following a $\chi^2$ distribution with $n$ degrees of freedom. By \cite{book_concentration}, we know that 
\[
 \mathbb{P}\big[\bar \sigma^2 > 1.1 \sigma^2\big]\leq e^{-c n}\ ,
\]
where $c$ is some positive universal constant. Next, let $\mathcal{G}$ the collection of subsets of $[n]$ of size $n/2$. We shall control the deviations of the random variables $Z_G:= \frac{1}{\sigma^2}\sum_{i\in G} Y_i^2$ uniformly over all $G\in \mathcal{G}$. Fix any $G\in \mathcal{G}$. The random variable $Z_G$ follows a $\chi^2$ distribution with $n/2$ degrees of freedom and non-centrality parameter $\sum_{i\in G}\theta_i^2/\sigma^2$. In particular, this distribution is stochastically larger than a (central) $\chi^2$ distribution with $n/2$ degrees of freedom. Let $Z$ be a random variable sampled according to this distribution. 
By Lemma 11.1 in \cite{MR2879672}, we know that for any $x>0$, 
\[
 \mathbb{P}\big[Z \leq \frac{n}{2e}x^{4/n}\big]\leq x
\]
Take $x= \binom{n}{n/2}^{-1}e^{- n/8}$. It follows that $\log(1/x)\leq n(\tfrac{1}{8}+ \log(2))$. Taking an union bound over all $Z_G$ for $G\in \mathcal{G}$, we conclude that 
\[
 \mathbb{P}\Big[\inf_{G\in \mathcal{G}}Z_G \leq \frac{n}{16e^{3/2}}\Big]\leq e^{-n/4}\ .
\]
 Since $\bar{\sigma}^2= \frac{2}{n}\sigma^2    \inf_G Z_G$, this implies that,  with high probability, $\bar{\sigma}^2\geq \frac{\sigma^2}{16e^{3/2}}$. We have proved that with high probability, 
 \[
  0.9 \leq \frac{\sigma}{\bar{\sigma}}\leq 8.5
 \]
The remainder of the proof follows easily.
\end{proof}

\section{Proofs of the results with known variance}
 
In all the proofs in this section, we assume by homogeneity and without loss of generality that $\sigma=1$. 
 
 \subsection{Proofs of the testing lower bounds with known variance}

The minimax separation distance $\rho^*_{\gamma}[k_0,\Delta]$ depends on $\gamma$, $n$, $k_0$ and $\Delta$. In these proofs, we shall relate the minimax separation distances for different values of the sample size. To make the arguments clearer, we explicit the dependency of it on the sample size and write $\rho^*_{\gamma}[n,k_0,\Delta]$ instead of $\rho^*_{\gamma}[k_0,\Delta]$ in this subsection.

\paragraph{Step 1 : Reduction of the problem.}

We start by simple reduction arguments to narrow the range of parameters.

\begin{lem}\label{lem:reduction}
For any $ k'_0\leq k_0$,
 \beq\label{eq:subproblem}
 \rho_{\gamma}^{*}[n,k_0,\Delta] \geq \rho_{\gamma}^{*}[n-k_0+k'_0,k'_0,\Delta]\ .
 \eeq
For any $ \Delta'\leq \Delta\leq n-k_0$, 
 \beq\label{eq:monotonic}
  \rho_{\gamma}^{*}[n,k_0,\Delta]\geq \rho_{\gamma}^{*}[n,k_0,\Delta'].
 \eeq
 Finally, 
 \beq\label{eq:reduction_sample}
 \rho_{\gamma}^{*}[n,k_0,\Delta]\geq \rho_{\gamma}^{*}[n',k_0,\Delta]\ ,\quad  \text{ for any }n\geq n'\ .
 \eeq
\end{lem}
\begin{proof}[Proof of Lemma \ref{lem:reduction}]
The second bound is a  consequence of the inclusion $\bbB_{0}[k_0+\Delta, k_0,\rho]\subset \bbB_{0}[k_0+\Delta', k_0,\rho]$. The third bound is also trivial. Let us turn to \eqref{eq:subproblem}, consider any $\zeta>0$ arbitrarily small and let $r:= \rho_{\gamma}^{*}[n,k_0,\Delta]+\zeta$. There exists a test $T$ satisfying $R[T;k_0,\Delta,r]\leq \gamma$. For any $n-k_0+k'_0$-dimensional vector $Y$ with mean $\theta$, extend it to $\tilde{Y}$ by adding $k_0-k'_0$ components following independent standard normal distribution with mean $r$. Since $R[T;k_0,\Delta,r]\leq \gamma$, we have
\[\sup_{\theta,\ \|\theta\|_0\leq k'_0} \mathbb P_{\theta}[T(\tilde{Y})=1]+ \sup_{\theta,\ \|\theta\|_0\leq k'_0+ \Delta,\ d_2(\theta,\bbB_0(k'_0))\geq r}\mathbb P_{\theta}[T(\tilde{Y})=0]\leq \gamma\]
implying that $\rho_{\gamma}^{*}[n-k_0+k'_0,k'_0,\Delta]\leq r$. Considering the infimum over all $\zeta>0$, we obtain \eqref{eq:subproblem}.
\end{proof}

As a consequence of the above lemma, we obtain the following reduction. 
\begin{prp}\label{prp:lower_reduced}
Theorem \ref{prp:lower} is true as soon as 
\begin{eqnarray}\label{eq:lower_bound_main_simple}
 \rho_{\gamma}^{*2}[0,\Delta]& \geq& \Delta \log\Big[1+ \frac{ \sqrt{n}}{4\Delta} \Big]\ ,\quad \text{ for any }\Delta\leq n\\
 \rho_{\gamma}^{*2}[n,n-\Delta,\Delta]& \geq& \Delta \log\Big[1  + \frac{n}{8\Delta^2}\Big]\ ,\quad \text{ for any }\Delta\leq n \label{eq:lower_bound_main3_simple} \\
 \rho_{\gamma}^{*2}[n,k_0,\Delta] &\geq& c\Delta \left[\frac{\log^2\big[1+  \frac{k_0}{\Delta}\big]}{\log\big[1+ \frac{k_0}{\sqrt{n}}\big]}\wedge \log\big[1+ \frac{k_0}{\Delta}\big]\right]\ ,
  \label{eq:lower_bound_main2_simple} 
\end{eqnarray}
 for any $k_0>\sqrt{n}$ and  $32\sqrt{(n-k_0)\wedge  k_0}\leq \Delta\leq k_0$.

\end{prp}

The proof of this reduction is postponed to the end of the subsection. In the sequel, we focus on (\ref{eq:lower_bound_main_simple}--\ref{eq:lower_bound_main2_simple}).
The first bound \eqref{eq:lower_bound_main_simple} has already been shown in \cite{baraud02}. For the sake of completeness, we shall provide a proof of it together with \eqref{eq:lower_bound_main3_simple}. Prior to this, we focus on \eqref{eq:lower_bound_main2_simple}.

\paragraph{Step 2. Le Cam's method.}

In this step, we explain the general strategy for proving the minimax lower bound, allowing us to introduce the main notation. 
We start by introducing probability measures on the space of parameters $\theta$. 
Fix some $\gamma\in (0,1/2)$.  Define $\overline{k}_0=k_0-\Delta/2$, $k_1=k_0+\Delta$ (the sparsity of the alternative) and  $\overline{k}_1=k_0+\Delta/2$. 

Let $m\geq 1$, $M > 0$, and $a_m>0$ be quantities whose values will be fixed later. Below, we shall build two symmetric probability measures $\mu_0$ and $\mu_1$ whose support is included in
 \beq\label{eq:support_mu}
 [-M,-a_{m}M] \bigcup [a_{m}M, M]\ .
 \eeq
Given $\mu_0$ and $\mu_1$, consider the probability measures $\overline{\mu}_0$ and $\overline{\mu}_1$ 
\begin{align*}
\overline{\mu}_1 = \frac{ \overline{k}_1}{ n}\mu_1 + (1 - \frac{ \overline{k}_1}{ n}) \delta_0  \quad \quad \mathrm{and}\quad 
\overline{\mu}_0(A):= \frac{ \overline{k}_0}{ n}\mu_0 +   (1 - \frac{ \overline{k}_0}{ n}) \delta_0\ .
\end{align*}
Let $\overline{\mu}_0^{\otimes n}$ and  $\overline{\mu}_1^{\otimes}$ be the corresponding $n$-dimensional product measure. Note that, when $\theta\sim \overline{\mu}_0^{\otimes n}$, its number of non-zero coefficients follows a Binomial distribution with parameters $n$ and $\overline{k}_0$.

Finally, we define
\[
{\mathbf  P}_0 := \int \P_{\theta}  \overline{\mu}^{\otimes n}_0(d\theta) \ ,\quad \quad {\mathbf  P}_1:= \int \mathcal \P_{\theta} \overline{\mu}^{\otimes n}_1(d\theta)\ .
\]
 the marginal probability distribution of $Y$ when $\theta\sim \overline{\mu}^{\otimes n}_0$ (resp. $\theta\sim \overline{\mu}^{\otimes n}_1$). By Chebychev inequality,  
 \beqn 
 \overline{\mu}_0^{\otimes n}\big[\|\theta\|_{0}> k_0\big]&\leq& \frac{4\overline{k}_0(n-\overline{k}_0)}{n\Delta^2} \leq \frac{4k_0(n-k_0)}{n\Delta^2} + \frac{2k_0}{n\Delta} \\
 &\leq&  \frac{4}{32^2}+ \frac{2 k_0}{16n \sqrt{k_0\wedge (n-k_0)}} \leq \tfrac{4}{32^2} + \tfrac{1}{8}  \leq 1/7.
 \eeqn 
  Similarly,
 \beqn \overline{\mu}_1^{\otimes n}\big[|\|\theta\|_{0}- (k_0+\Delta/2)|>\Delta/4\big]&\leq& \frac{16\overline{k}_1(n-\overline{k}_1)}{n\Delta^2} \leq \frac{16k_0(n-k_0)}{n\Delta^2} + \frac{8(n-k_0)}{n\Delta}\\
 &\leq& \frac{1}{32}+ \frac{1}{4}\leq \frac{9}{32} .
 \eeqn 
 With $\overline{\mu}_1^{\otimes n}$-probability larger than $1-9/32$, $\theta$ is therefore $k_1$-sparse and $d_2^2(\theta,\bbB_0(k_0))\geq \Delta a_{m}^2M^2/4$. Given any test $T$, we apply Fubini identity to lower bound  its risk \eqref{eq:risk} as
 \beqn 
 R[T;k_0,\Delta,\Delta^{1/2}a_m M/2]&=& \sup_{\theta \in \bbB_0[k_0]}\P_{\theta}[T=1] +  \sup_{\theta \in \bbB_0[k_1,k_0,\Delta^{1/2}a_m M/2]}\P_{\theta}[T=0]\\
 &\geq& \int \P_{\theta}[T=1]\overline{\mu}_0^{\otimes n}(d\theta) - \overline{\mu}_0^{\otimes n}[\|\theta\|_{0}> k_0] \\ &&  + \int \P_{\theta}[T=0]\overline{\mu}_1^{\otimes n}(d\theta)- \overline{\mu}_1^{\otimes n}\big[|\|\theta\|_{0}- (k_0+\Delta/2)|>\Delta/4\big]\\
 &\geq & \mathbf{P}_0[T=1] + \mathbf{P}_{1}[T=0] - 0.45 = 0.55  + \mathbf{P}_{1}[T=0]- \mathbf{P}_0[T=0]\\
 &\geq & 0.55 - \|\mathbf{P}_0 - \mathbf{P}_1\|_{TV}\ . 
 \eeqn 
 As a consequence, the minimax separation distance  $\rho^*_{\gamma}[n,k_0,k_1]$ is larger than $\Delta^{1/2}a_{m}M/2$, as soon as 
 \beq\label{eq:TV_objective}
 \|\mathbf{P}_0 - \mathbf{P}_1\|_{TV}\leq \delta 
 \eeq
 where $\delta:= 0.55 -\gamma\geq 0.05$.

 In the remainder of the proof, we shall construct the prior measures $\mu_0$ and $\mu_1$ and give explicit values to the quantities $m$, $a_{m}$ and $M$ so that \eqref{eq:TV_objective} is satisfied and $\Delta^{1/2}a_{m}M$ is the largest possible.

\paragraph{Step 3: Construction of the prior distributions $\mu_0$ and $\mu_1$.}
We choose prior measures $\mu_0$ and $\mu_1$ such that the first moments of $\overline{\mu}_0$ and $\overline{\mu}_1$ are matching while $a_{m}$ and $M$ are as large as possible. The following lemma proved at the end of the subsection  ensures the existence of such probability measure for a certain choice of $a_m$. 

\begin{lem}\label{lem:nemirovski2}
Given any positive and even integer $m$ and $p\in (0,1)$, define 
\beq\label{eq:defintion_am}
a_{m}:= \tanh\Big[\frac{1}{m}\ \arg\cosh\big(\frac{1+p}{1-p}\big)\Big]\ .
\eeq
 There exists two positive and symmetric measures $\nu_0$ and $\nu_1$ whose support lie in $[-1,-a_{m}]\cup [a_{m},1]$ satisfying:
\begin{eqnarray}
 \int \nu_0(dt)&=&p\quad \quad \int \nu_1(dt)=1 \label{eq:condition_moment0} \\
 \int t^q \nu_0(dt)&= &\int t^q \nu_1(dt) ,\quad \quad q=1,\ldots, m\ .\label{eq:condition_momentq}
\end{eqnarray}
\end{lem}
The implicit construction of $\nu_0$ and $\nu_1$ is based on a careful application of Hahn-Banach theorem together with extremal properties of Chebychev polynomials. It is inspired by the work of \cite{Juditsky_convexity}, but we go one step further to obtain the right dependency of $a_{m}$ with respect to $p$.

Fix $p=\overline{k}_0/\overline{k}_1$. Then, given $m\geq 1$, we consider the measures $\nu_0$ and $\nu_1$ as defined by Lemma \ref{lem:nemirovski2} and the following remark. For $p=0$, we can define $\mu_0$ arbitrarily (take for instance $\mu_0= 0.5\delta_{M}+ 0.5\delta_{-M}$). Given any measurable event $A$, we define $\mu_0$ (for $p\in (0,1)$) and $\mu_1$ by 
\beq
\mu_0(A):= p^{-1}\nu_0[M  . A] \ , \quad    \mu_1(A):= \nu_1[M .A] \ .
\eeq
Note that $\mu_0$ and $\mu_1$ are symmetric and satisfy the support property \eqref{eq:support_mu} claimed at the beginning of the proof.
Moreover, $\mu_0$ and $\mu_1$ have been defined in such a way that the moments of $\overline{\mu}_0$ and $\overline{\mu}_1$ are matching
\begin{equation}\label{eq:mom}
\int t^q \overline{\mu}_0(dt)= \int t^q \overline{\mu}_1(dt) ,\quad \quad q=1,\ldots, m\ .
\end{equation}

\paragraph{Step 4: Choice of $m$ and $M$.}

In the sequel, we take $M^2:=m/(32 e)$ and 
\beq\label{eq:choiseM_k0_grand}
m:= 2 \lfloor m_0\vee x_0\rfloor \ , \quad  m_0 := 3\log\Big[\frac{8  \bar{k}^2_1}{\delta^2  n} \Big]\ ,\quad x_0:= \arg\cosh\big[1+ \frac{\bar{k}_0}{\Delta}\Big]\geq \log\big(1+\frac{\overline{k}_0}{\Delta}\big)\ . 
\eeq
Equipped with this choice of parameters, we have
\begin{align*}
\Delta a^2_{m}M^2&=  \Delta^2 \frac{m}{32e}\tanh^2\big[\frac{x_0}{m}\big]\nonumber\\
&\geq  c \Delta \frac{x^2_0}{m} \quad \text{(since $\tanh(t)\geq 0.4 t$ for any $t\in (0,1)$)}\nonumber\\
&\geq  c \Delta \big[\frac{x_0^2}{m_0}\wedge x_0\big]\quad \text{(by definition of $m$)}\nonumber\\
&\geq  c\Delta \Big[\frac{\log^2\big[1+  \frac{\overline{k}_0}{\Delta}\big]}{\log\big[ \frac{4  \overline{k}^2_1}{\delta^2  n}\big]}\wedge \log\big[1+ \frac{2\overline{k}_0}{\Delta}\big] \Big] \nonumber\\
&\geq  c\Delta \Big[\frac{\log^2\big[1+  \frac{k_0}{\Delta}\big]}{\log\big[1+ \frac{k_0}{ \sqrt{n}}\big]}\wedge \log\big[1+ \frac{k_0}{\Delta}\big] \Big]\ ,\label{eq.parmchoice}
\end{align*}
where we used in the last line that $\Delta\leq k_0$ and $k_0\geq \sqrt{n}$ and $\delta\geq 0.05$. Hence, it suffices to the prove \eqref{eq:TV_objective} to obtain  \eqref{eq:lower_bound_main2_simple}.

\paragraph{Step 5: Control on the total variation distance between $\mathbf  P_0$ and $\mathbf  P_1$.}

It remains to control the total variation distance between $\mathbf  P_0$ and $\mathbf  P_1$, relying on the fact that the $m$ first moments of $\overline{\mu}_0$ and $\overline{\mu}_1$ are matching. This is done in the following lemma. 
\begin{lem}\label{lem:upper_total_variation_distance}
The measures ${\mathbf P}_0$ and ${\mathbf P}_1$ satisfy
\beq
\|{\mathbf P}_0 - {\mathbf P}_1\|_{TV}^2 \leq  \exp\Big[4\frac{ \overline{k}_1^2}{ n} e^{-m/3}\Big] - 1\ ,
\label{eq:TV2}
\eeq
as soon as $32eM^2\leq m$.
\end{lem}
 Although we take a slightly different path, the proof of this lemma is based on the same approach as in \cite{cailow2011}.

 In the previous step, we have chosen $m$ in Equation \eqref{eq:choiseM_k0_grand} and $M$ in such a way that $\|{\mathbf P}_0 - {\mathbf P}_1\|^2_{TV}\leq \exp(\delta^2/2) -1\leq \delta$. This concludes the proof that Equation~\eqref{eq:TV_objective} holds, and therefore that Equation~\eqref{eq:lower_bound_main2_simple} holds by Equation~\eqref{eq:choiseM_k0_grand}.

 \paragraph{Step 6: Proof of \eqref{eq:lower_bound_main_simple} and \eqref{eq:lower_bound_main3_simple}}

Let us first prove \eqref{eq:lower_bound_main_simple}. As explained earlier, a similar bound can be found in \cite{baraud02}. We elaborate on Le Cam's approach. Let $M>0$ be a positive quantity that will be fixed later. We first define a suitable prior measure $\mu_1^n$ on the space $\bbB_0[\Delta,0,\Delta^{1/2}M]$.

Denote $\cS(\Delta,n)$ the collection of all subset $S$ of $[n]$ of size $\Delta$. For any $S\in \cS(\Delta,n)$, let $\mu_1^S$ denote the distribution of a vector $\theta$ where for all $i\in S$, $\theta_i\sim \frac{1}{2}\delta_{M}+ \frac{1}{2}\delta_{-M}$ and for all $i\notin S$, $\theta_i$ follows a Dirac distribution at zero. As a consequence, $\mu_1^n:= \binom{n}{\Delta}^{-1}\sum_{S\in \cS(\Delta,n)}\mu_0^S$ is a probability distribution over $\bbB_0[\Delta,0,\Delta^{1/2}M]$. Finally, we denote $\mathbf{P}_1:=\int \P_{\theta}\mu_1^{n}(d\theta)$.
Given a test $T$, its risk  \eqref{eq:risk} is bounded 
 \beqn 
 R[T;0,\Delta, \Delta^{1/2} M]&=& \P_{0}[T=1] +  \sup_{\theta \in \bbB_0[\Delta,0,\Delta^{1/2} M]}\P_{\theta}[T=0]
 \geq  \P_{0}[T=1] -   \int \P_{\theta}[T=0]\mu_1^{n}(d\theta)\\
 &\geq & 1 - \|\P_0 - \mathbf{P}_1\|_{TV}\ . 
 \eeqn 
 As a consequence, the minimax separation distance  $\rho^*_{\gamma}[n,0,\Delta]$ is larger than $\Delta^{1/2}M$, as soon as 
 $\|\P_0 - \mathbf{P}_1\|_{TV}\leq \delta $, where $\delta:=1  -\gamma\geq 0.5$.
 By Cauchy-Schwarz inequality,
 \[
 \|\P_0 - \mathbf{P}_1\|^2_{TV}\leq \int  \Big[\frac{d\mathbf{P}_1}{d\P_0}\Big]^2d\P_0- 1 = \binom{n}{\Delta}^{-2}\sum_{S,S'} \E_{0}\Big[\int \frac{d\P_\theta}{d\P_0} \mu^{S}_1(d\theta)\int \frac{d \P_\theta}{d\P_0} \mu^{S'}_1(d\theta)\Big].
 \]
For fixed $S$ and $S'$, the expectation 
 \beqn
   \E_{0}\Big[ \int \frac{d\P_\theta}{d\P_0} \mu^{S}_1(d\theta)\int \frac{d \P_\theta}{d\P_0} \mu^{S'}_1(d\theta)\Big]&= &\prod_{i\in S\cap S'}\E\left[e^{-M^2} \cosh^2(MY_i)\right]= \cosh(M^2)^{|S\cap S'|}.
 \eeqn 
When $|S|$ and $|S'|$ are distributed uniformly in $\cS(\Delta,n)$, the size $X:=|S\cap S'|$ follows an hypergeometric distribution with parameters $n$, $\Delta$ and $\Delta/n$. We know from \cite[p.173]{aldous85} that $X$ is distributed as the random variable $\E[Z|\cB_n]$ where $Z$ is a Binomial random variable with parameters $(\Delta,\Delta/n)$  and $\cB_n$ is some $\sigma$-algebra. Applying  Jensen inequality, we obtain
 \[
 \|\P_0 - \mathbf{P}_1\|^2_{TV}+1 \leq \E[\cosh(M^2)^{Z}]= \Big[1+ \frac{\Delta}{n} [\cosh(M^2)-1]\Big]^{\Delta}\leq \exp\big[\frac{\Delta^2}{n}(\cosh(M^2)-1)\big].
 \]
Taking 
\[M^2 := \arg\cosh\big[1+ \frac{\delta^2  n}{2\Delta^2}\big]\geq \log\Big[1 + \sqrt{\frac{\delta^2 n}{\Delta^2}}\Big]\geq \log\Big[1 + \frac{ \sqrt{n}}{4\Delta}\Big],\]
we conclude that $ \|\P_0 - \mathbf{P}_1\|^2_{TV}\leq e^{\delta^2/2 }-1 \leq \delta^2$, implying that 
\[
 \rho^{*2}_{\gamma}[n,0,\Delta]\geq \Delta \log\Big[1 + \frac{ \sqrt{n}}{4\Delta}\Big]\ .
\]
We have proved \eqref{eq:lower_bound_main_simple}.
 
 \bigskip 
 
 Finally, we turn to \eqref{eq:lower_bound_main3_simple}. Again, $M>0$ is a positive quantity that will be fixed later. Define $\theta_1$ as the constant vector whose components are all equal to $-M$. For any $S\in \cS(\Delta,n)$, define $\theta_0^S$ the vector whose coordinates in $S$ are equal to zero and whose remaining components are equal to $-M$. Let $\mu_0^n:= \binom{n}{\Delta}^{-1}\sum_{S}\delta_{\theta_0^S}$. Finally, we denote $\mathbf{P}_0:=\int \P_{\theta}\mu_0^{n}(d\theta)$. For any test $T$,  $R[T,n-\Delta, n, \Delta^{1/2} M]\geq  1 - \|\P_{\theta_1} - \mathbf{P}_0\|_{TV}$ so that $\rho^*_{\gamma}[n,n-\Delta,\Delta]\geq \Delta^{1/2}M$ when  
 $\|\P_{\theta_1} - \mathbf{P}_0\|_{TV}\leq \delta $. Arguing  as above, we get
 \beqn 
 \|\P_{\theta_1} - \mathbf{P}_0\|_{TV}^2&\leq& \int  \Big[\frac{d\mathbf{P}_0}{d\P_{\theta_1}}\Big]^2d\P_0- 1= \binom{n}{\Delta}^{-2}\sum_{S,S'} e^{M^2 |S\cap S'|} -1 \\ 
 &\leq& \Big[1+ \frac{\Delta }{n}\big(e^{M^2}-1\big)\Big]^\Delta - 1\leq \exp\big[\frac{\Delta^2}{n}(e^{M^2}-1)\big] - 1.
 \eeqn 
 Choosing $M^2= \log\big[1  + \frac{\delta^2 n}{2\Delta^2}\big]\geq  \log\big[1  + \frac{n}{8\Delta^2}\big]$, we prove \eqref{eq:lower_bound_main3_simple}.

 \begin{proof}[Proof of Proposition \ref{prp:lower_reduced}]
 To derive Theorem \ref{prp:lower}, we only need to deduce from (\ref{eq:lower_bound_main_simple}--\ref{eq:lower_bound_main3_simple})  the lower bounds in the regime  (i) $k_0\leq \sqrt{n}$, (ii) $k_0> \sqrt{n}$ and $\Delta>k_0$ and (iii) $k_0>\sqrt{n}$ and $\Delta\leq 32\sqrt{(n-k_0)\wedge  k_0}$.

 \medskip 
 
 \noindent
(i) $k_0\leq \sqrt{n}$. We combine \eqref{eq:subproblem} and \eqref{eq:lower_bound_main_simple} to obtain
\[
 \rho_{\gamma}^{*2}[n,k_0,\Delta]\geq \rho_{\gamma}^{*2}[n-k_0,0,\Delta] \geq \Delta \log\Big[1+ \frac{ (n-k_0)^{1/2}}{4\Delta} \Big]\geq \Delta \log\Big[1+ \frac{ \sqrt{n}}{8\Delta} \Big]\ .
\]

 \medskip 

 \noindent
(ii) $k_0> \sqrt{n}$ and $\Delta>k_0$. We gather \eqref{eq:monotonic} and \eqref{eq:lower_bound_main2_simple} to obtain
\[
 \rho_{\gamma}^{*2}[n,k_0,\Delta]\geq \rho_{\gamma}^{*2}[n,k_0,k_0]\geq c k_0 \frac{\log^2(2)}{\log\big[1+ \frac{k_0}{\sqrt{n}}\big]}\ .
\]

 \medskip 

 \noindent
(iii) $k_0>\sqrt{n}$ and $\Delta\leq 32\sqrt{(n-k_0)\wedge  k_0}$. We shall consider two subcases $\Delta \leq n^{1/3}$ and $\Delta > n^{1/3}$. For $\Delta \leq n^{1/3}$ and $k_0\leq n/2$, we apply \eqref{eq:subproblem} together with $\sqrt{n}/\Delta\geq n^{1/6}$.
\beqn 
 \rho_{\gamma}^{*2}[n,k_0,\Delta]&\geq& \rho_{\gamma}^{*2}[n-k_0,0,\Delta]\geq \Delta \log\Big[1+ \frac{ \sqrt{n}}{8\Delta} \Big]\geq c\Delta \log\Big[1+ \frac{ k_0}{\Delta} \Big]\\ &\geq& c\Delta \left[\frac{\log^2\big[1+  \frac{k_0}{\Delta}\big]}{\log\big[1+ \frac{k_0}{\sqrt{n}}\big]}\wedge \log\big[1+ \frac{k_0}{\Delta}\big]\right]\ .
\eeqn 
For $1\leq \Delta \leq n^{1/3}$ and $k_0> n/2$, we use \eqref{eq:reduction_sample} and \eqref{eq:lower_bound_main3_simple}.
\beqn 
 \rho_{\gamma}^{*2}[n,k_0,\Delta]&\geq& \rho_{\gamma}^{*2}[k_0+\Delta,k_0,\Delta]\geq \Delta \log\Big[1  + \frac{k_0+\Delta}{8\Delta^2}\Big] \geq c \Delta \log\big[1+ n\big]\\
 &\geq& c\Delta \log\Big[1+ \frac{ k_0}{\Delta} \Big]\\ &\geq& c\Delta \left[\frac{\log^2\big[1+  \frac{k_0}{\Delta}\big]}{\log\big[1+ \frac{k_0}{\sqrt{n}}\big]}\wedge \log\big[1+ \frac{k_0}{\Delta}\big]\right]\ .
\eeqn 
For $\Delta > n^{1/3}$ and $k_0\leq n/2$, we  define $k'_0:= \lfloor \Delta^2/(32)^2\rfloor $ and $n':= n-k_0+k'_0$. Consequently, $\Delta > 32\sqrt{k'_0\wedge (n'-k'_0)}$ and $k'_0\geq \sqrt{n'}$   for $n$ large enough. Then, \eqref{eq:subproblem} together with \eqref{eq:lower_bound_main2_simple} gives us 
\beqn
 \rho_{\gamma}^{*2}[n,k_0,\Delta]&\geq & \rho_{\gamma}^{*2}[n-k_0+k'_0,k'_0,\Delta]\geq c\Delta \left[\frac{\log^2\big[1+  \frac{k'_0}{\Delta}\big]}{\log\big[1+ \frac{k'_0}{\sqrt{n'}}\big]}\wedge \log\big[1+ \frac{k'_0}{\Delta}\big]\right]\\
 &\geq &  c\Delta \log(n)\geq c'\Delta \left[\frac{\log^2\big[1+  \frac{k_0}{\Delta}\big]}{\log\big[1+ \frac{k_0}{\sqrt{n}}\big]}\wedge \log\big[1+ \frac{k_0}{\Delta}\big]\right]\ .
\eeqn
The last case $\Delta>n^{1/3}$ and $k_0>n/2$ is handled similarly.

\end{proof}

\begin{proof}[Proof of Lemma \ref{lem:nemirovski2}]
 For the sake of clarity, we simply write $a$ for $a_{m}$ in this proof.
 Let $\cP^{\text{sym}}_{m}$ denote the vector space of \textit{symmetric} polynomials of degree smaller or equal to $m$. 
 Define the linear function $g$ on $\cP^{\text{sym}}_{m}$ by $g: P\mapsto P(0)$. We endow $\cP^{\text{sym}}_{m}$ with the uniform norm $\|.\|_{[a,1]}$ on $[a,1]$. Let $\nu^*$ be the norm of this linear functional. By Hahn-Banach theorem, we can extend this functional from $\cP^{\text{sym}}_{m}$ to the entire space $C[a,1]$ of continuous functions on $[a,1]$ without increasing the norm of the functional. By Riesz-Markov theorem, this linear functional can be represented as a measure $\nu$ on $[a,1]$. As a consequence, $\int P(t)\nu(dt)= P(0)$ for all $P\in \cP^{\text{sym}}_m$ and the total variation $\|\nu\|_{TV}:= \int|\nu(dt)|$ equals $\nu^*$. 
 
 We extend $\nu$ to a symmetric measure $\bar{\nu}^{\text{sym}}$ on $[-1,-a]\cup [a,1]$ such that $\bar{\nu}^{\text{sym}}(\cA)= (\nu(\cA)+ \nu(-\cA))/2$. Let $\bar{\nu}^{\text{sym}}_+$ and $\bar{\nu}^{\text{sym}}_-$ respectively denote the positive and negative part of $\bar{\nu}^{\text{sym}}$ so that  $\bar{\nu}^{\text{sym}}=\bar{\nu}^{\text{sym}}_+- \bar{\nu}^{\text{sym}}_-$. Finally, we define 
 \[\nu_1:= \frac{2\bar{\nu}^{\text{sym}}_+}{1+\nu^*}\ , \quad \quad \nu_0:= \frac{2\bar{\nu}^{\text{sym}}_-}{1+\nu^*}\ .\]
 For any even integer $q\leq m$, 
 \[
 \int t^q (\nu_1(dt)-\nu_0(dt))= \frac{2}{1+\nu^*}\int t^q \bar{\nu}^{\text{sym}}(dt)= \frac{2}{1+\nu^*}\int_{a}^{1} t^q \nu(dt)=0\ ,
 \]
where we used the symmetry of $t\mapsto t^{q}$ and the definition of the functional $g$ in the last equality. For any odd integer $q$
\[
 \int t^q (\nu_1(dt)-\nu_0(dt))= \frac{2}{1+\nu^*}\int t^q \bar{\nu}^{\text{sym}}(dt)= 0\ ,
\]
 because $q\mapsto t^{q}$ is antisymmetric. As a consequence, $\nu_0$ and $\nu_1$ satisfy the property \eqref{eq:condition_momentq}. As for the measure of $\nu_0$ and $\nu_1$, we have 
 \[\int (\nu_1(dt)-\nu_0(dt))=  \frac{2}{1+\nu^*}\int  \bar{\nu}^{\text{sym}}(dt)= \frac{2}{1+\nu^*}\]
 by definition of $g$. Also 
$$\int (\nu_1(dt)+\nu_0(dt))=  \frac{2\nu^*}{1+\nu^*},$$
by definition of $\nu^*$. As a consequence, $\int \nu_1(dt)= 1$ and $\int \nu_0(dt)= \frac{\nu^*-1}{\nu^*+1}$. To conclude the proof of \eqref{eq:condition_moment0}, we only need to show that 
 \beq\label{eq:objective_nu*}
 \nu^*= \frac{1+p}{1-p}\ .
 \eeq
 
 Denote $\cP_{m/2}$ the space of polynomials of degrees smaller or equal to $m/2$. We endow it with the supremum norm $\|.\|_{[a^2,1]}$ on $[a^2,1]$. Then the mapping $\phi: P(x) \mapsto  P(x^2)$ is an isometry from $(\cP_{m/2},\|.\|_{[a^2,1]})$ to $(\cP_{m}^{\text{sym}},\|.\|_{[a,1]})$. Also, for $P\in \cP_{m}^{\text{sym}}$, $P(0)=g(P)= [\phi^{-1}(P)](0)$. As a consequence, $\nu^*$ is characterized as 
\[
 \nu^*= \sup_{P\in \cP_{m/2}, \, \|P\|_{[a^2,1]}\leq 1 }P(0)
 \]
 Define the linear function $h: x\mapsto \frac{2}{1-a^2}t-  \frac{1+a^2}{1-a^2}$ mapping $[a^2,1]$ to $[-1,1]$. By substitution, we deduce that 
 \[
  \nu^*= \sup_{P\in \cP_{m/2}, \, \|P\|_{[-1,1]}\leq 1 }P\left(-  \tfrac{1+a^2}{1-a^2}\right)= \sup_{P\in \cP_{m/2}, \, \|P\|_{[-1,1]}\leq 1 }P\left(  \tfrac{1+a^2}{1-a^2}\right)\ ,
 \]
where we used the symmetry of the problem in the second identity. By Chebychev's Theorem, this supremum is achieved by the Chebychev polynomial of order $m/2$. Hence, we get
\[
 \nu^*= \cosh\Big[\frac{m}{2}\arg\cosh\big(\tfrac{1+a^2}{1-a^2}\big)\Big]\ . 
\]
Since $\frac{1+\tanh^2(x)}{1+\cosh^2(x)}= \cosh^2(x)+\sinh^2(x)= \cosh(2x)$, we obtain $\nu^*= \tfrac{1+p}{1-p}$, which concludes the proof.

\end{proof}

\begin{proof}[Proof of Lemma \ref{lem:upper_total_variation_distance}]
By Cauchy-Schwarz inequality, we relate the total variation distance to the $\chi^2$ distance.
$$\|{\mathbf  P}_0 - {\mathbf  P}_1\|_{TV}^2 \leq \int \frac{(d{\mathbf  P}_1 - d{\mathbf  P}_0)^2}{d{\mathbf  P}_0}.$$
Since ${\mathbf  P}_0:= \otimes_{i} {\mathbf P}_{0,i} $ and ${\mathbf  P}_1=\otimes_{i} {\mathbf P}_{1,i} $ are $n$-dimensional  product measures. Developing the likelihood ratio, we arrive at
$$ \int \frac{(d{\mathbf  P}_0  - d{\mathbf  P}_1 )^2}{d{\mathbf  P}_0} =  \int \frac{(d{\mathbf  P}_1)^2}{d{\mathbf  P}_0} - 1 = \Big(\int \frac{(d{\mathbf P}_{1,1})^2}{d{\mathbf P}_{0,1}}\Big)^{ n} -1 = \Big(1 + \int \frac{(d{\mathbf P}_{1,1}-d{\mathbf P}_{0,1})^2}{d{\mathbf P}_{0,1}}\Big)^{ n} -1.$$
So the two previous equations imply that
\beq\label{eq:TV}
\|{\mathbf  P}_0 - {\mathbf  P}_1\|_{TV}^2 \leq \Big(1 + \int \frac{(d{\mathbf P}_{1,1}-d{\mathbf P}_{0,1})^2}{d{\mathbf P}_{0,1}}\Big)^{ n} -1.
\eeq
We now focus on the $\chi^2$ distance $\int (d{\mathbf P}_{1,1}-d{\mathbf P}_{0,1})^2/d{\mathbf P}_{0,1}$. 
Recall that $k_0\geq \sqrt{n}$ and $m\geq 2$. We have by Equation~\eqref{eq:mom} that
\begin{eqnarray}
\lefteqn{\frac{e^{y^2}(d{\mathbf P}_{1,1}(y) - d{\mathbf P}_{0,1}(y))^2}{(dy)^2}} \nonumber&&\\ &=& \frac{1}{2\pi} \Big( \int \exp(yu - u^2/2) \overline{\mu}_1(du) -  \int \exp(yu - u^2/2) \overline{\mu}_0(du)\Big)^2\nonumber\\
&=&\frac{1}{2\pi} \Big( \int \sum_{l=0}^{\infty} \frac{(yu - u^2/2)^l}{ l!} \overline{\mu}_1(du) -  \int \sum_{l=0}^{\infty} \frac{(yu - u^2/2)^l}{ l!} \overline{\mu}_0(du)\Big)^2\nonumber\\
&=&\frac{1}{2\pi} \Big( \sum_{l\geq m/2 +1} \int  \frac{(yu - u^2/2)^l}{ l!} (\overline{\mu}_1(du) -  \overline{\mu}_0(du))\Big)^2\quad \text{ by \eqref{eq:mom}}\nonumber\\
&\leq&\frac{1}{2\pi} \Big( \frac{ 2\overline{k}_1}{ n}\sum_{l\geq m/2 +1}  \frac{2^{l-1}M^l|y|^l+ M^{2l}/2}{ l!} \Big)^2 \quad \text{as $(a+b)^l\leq 2^{l-1}(a^l+b^l)$}\nonumber\\
&\leq &\frac{ \overline{k}_1^2}{2\pi n^2} \sum_{l\geq m/2 +1}  \Big(\frac{2^lM^l|y|^l+ M^{2l}}{ l!} \Big)^2\nonumber\\
&\leq &\frac{ \overline{k}_1^2}{\pi  n^2} \sum_{l\geq m/2 +1}  \frac{(2M)^{2l}|y|^{2l}+ 2M^{4l}}{ l!^2},
\label{eq:noundmo}
\end{eqnarray}
where we used again $(a+b)^2\leq 2(a^2+b^2)$. Since the function $x\mapsto \exp(-x)$ is convex, we can lower bound the density $d{\mathbf P}_{0,1} (y)/dy$ as follows
\beqn 
\frac{d{\mathbf P}_{0,1} (y)}{dy}&  =& \frac{1}{\sqrt{2\pi}}  \int_{-M}^{M} \exp(-(y-u)^2/2) \overline{\mu}_0(du) \\
& \geq  &\frac{1}{\sqrt{2\pi}}   \exp\left[ - \int_{-M}^{M} \frac{(y-u)^2}{2} \overline{\mu}_0(du) \right]\\
&\geq & \frac{e^{-y^2/2}}{\sqrt{2\pi}}e^{-M^2/2}\ ,
\eeqn 
where we used in the last line the symmetry of $\overline{\mu}_0$ and that its support lies in  $[-M;M]$.
Plugging the last inequality into Equation~\eqref{eq:noundmo}, we are equipped to bound the $\chi^2$ distance between ${\mathbf P}_{0,1}$ and ${\mathbf P}_{1,1}$. 
\begin{align*}
 \int \frac{(d{\mathbf P}_{0,1}  - d{\mathbf P}_{1,1})^2}{d{\mathbf P}_{0,1}}&\leq e^{M^2/2}\sqrt{2\pi} \int \frac{(d{\mathbf P}_{0,1}(dy)  - d{\mathbf P}_{1,1}(dy))^2}{\exp(-y^2/2)} dy\\
&\leq 2e^{M^2/2} \Big(\frac{ \overline{k}_1}{ n}\Big)^2  \sum_{l\geq m/2+1}  \int \frac{e^{-y^2/2}}{\sqrt{2\pi}}\cdot \frac{(2M)^{2l}|y|^{2l}+ 2M^{4l}}{ l!^2}dy\\
&\leq 2e^{M^2/2}\Big(\frac{ \overline{k}_1}{ n}\Big)^2  \sum_{l\geq m/2+1}  \frac{(2M)^{2l}(2l-1)!!}{l!^2}+ \frac{2M^{4l}}{l!^2}\\
&\leq 2e^{M^2/2}\Big(\frac{ \overline{k}_1}{ n}\Big)^2  \sum_{l\geq m/2 +1}  \Big(\frac{M^{2l}8^l}{l!}+ \frac{2M^{4l}}{l!^2}\Big)\ ,
\end{align*}
where we used the expression of $l$-th moments of a normal distribution in the second line.
Now, assume that $32eM^2/m\leq 1$. Since $l!\geq (l/e)^l$, we have
\begin{align*}
 \int \frac{(d{\mathbf P}_{0,1}  - d{\mathbf P}_{1,1})^2}{d{\mathbf P}_{0,1}}
&\leq 2e^{M^2/2}\Big(\frac{ \overline{k}_1}{ n}\Big)^2  \sum_{l\geq m/2 +1}  \Big(\frac{8eM^{2}}{l}\Big)^l+ \Big(\frac{2e^2M^{4}}{l^2}\Big)^l\\
&\leq 2e^{M^2/2}\Big(\frac{ \overline{k}_1}{ n}\Big)^2  \sum_{l\geq m/2 +1}  \Big(\frac{16eM^{2}}{m}\Big)^l+ \Big(\frac{8e^2M^{4}}{m^2}\Big)^l\\
&\leq 4 e^{M^2/2} \Big(\frac{ \overline{k}_1}{ n}\Big)^2  2^{-m/2}\\
& \leq 4 \Big(\frac{ \overline{k}_1}{ n}\Big)^2 e^{-m/3}\ ,
\end{align*}
where we use in the two last line the $m\geq 2$ and $32eM^2/m\leq 1$.
Coming back to \eqref{eq:TV}, we conclude that, as soon as $32eM^2\leq m$,
\[
\|{\mathbf P}_0 - {\mathbf P}_1\|_{TV}^2 \leq \big(1 + \int \frac{(d{\mathbf P}_{0,1}  - d{\mathbf P}_{1,1})^2}{d{\mathbf P}_{0,1}}\big)^{ n}-1\leq  \exp\Big[ 4\frac{ \overline{k}_1^2}{ n} e^{-m/3}\Big] - 1\ .\]

\end{proof}

\subsection{Proofs of the testing upper bounds with known variance}\label{sec:proof_KVUB}

\subsubsection{Analysis of $T^{HC}_{\alpha,k_0}$}

We will in fact prove a sharper result than Proposition \ref{cor:T_HC}. To study the rejection regions of this test, additional notation is needed. Given  $\beta\in (0,1)$, let
\beq \label{eq:definition_k+}
q^{HC}_{+}:= 11 \log\left(\frac{8}{\alpha \beta}\right)+ 6 \log\Big(\log\Big(4\frac{n}{\alpha}\Big)\Big)
\eeq
For any integer $q\in [1,n-k_0]$, define  $t_{q}$ 
\beq\label{eq:definition_t}
t_{q}:= \Big \lceil \sqrt{2 \big(6+ \log\big(\tfrac{n}{q^2}\big)_+ + \log\log(\tfrac{18}{\alpha\beta}) \big)}\Big\rceil \ .
\eeq
and 
\beq\label{eq:definition_mu}
 \mu^{HC}_{q} := 
 \left\{ \begin{array}{cc} 
t^{HC}_{*,\alpha}+ \sqrt{2\log[(k_0+1)/\beta]} & \text{ if }q< q^{HC}_{+}\text{ or }t_q\geq  t^{HC}_{*,\alpha}          
\\         
t_{q}+ \sqrt{2\big(3+ \log(\tfrac{k_0}{q})_+  + \log(\tfrac{32\log(2/\beta)}{q})_+\big)  }& \text{ if }q\geq q^{HC}_{+}\text{ and }t_q<  t^{HC}_{*,\alpha} \ .
\end{array}
\right.
\eeq

\begin{prp}\label{prp:T_HC_levelpower}
The type I error  probability of  $T^{HC}_{\alpha,k_0}$ is smaller or equal to $\alpha$, that is 
$\P_{\theta}[T^{HC}_{\alpha,k_0}=1]\leq \alpha$ for all $\theta\in \mathbb{B}_0(k_0)$.
 Besides, any $\theta\in \mathbb{R}^n$ such that
$$|\theta_{(k_0+q)}|\geq \mu^{HC}_{q}, ~~~\text{for some } q\in [1,n-k_0]\ ,$$
belongs to the high probability rejection region of $T^{HC}_{\alpha,k_0}$, that is $\P_{\theta}[T^{HC}_{\alpha,k_0}=1]\geq 1-\beta$.
\end{prp}

Proposition \ref{cor:T_HC} is a straightforward corollary of Proposition \ref{prp:T_HC_levelpower}.

\begin{proof}[Proof of Proposition \ref{prp:T_HC_levelpower}]

We first focus on the type I error and then consider the power of the procedure.

\noindent
{\bf Level of the Test}.
Consider any $\theta\in \mathbb{B}_0(k_0)$. For any $t>0$,  $N_t-k_0$ is stochastically smaller than a Binomial distribution with parameters $n-k_0$ and $2\Phi(t)$. Since $\Phi(t)\leq \exp(-t^2/2)$, we obtain by a simple union bound that
\[
 \P_{\theta}[N_{t^{HC}_{*,\alpha}}\geq k_0+1]\leq 2(n-k_0)\exp\Big[-\frac{(t^{HC}_{*,\alpha})^2}{2}\Big]\leq \alpha /2\ .
\]
Also, by Bernstein inequality, we have
\[
 \P_{\theta}[N_t\geq k_0 +2(n-k_0)\Phi(t)+ u^{HC}_{t,\alpha}]\leq \frac{3\alpha}{\pi^2 t^2}\ .
\]
Applying again an union bound yields 
\[
 \P_{\theta}[T^{HC}_{\alpha,k_0}=1]\leq \alpha \sum_{t=1}^{\infty} \frac{3}{\pi^2 t^2 }+ \alpha/2\leq \alpha \ .
\]

\bigskip

\noindent 
  {\bf Power of the test}. To ease the notation, we respectively write $\mu_q$, $u_{q}$ and $q_+$ for $\mu^{HC}_{q}$, $u_{t_q,k_0}^{HC}$ and $q_+^{HC}$.
Let $\theta$ be any vector such that $|\theta_{(k_0+q)}|\geq \mu$.
The proof is divided into two different cases depending on the value of $q$.

\bigskip 

\noindent 
  {\bf Case 1}: Assume that  $q< q_+$. In that situation, we focus on $N_{t^{HC}_{*,\alpha}}$:
 \[
 \P_{\theta}[T^{HC}_{\alpha,k_0} = 0] \leq \P_{\theta}[N_{t^{HC}_{*,\alpha}}\leq k_0].
 \]
 Restricting ourselves to the $k_0+1$ largest absolute values of $\theta$, we get
 \[
 \P_{\theta}[T^{HC}_{\alpha,k_0} = 0] \leq \sum_{i=1}^{k_0+1} \Phi[|\theta|_{(i)}-t^{HC}_{*,\alpha}]\leq (k_0+1)\Phi[\mu_q-t^{HC}_{*,\alpha}]\ ,
 \]
 which is smaller than $\beta$, since by definition \eqref{eq:definition_mu}, we have $\mu_q\geq t^{HC}_{*,\alpha}+ \sqrt{2\log[(k_0+1)/\beta]}$.

 \bigskip 
 
 \noindent 
 {\bf Case 2}: We now assume that $q\geq q_+$. By definition \eqref{eq:definition_k+} and \eqref{eq:definition_t} of $q_+$ and $t_q$, this enforces that $t_q< t^{HC}_{*,\alpha}$. Observe that $N_{t_q}$ is stochastically larger than the sum of a Binomial random variable with parameters $(k_0+q)$ and $(1-\Phi(\mu_q-t_q))$ and a Binomial random variable with parameters $(n-k_0-q)$ and $2\Phi(t_q)$. Applying Bernstein inequality to these two random variables, we derive that,
 with probability larger than $1-\beta$, 
\[
 N_{t_q} \geq (k_0+q)(1-\Phi(\mu_q-t_q))+ (n-k_0-q)2\Phi(t_q)-  v_{q}\ ,
\]
where 
\[v_{q}:=  \sqrt{2(k_0+q)\Phi(\mu_q-t_q)\log(2/\beta)}+ 2\sqrt{n\Phi(t_q)\log(2/\beta)} +  \frac{4}{3}\log(2/\beta)\ .\]
As a consequence of \eqref{eq:definition_rejection_HC}, we have $\P_{\theta}[T^{HC}_{\alpha,k_0}=1 ]\geq 1-\beta$ as soon as
\[
 q(1-\Phi(\mu_q-t_q)-2\Phi(t_q))\geq k_0 \Phi(\mu_q-t_q)+ u_{q}+ v_{q}\ .
\]
Since $t_q\geq 1$ and $\mu_q-t_q\geq 2$, we have $2\Phi(t_q)\leq 0.4$ and $\Phi(\mu_q-t_q)\leq 0.1$. As a consequence,  the above inequality holds when the four following conditions are satisfied
 \begin{eqnarray} \label{eq:cond1}
 q &\geq& 8k_0 \Phi(\mu_q-t_q) \ , \\ \label{eq:cond2}
 q &\geq & 8 \sqrt{2(k_0+q)\Phi(\mu_q-t_q)\log(2/\beta)} \ ,\\
 q &\geq  & 16\sqrt{n\Phi(t_q)}\left[\sqrt{\log\left(\frac{t_q^2\pi^2}{3\alpha}\right)} + \sqrt{\log\left(\frac{2}{\beta}\right)}\right] \ , \label{eq:cond3}\\
 q &\geq & \frac{32}{3}\log\left[\frac{2t^{HC}_{*,\alpha}\pi }{\sqrt{3\alpha}\beta}\right]\ . \nonumber
 \end{eqnarray}
The last condition is a consequence of the condition $q \geq q_+$. To finish the proof, it suffices to show  that \eqref{eq:cond1}, \eqref{eq:cond2}, and \eqref{eq:cond3} are ensured by  our choice of $\mu_q$ and $t_q$.
Inequalities \eqref{eq:cond1} and \eqref{eq:cond2}
hold when 
\[
 \Phi(\mu_q-t_q)\leq \frac{1}{8}\left[\frac{q}{k_0}\wedge 1\right]\left[\frac{q}{32\log(2/\beta)}\wedge 1\right]\ .
 \]
 In view of the definition \eqref{eq:definition_mu} of $\mu_q$,  this last inequality is true.
  Since $(\sqrt{x+y}+\sqrt{z})^2\leq (2+x)(1+y+z)$, Condition \eqref{eq:cond3} holds if
 \[
 \log(et_q)\Phi(t_q)\leq \frac{q^2}{2^9n\log(\frac{2e\pi^2}{3\alpha\beta})}.
\]
Since $\Phi(t_q)\leq \tfrac{e^{-t_q^2/2}}{t_q\sqrt{2\pi}}$ and $t- \log(t)-1>0$ for any $t>0$, we only need that
\[
 t_q^2 \geq  2 \log\left[\frac{2^9 n}{\sqrt{2\pi}q^2}\right]_{+} + 2\log\log(\frac{2e\pi^2}{3\alpha\beta})\ ,
\]
which is a straightforward consequence of  our choice \eqref{eq:definition_t} of $t_q$. This concludes the proof.

\end{proof}

 \subsubsection{Analysis of $T^{B}_{\alpha,k_0}$}

To properly characterize the power of $T^{B}_{\alpha,k_0}$ additional notation is needed. Let 
\beq\label{eq:definition_bk0}
 v_{k_0}^B:= \frac{k_0\sqrt{8e}}{\sqrt{\log\big(1+ \frac{k_0^2}{n}\big)}} \big[\sqrt{\log(2/\alpha)}+ \sqrt{\log(2/\beta)}\big]\ .
\eeq
For any integer $q>4v_{k_0}^B$ define 
\beq\label{eq:definition_muB}
 \mu^{B}_{q}:= \frac{16}{s_{k_0}}\sqrt{\frac{k_0+v_{k_0}^B}{q}} \ .
 \eeq

\begin{prp}\label{prp:test_log}
The type I error  probability of  $T^{B}_{\alpha,k_0}$ is smaller or equal to $\alpha$, that is 
$\P_{\theta}[T^{B}_{\alpha,k_0}=1]\leq \alpha$ for all $\theta\in \mathbb{B}_0(k_0)$.
 Besides, any $\theta\in \mathbb{R}^n$ such that any of the two following conditions is fulfilled
\begin{eqnarray}
 \label{eq:separation_test_log}
|\theta_{(k_0+q)}|&\geq& \mu^{B}_{q}\ ,\quad \quad \text{ for some } q > 22v_{k_0}^B \ ,\\ 
\label{eq:separation_test_log_l2}
\sum_{i=1}^n \Big[|s_{k_0} \theta_i|^2\wedge 4\Big]&\geq& 50(k_0+v_{k_0}^B)\ ,
\end{eqnarray}
belongs to the high probability rejection region of $T^{B}_{\alpha,k_0}$, that is $\P_{\theta}[T^{B}_{\alpha,k_0}=1]\geq 1-\beta$.

 \end{prp}

\begin{proof}[Proof of Proposition \ref{cor:TB_power}]
 Proposition \ref{cor:TB_power} is a simple consequence of Proposition \ref{prp:test_log} based on the fact that $v_{k_0}^B\geq c_{\alpha,\beta}\sqrt{n}$ and $s_{k_0}=1$ for $k_0< \sqrt{n}$ whereas $s_{k_0}\geq c \log(1+k_0/\sqrt{n})$ for $k_0\geq \sqrt{n}$. 
\end{proof}

 \begin{proof}[Proof of Proposition \ref{prp:test_log}]
 To ease the notation, we respectively write $v_{k_0}$, $\mu_q$ and $s$ for $v_{k_0}^B$, $\mu_q^B$ and $s_{k_0}$. The proof is divided into two main lemmas. First, we prove that $Z(s)$ concentrates well around its expectation using the Gaussian concentration Theorem. 
 \begin{lem}\label{lem:concentration_Z}
For any $x>0$ and any $\theta\in \mathbb{R}^n$ and any $s>0$,  it holds 
 \beq\label{eq:concentration_Z}
 \P_{\theta}\left[\big| Z(s) - \E_{\theta}[Z(s) ]\big|\geq  \frac{e^{s^2/2}}{s} \sqrt{8xn}\right]\leq 2e^{-x} \ , 
\eeq  
 \end{lem}
 Note that Hoeffding's inequality allows to recover a similar inequality with a less stringent dependency with respect to $s$.

In view of the above deviation inequality, it suffices to control the expectations of $\E_{\theta}[Z(s)]$  to derive the type I and type II error  probabilities.  When $X\sim \cN(\mu,1)$, the expectation of $\kappa_{s}(X)$ satisfies 
\[
 \E[\kappa_{s}(X)]= \int_{-1}^1 (1-|\xi|) \cos(s\xi \mu)d\xi = 2\int_{0}^{1}(1-\xi)\cos(s \xi \mu)d\xi= 2\frac{1-\cos(s\mu)}{(s\mu)^2}  \ .
\]
Define the function $g$ by $g(0)=0$ and  $g(x):= 1- 2\frac{1-\cos(x)}{x^2}$ for $x\neq 0$. We have  
\beq \label{eq:decomposition_expectation_Z}
 \E_{\theta}[Z(s)]= \sum_{i=1}^n g(s\theta_i)\ .
\eeq
Since $\cos(x)\geq 1-x^2/2$, $g$ takes values in $[0,1]$.

\bigskip
\noindent 
  {\bf Level of the Test}.  Consider any $\theta\in \mathbb{B}_0(k_0)$. Since $g$ takes values in $[0,1]$ and since $g(0)=0$, we have  $\E_{\theta}[Z(s)] \leq k_0$. Gathering this bound with the deviation inequality \eqref{eq:concentration_Z} and the definition of $s$, we conclude that 
\[
\P_{\theta}\left[Z(s)\geq k_0 + \frac{e^{s^2/2}}{s} \sqrt{8n\log(2/\alpha)}\right]\leq \alpha\ . 
\]
In view the definition \eqref{eq:definition_Tb_alpha} of $T_B$, this implies  that $\P_{\theta}[T^B_{k_0}=1]\leq \alpha$.

\bigskip

\noindent 
  {\bf Power of the test}. Turning to  the type II error, we first consider any vector $\theta$ satisfying Condition \eqref{eq:separation_test_log}.
Applying again the deviation inequality \eqref{eq:concentration_Z} together with \eqref{eq:upper_es}, we have 
\[\P_{\theta}\left[Z(s)\leq \E_{\theta}(Z(s)) -  \frac{e^{s^2/2}}{s} \sqrt{8n\log(2/\alpha)}\right]\leq \beta.\]
Observe that
\beq\label{eq:upper_es} 
\frac{e^{s^2/2}}{s}\sqrt{n}= \frac{\sqrt{e}k_0}{\sqrt{\log(ek_0^2/n)}}\mathbf{1}_{k_0\geq \sqrt{n}} + \sqrt{ne}\mathbf{1}_{k_0 < \sqrt{n}}\leq  \sqrt{e} \frac{k_0}{\sqrt{\log(1+\frac{k_0^2}{n})}}\ .
\eeq 
Hence,  the error probability $\P_{\theta}[T^{B}_{\alpha,k_0}=0]$ is smaller than $\beta$ as soon as 
\beq\label{eq:objective}
 \E_{\theta}(Z(s)) \geq k_0+ v_{k_0}\ ,
\eeq
where $v_{k_0}$ is defined in \eqref{eq:definition_bk0}. Thus, it suffices to prove \eqref{eq:objective}. The control of the expectation $\E_{\theta}[Z(s)]= \sum g(s\theta_i)$ is more challenging than under the null hypothesis.

Observe that, for large $x$, $g(x)$ goes to one at rate $1/x^2$. For small $x$, a Taylor expansion of $\cos(x)$ leads to $g(x)=O(x^2)$. Let us provide non-asymptotic lower bounds of $g$ matching these two asymptotic behaviors. Since $1-\cos(x)\leq 2$, it follows from the definition of $g$ that $g(x)\geq 1-4/x^2$ for any $x\neq 0$. Studying the three first derivatives of $g$, we derive that $g$ is increasing on $[0,\pi]$.  By Taylor's Lagrange formula, $g(x)\geq  \frac{x^2}{6}\cos(\pi/3)$ for any $x\in [0,\pi/3]$. Since $g$ is increasing on $[0,\pi]$, this implies that $g(x)\geq (\frac{\pi}{3\cdot 2.1})^2\cdot \frac{x^2}{12}\geq x^2/50$ for any $x\in [0,2.1]$.  We have proved that
\beq\label{lower_W}
g(x)\geq \left\{ \begin{array}{cc}
\frac{x^2}{50}& \text{ if } |x|\leq 2.1,\\
1 - \frac{4}{x^2} & \text{ for any }x\neq 0. \ 
                                                                \end{array}
\right.
\eeq 
Observe that the function $f$ defined by $f(x)= x^2/50$ if $|x|\leq 2.1$ and $f(x)=1-4/x^2$ for $x>2.1$ is increasing with respect to $|x|$. Since $g(x)$ is non-negative for all $x$, it follows from Condition \eqref{eq:separation_test_log} that the two following inequalities hold
\begin{eqnarray} \label{eq:lower_EZ1}
 \E_{\theta}[Z(s)]&\geq& (k_0+q) \frac{(s\mu_{q})^2}{50}\mathbf{1}_{s\mu_{q}\leq 2.1}\ , \\
 \E_{\theta}[Z(s)]&\geq& (k_0+q) \Big(  1- \frac{4}{(s\mu_{q})^2}\Big) \ .\label{eq:lower_EZ2}
\end{eqnarray}
We consider two cases.

\noindent 
{\bf Case 1}: $q\leq \frac{256(k_0+v_{k_0})}{(2.1)^2}$. By Definition \eqref{eq:definition_muB}  of $\mu_{q}$, we have
$(s\mu_q)^2/4= 64(k_0+v_{k_0})/q$. The above condition on $q$ enforces that $62(k_0+v_{k_0})\geq q$, which in turn implies that 
$q(s\mu_q)^2/4-q\geq 2k_0$. Hence, we have
\[
 q \Big(  1- \frac{4}{(s\mu_{q})^2}\Big) \geq \frac{8 k_0}{(s\mu_{q})^2} \ . 
\]
Also, the condition on $q$ enforces that $\frac{4}{(s\mu_{q})^2}\leq (2/2.1)^2$. Since we assume in \eqref{eq:separation_test_log} that  $q\geq   2/[1-(2/2.1)^2]v_{k_0}$, it follows that
\[
 q \Big(  1- \frac{4}{(s\mu_{q})^2}\Big) \geq \frac{8 k_0}{(s\mu_{q})^2} \vee (2v_{k_0})\ . 
\]
Then, the lower bound \eqref{eq:lower_EZ2} implies that 
\beqn 
\E_{\theta}[Z(s)] &\geq& k_0 \Big(  1- \frac{4}{(s\mu_{q})^2}\Big) + \frac{8 k_0}{(s\mu_{q})^2} \vee (2v_{k_0})\\
&\geq  & k_0 + v_{k_0}\ ,
\eeqn 
where we used $x\vee y\geq (x+y)/2$. We have proved \eqref{eq:objective}.

\noindent 
{\bf Case 2}: $q> \frac{256(k_0+v_{k_0})}{(2.1)^2}$. This implies that $s\mu_q\leq 2.1$ allowing us to apply the lower bound \eqref{eq:lower_EZ1}. Thus, 
\[
\E_{\theta}[Z(s)] \geq q\frac{(s\mu_{q})^2}{50}= \frac{256}{50} (k_0+ v_{k_0})\geq k_0+v_{k_0}\ .
\]

\bigskip 

It remains to prove that the power of $T^{B}_{\alpha,k_0}$  is larger than $1-\beta$ for any $\theta$ satisfying \eqref{eq:separation_test_log_l2}. From the lower bound \eqref{lower_W} (and since the function $f$ derived from this lower bound \eqref{lower_W} is increasing with respect to $|x|$), we deduce that
\[
 \E_{\theta}[Z(s)]= \sum_{i=1}^n g(s\theta_i)\geq \sum_{i=1}^n \frac{(|s\theta_i|\wedge 2.1)^2}{50}\geq  k_0+v_{k_0}\ ,
\]
which implies  \eqref{eq:objective}. This  concludes the proof.

 \end{proof}

\begin{proof}[Proof of Lemma \ref{lem:concentration_Z}]
As a matter of fact, $Z(s)=f(Y_1,\ldots, Y_n)$ is a lipschitz function of the Gaussian vector $Y_1,\ldots Y_n$. In order to apply the Gaussian concentration theorem, we need to bound its lipshitz norm $\|f\|_L$. The derivative $\kappa'_s(x)$ of $\kappa_s(x)$ satisfies
\beqn 
\big|\kappa_s'(x)\Big|&=& \Big|\int_{-1}^1 (1-|\xi|) e^{s^2\xi^2/2} s\xi \sin(s\xi x)d\xi \Big|\\
&\leq& 2 \int_{0}^{1} e^{s^2\xi^2/2} s\xi d\xi \leq 2s^{-1}e^{s^2/2}.
\eeqn 
As a consequence,  $\|f\|_L\leq \frac{2}{s}\sqrt{n}e^{s^2/2}$, which concludes the proof.
 \end{proof}

\subsubsection{Analysis of $T_{\alpha,k_0}^I$}

   For any $l\in \cL_{k_0}$, let
\beq\label{eq:defi_omega}
v^I_{k_0,l}:= \sqrt{2ln^{1/2}}\Big[\sqrt{\log\Big(\frac{\pi^2 [1+\log_2(l/l_{k_0})]^2}{6\alpha}\Big)}+ \sqrt{\log\Big(\frac{1}{\beta}\Big)}\Big]
\eeq
Define $q^I_{\min}:= 16l_{k_0} + 4 v^I_{k_0,l_{k_0}}$. For any integer $q\geq q^{I}_{\min}$, let
\beq \label{eq:definition_l(q)_mu^I}
l(q):= \max\big\{l\in \cL_{k_0}, \, q\geq 16l + 4v^I_{k_0,l})\big\}\ ,\quad \quad \mu^I_{q} := \frac{2\log(k_0/l(q))}{\sqrt{\log(l(q)/\sqrt{n})}}\ .
\eeq

\begin{prp}\label{prp:test_intermediary}
 Assume that $k_0\geq 20\sqrt{n}$ and that $n$ is large enough. The type I error  probability of  $T^{I}_{\alpha,k_0}$ is smaller or equal to $\alpha$.  Besides, any $\theta\in \mathbb{R}^n$ such that
\beq\label{eq:separation_test_intermediary}
|\theta_{(k_0+q)}|\geq \mu_{q}^{I}\, \text{ for some }q \geq q^I_{\min} \ ,
\eeq 
belongs to the high probability rejection region of $T^{I}_{\alpha,k_0}$, that is $\P_{\theta}[T^{I}_{\alpha,k_0}=1]\geq 1-\beta$.

\end{prp}

Proposition \ref{cor:power_TI} is a straightforward corollary of the above proposition. Indeed,  we have $q_{\min}^I\leq c_{\alpha,\beta} \sqrt{k_0n^{1/2}}$. Since $l(q)\geq l_{k_0}= \lceil \sqrt{k_0n^{1/2}}\rceil$ and $k_0\geq 5\sqrt{n}$, it follows that $\log(l(q)/\sqrt{n})\geq \log(l_{k_0}/\sqrt{n})\geq c \log(1+k_0/\sqrt{n})$. For any  we have $l\in \cL_{k_0}$, $v_{k_0,l}^I\leq c_{\alpha,\beta} l$ implying that $l(q)\geq c_{\alpha,\beta} [q\wedge k_0]$ and $\mu_q^I\leq c_{\alpha,\beta}\frac{1+ \log(1+k_0/q)}{\sqrt{\log(1+k_0/\sqrt{n})}}$. 
 
\medskip 
 
 Before proving Proposition \ref{prp:test_intermediary}, we start with a deviation inequality inequality for $V(r_{k_0,l},w_l)$. 
\begin{lem}\label{lem:concentration_Vq}
For any $\theta\in \mathbb{R}^n$, any $k_0\geq 20\sqrt{n}$ and any $l\in \cL_{k_0}$ and any $x>0$, we have 
\beq\label{eq:concentration_Ta}
 \P_{\theta}\left[ V(r_{k_0,l},w_l) - \E_{\theta}[V(r_{k_0,l},w_l) ]\geq \sqrt{2ln^{1/2}x} \right]\leq e^{-x} \ .
\eeq
\end{lem}

\begin{proof}[Proof of Lemma \ref{lem:concentration_Vq}]
Fix $\theta\in \mathbb{R}^n$ and $l\in \cL_{k_0}$. Then $n- V(r_{k_0,l},w_l)=\sum_{i}\eta_{r_{k_0,l},w_l}(Y_i)$ is a sum of $n$ independent random variables bounded in absolute values by
\beqn 
\frac{\sqrt{2}r_{k_0,l}}{\sqrt{\pi}(1-2\Phi(r_{k_0,l}))}e^{(w_l^2-r_{k_0,l}^2)/2}&\leq&  \frac{4}{\sqrt{\pi}}\frac{l^{3/2}\sqrt{\log(k_0/l)}}{n^{1/4}k_0} \quad  \text{(by definition \eqref{eq:param} of $r_{k_0,l}$ and $w_l$) }\\
&\leq &  \sqrt{\frac{\log(4)}{\pi}} \frac{l^{1/2}}{n^{1/4}}\leq \frac{l^{1/2}}{n^{1/4}}\ , 
\eeqn 
where we used that $l\leq k_0/4$. Then, Hoeffding's inequality yields 
\[\P_{\theta}\left[ V(r_{k_0,l},w_l) - \E_{\theta}[V(r_{k_0,l},w_l) ]\geq \sqrt{2ln^{1/2}x} \right]\leq e^{-x} \ ,\quad \text{for any $x>0$.}\]

\end{proof}

\begin{proof}[Proof of Proposition \ref{prp:test_intermediary}]

To ease the notation, we respectively write $v_{l}$, $\mu_q$ and $r_{l}$ for $v_{k_0,l}^I$, $\mu_q^I$, and $r_{k_0,l}$.
We start with a few simple observations that will be used multiple times.

\begin{lem}\label{lem:aqr_q}
 For any $l \in \cL_{k_0}$, we have $r_{l}\geq \sqrt{2\log(4)}\geq  1$, $(1-2\Phi(r_{l}))\geq 0.65$ and  $r_{l}\leq \sqrt{2}w_l$.
\end{lem}

\begin{proof}[Proof of Lemma \ref{lem:aqr_q}]
Since for all $l\in \cL_{k_0}$, $l\leq k_0/4$, $r_{l}\geq \sqrt{2\log(4)}\geq 1$. Computing the quantile of a standard normal distribution, we obtain $1-2\Phi(1)\geq 0.65$. For all $l\in \cL_{k_0}$, we have $l^2\geq l_{k_0}^2\geq k_0\sqrt{n}$, which implies $r_{l}\leq \sqrt{2}w_l$.
\end{proof}

 Let us  now consider the expectation of the statistic $V(r_{l},w_l)$. 
 Given this alternative expression of $\eta_{r,w}(x)$,
\beq\label{eq:def_eta_alt}
\eta_{r,w}(x)=  \frac{1}{1-2\Phi(r)}\int_{-r}^{r} \frac{e^{-\xi^2/2}}{\sqrt{2\pi}} e^{\xi^2 w^2/(2r^2) }\cos(\frac{\xi w x}{r})d\xi\ ,
\eeq
we get, for $X\sim \cN(x,1)$, 
 \beq \label{eq:definition_psiq}
  \E[\eta_{r_{l},w_l} (X)]= \frac{1}{1-2\Phi(r_{l})}\int_{-r_{l}}^{r_{l}} \phi(\xi)\cos(\xi x \frac{w_l}{r_{l}} )d\xi.
 \eeq
 In the sequel, we denote $\Psi_l(x)$ this expectation. Obviously, 
 \[
  \Psi_l(x)\leq  \frac{1}{1-2\Phi(r_{l})}\int_{-r_{l}}^{r_{l}} \phi(\xi)d\xi = 1\ ,
 \]
 whereas $\Psi_l(x)=1$ if and only if $x=0$. The following lemma states sharper bounds for $\Psi_l(x)$. 
\begin{lem}\label{lem:control_psi}
 For any $x\in \mathbb{R}$, 
 \beq\label{eq:lower_V_theta}
  -\frac{l}{k_0} \leq \Psi_l(x)  \leq  2\exp\big(- \frac{w_l^2x^2}{2r_{l}^2}\big) + \frac{l}{k_0}\ .
 \eeq
\end{lem}
As Lemmas \eqref{lem:control_psi} and \ref{lem:concentration_Vq} provide controls  on both the expectation and the deviation of $V(r_{l},w_l)$, we are equipped to bound the type I and type II error probabilities of $T_{\alpha,k_0}^I$.

\bigskip
\noindent 
  {\bf Level of the Test}. Consider  any $\theta\in \mathbb{B}_0(k_0)$. Since $\Psi_l(0)=1$, 
\begin{eqnarray}
\nonumber
\label{eq:control_ETaH0}
 \E_{\theta}[V(r_{l},w_l)]&=& \sum_{i=1}^{n}(1- \Psi_l(\theta_{(i)})) = \sum_{i=1}^{k_0}(1- \Psi_l(\theta_{(i)}))\\
 &\leq& k_0\big[1+ \frac{l}{k_0}\big]\leq k_0 + l \ ,
\end{eqnarray}
where we used Lemma \ref{lem:control_psi}.
Applying the deviation inequality \eqref{eq:concentration_Ta} to $V(r_{l},w_l)$ with the weight $\log(\frac{\pi^2 [1+\log_2(l/l_{k_0})]^2}{6\alpha})$, we derive that, with probability larger than $1-\frac{6\alpha}{\pi^2[1+\log_2(l/l_{k_0})]^2}$, 
\beq\label{eq:concentration_union}
 V(r_{l},w_l) \leq  \E_{\theta }[V(r_{l},w_l) ] +  \sqrt{2ln^{1/2}\log\Big(\frac{\pi^2 [1+\log_2(l/l_{k_0})]^2}{6\alpha}\Big)} \ .
\eeq
Since $\sum_{l\in \cL_{k_0}}\frac{1}{[1+\log_2(l/l_{k_0})]^2}\leq \sum_{i=1}^{\infty} i^{-2}= \pi^2/6$, it follows that \eqref{eq:concentration_union} is simultaneously  valid for all $l\in \cL_{k_0}$ with probability larger than $1-\alpha$. Together with \eqref{eq:control_ETaH0}, this implies that the size of the test of $T_{\alpha,k_0}^{I}$ is smaller than $\alpha$.

\bigskip
\noindent 
  {\bf Power of the Test}.
Let us now consider any vector $\theta$ satisfying \eqref{eq:separation_test_intermediary}. We take $q\geq q_{\min}^{I}$ such that $|\theta_{(k_0+q)}|\geq \mu_{q}$. In the sequel, we simply write $l$ for $l(q)$. 
Using \eqref{eq:lower_V_theta} together with the bound $\Psi_{l}(x)\leq 1$ we obtain
\begin{eqnarray}
\nonumber
\E_{\theta}[V(r_{l},w_l)]&\geq& \sum_{i=1}^{k_0+q} [1- \Psi_{l}(\theta_{(i)})]
\\ \nonumber
&\geq &(k_0+ q) \Big[1- 2\exp\Big(- \frac{w_l^2\mu_{q}^2}{2r_{l}^2}\Big)- \frac{l}{k_0} \Big]\\ 
&\geq & (k_0+q)\Big[ 1 - \frac{3l}{k_0} \Big]= k_0 -  3l + q\Big(1-\frac{3l}{k_0}\Big) \geq k_0 + \frac{q}{4}-3l \ , 
\label{eq:control_ETaH1} 
\end{eqnarray}
where we used the definition \eqref{eq:separation_test_intermediary} of $\mu_{q}$ and $k_0\geq 4l_{\max}\geq 4l$ in the last line. Together with the deviation inequality \eqref{eq:concentration_Ta}, we obtain
\beq\label{eq:concentration_alternative}
 \P_{\theta}\left[V(r_{l},w_l) \geq  k_0 + \frac{q}{4}- 3l -  \sqrt{2ln^{1/2}\log\Big(\frac{1}{\beta}\Big)}\right]\geq 1-\beta \ , 
\eeq
 Coming back to the definition \eqref{eq:definition_l(q)_mu^I} of $l$, this implies that $T_{\alpha,k_0}^{I}$ rejects the null hypothesis with probability larger than $1-\beta$.

\end{proof}

  \begin{proof}[Proof of Lemma \ref{lem:control_psi}]

If we replace the integral of $[-r_{l},r_{l}]$ by an integral over $\mathbb{R}$ in the definition \eqref{eq:definition_psiq} of $\Psi_l(x)$, we recognize the Fourier transform of a standard normal variable.
 \begin{eqnarray}\nonumber
[1-2\Phi(r_{l})]\Psi_q(x)&= &\int_{\mathbb{R}} \phi(\xi)\cos(\xi x \frac{w_l}{r_{l}} )d\xi- 2 \int_{r_{l}}^{\infty} \phi(\xi)\cos(\xi \tfrac{w_l}{r_{l}} x)d\xi\\ &=&   e^{-(w_lx/r_{l})^2/2} -2 \int_{r_{l}}^{\infty} \phi(\xi)\cos(\xi \tfrac{w_l}{r_{l}} x)d\xi\ . \label{eq:decomposition_psiq}
 \end{eqnarray}
Denote $\vartheta_l(x):= \int_{r_{l}}^{\infty} \phi(\xi)\cos(\xi \tfrac{w_l}{r_{l}} x)d\xi$ the remainder term.

 Let $\bar{r}_l\geq r_{l} $ be the smallest number satisfying $ \bar{r}_l  \equiv \pi/2 [\pi]$.
Since the function $\xi\mapsto \phi(\xi)$ is decreasing on $[r_{l},\infty)$, the integral in $\vartheta_l(x)$ can be decomposed as an alternative sum
\[\vartheta_l(x):=  \sum_{i=1}^{\infty} \int_{\bar{r}_l+(i-1)\frac{r_{l}\pi}{w_lx}}^{\bar{r}_l+ i\frac{r_{l}\pi}{w_lx}}\phi(\xi) \cos(\xi \tfrac{w_l}{r_{l}} x)d\xi+ \int_{r_{l}}^{\bar{r}_l}\phi(\xi) \cos(\xi \tfrac{w_l}{r_{l}} x)d\xi \ , \]
where the sign of the integral over $[r_{l},\bar{r}_l]$ is opposite to the one over $[\bar{r}_l, \bar{r}_l+ r_{l}\pi/(w_lx)]$. As a consequence, 
\begin{eqnarray}\nonumber
\big|\vartheta_l(x)\big|&\leq & \Big|\int_{\bar{r}_l}^{\bar{r}_l+ \frac{r_{l}\pi}{w_lx}}\phi(\xi) \cos(\xi \tfrac{w_l}{r_{l}} x)d\xi\Big|\bigvee \Big|\int_{r_{l}}^{\bar{r}_l}\phi(\xi) \cos(\xi \tfrac{w_l}{r_{l}} x)d\xi \Big|\\ \nonumber
&\leq & \int_{r_{l}}^{r_{l} + \frac{r_{l}\pi}{w_lx}}\phi(\xi)d\xi
 =\phi(r_{l})\int_0^{\frac{r_{l}\pi}{w_lx}} e^{-r_{l} \xi }e^{-\xi^2/2} d\xi\\ 
 &\leq & \frac{\phi(r_{l})}{r_{l}}= \frac{l}{k_0\sqrt{2\pi}r_{l}}. \label{eq:upper_vareta}
\end{eqnarray}

Coming back to the decomposition of $\Psi_l(x)$, we obtain
\beqn 
\big|\Psi_l(x)- \frac{e^{-(w_lx/r_{l})^2/2}}{1-2\Phi(r_{l})}\big|   &\leq & \frac{l}{k_0}\cdot \frac{\sqrt{2}}{\sqrt{\pi}r_{l}[1-2\Phi(r_{l})]} \leq   \frac{l}{k_0}\ ,
\eeqn 
where we used Lemma \ref{lem:aqr_q} in the last inequality. Since $1-2\Phi(r_{l})$ is larger than $1/2$ (Lemma \ref{lem:aqr_q} again), the above inequality implies  \eqref{eq:lower_V_theta}.

 \end{proof}

\subsubsection{Analysis of $T^C_{\alpha,k_0}$}

 \begin{proof}[Proof of Corollary \ref{cor:power_combined2}]
 The first bound is a straightforward consequence of Propositions \ref{cor:T_HC}, \ref{cor:TB_power}, and \ref{cor:power_TI}. We focus on the second bound \eqref{eq:upper_adaptatif_l2}. Choosing the constant $c'_{\alpha,\beta}$ small enough, we claim that Condition \eqref{eq:upper_adaptatif_l2} implies that either Condition \eqref{eq:upper_adaptatif_linfini} is true for some $q\leq \Delta$ or that \eqref{eq:separation_log_l2} is true. Corollary \ref{cor:power_combined2} is then a straightforward consequence of this claim.

 We will prove this claim by contraposition. In the sequel, we assume that both \eqref{eq:separation_log_l2} and \eqref{eq:upper_adaptatif_linfini} for all $q\leq \Delta$ are not satisfied. The analysis is divided into 5 cases depending on the values of $k_0$, $\Delta$ and $n$.
 
 \bigskip 
 
 \noindent 
 {\bf Case A.1}: $k_0\leq \sqrt{n}$ and $\Delta\leq \sqrt{n}$. We consider two subcases: (i) $\Delta\leq \sqrt{n}/2$ and (ii) $\Delta> \sqrt{n}/2$. In case (i), the fact that Condition \eqref{eq:upper_adaptatif_linfini} is not satisfied implies 
 \beqn 
 d^2_2(\theta,\mathbb{B}_0(k_0))&=& \sum_{q=1}^{\Delta}\theta^2_{(k_0+q)}\leq c_{\alpha,\beta}\sum_{q=1}^{\Delta}\log(1+ \frac{\sqrt{n}}{q})\leq c_{\alpha,\beta}\big[ \Delta\log(2) + \sum_{q=1}^{\Delta}\log(\frac{\sqrt{n}}{q})\Big]\\ &&=c_{\alpha,\beta}\Big[ \Delta\log(2\sqrt{n})-\log(\Delta!)\Big]\leq c_{\alpha,\beta} \Delta\log\big(\frac{2e \sqrt{n}}{\Delta}\big)\leq 4c_{\alpha,\beta}\Delta\log\big(1+ \frac{\sqrt{n}}{\Delta}\big)\ ,
 \eeqn 
 which contradicts  \eqref{eq:upper_adaptatif_l2}.
 In case (ii), $\log(1+\sqrt{n}/\Delta)\geq\log(2)$. For $n$ large enough, $\sqrt{n}/\lfloor \sqrt{n}/2\rfloor\geq 3$.  Using the above bound, we get
 \beqn 
 d^2_2(\theta,\mathbb{B}_0(k_0))&=& \sum_{q=1}^{\Delta}\theta^2_{(k_0+q)}\leq c_{\alpha,\beta} \sum_{q=1}^{\lfloor \sqrt{n}/2\rfloor}\log(1+ \frac{\sqrt{n}}{q})+ c_{\alpha,\beta} \sum_{q=\lfloor \sqrt{n}/2\rfloor+1}^{\Delta}\log(1+ \frac{\sqrt{n}}{q})\\ 
 & \leq & 4c_{\alpha,\beta} \lfloor \sqrt{n}/2\rfloor\log(4)+ c_{\alpha,\beta} (\Delta- \lfloor \sqrt{n}/2\rfloor)\log(3)
 \leq c_{\alpha,\beta} 4\Delta \log(4)\\ &\leq& c''_{\alpha,\beta} \Delta\log\big(1+ \frac{\sqrt{n}}{\Delta}\big)\ ,
 \eeqn 
  which contradicts again \eqref{eq:upper_adaptatif_l2} if $c'_{\alpha,\beta}$ in \eqref{eq:upper_adaptatif_l2} is chosen small enough.
 
 \bigskip

 \noindent 
 {\bf Case A.2}: $k_0\leq \sqrt{n}$ and $\Delta\geq \sqrt{n}$. We start from 
 \[
  d^2_2(\theta,\mathbb{B}_0(k_0))= \sum_{q=1}^{\lfloor \sqrt{n}\rfloor}\theta^2_{(k_0+q)}+ \sum_{q=\lfloor \sqrt{n}\rfloor+1}^{\Delta }\theta^2_{(k_0+q)}
 \]
The first sum is small in front of $\sqrt{n}$ by Case A.1(i). 
 Since \eqref{eq:upper_adaptatif_linfini}, is not not satisfied this implies that all $|\theta_{(k_0+q)}|$ for $q\geq \sqrt{n}$ are (up to multiplicative constants) smaller than $1= s_{k_0}$. Together with the fact that
  \eqref{eq:separation_log_l2} is not satisfied, this implies that 
  \[
   \sum_{q=\lfloor \sqrt{n}\rfloor+1}^{\Delta }\theta^2_{(k_0+q)}\leq c_{\alpha,\beta} \sqrt{n}\ . 
  \]
We have proved that 
\[
   d^2_2(\theta,\mathbb{B}_0(k_0))\leq c_{\alpha,\beta} \sqrt{n}\leq c''_{\alpha,\beta} \Delta\log\big(1+ \frac{\sqrt{n}}{\Delta}\big)
\]
  which contradicts  \eqref{eq:upper_adaptatif_l2} if $c'_{\alpha,\beta}$ in \eqref{eq:upper_adaptatif_l2} is chosen small enough.
 
 \bigskip

 \noindent 
 {\bf Case B.1}: $k_0> \sqrt{n}$ and $\Delta\leq \sqrt{n}$. We argue exactly as in case A.1(i). The fact that Condition \eqref{eq:upper_adaptatif_linfini} is not satisfied implies that 
 \beqn 
 d^2_2(\theta,\mathbb{B}_0(k_0))\leq c_{\alpha,\beta}\Delta\log\big(1+ \frac{k_0}{\Delta}\big)\,
 \eeqn 
 which contradicts  \eqref{eq:upper_adaptatif_l2} if $c'_{\alpha,\beta}$ in \eqref{eq:upper_adaptatif_l2} is chosen small enough.
 \bigskip 
 
\noindent  
 {\bf Case B.2}: $k_0 >\sqrt{n}$ and $ \sqrt{n}< \Delta\leq k_0$. 
 The fact that Condition \eqref{eq:upper_adaptatif_linfini} is not satisfied implies
 \beqn 
 d^2_2(\theta,\mathbb{B}_0(k_0))&\leq& \frac{c_{\alpha,\beta}}{\log\big(1+\frac{  k_0}{\sqrt{n}}\big)}\sum_{q=1 }^{\Delta} \log^2\big(1+ \frac{k_0}{q}\big)\\
 &\leq &2\frac{c_{\alpha,\beta}}{\log\big(1+\frac{  k_0}{\sqrt{n}}\big)} \Big[\Delta \log^2\big(1+ \frac{k_0}{\Delta}\big) + \sum_{q=1}^{\Delta}\log^2 \big(\frac{\Delta}{q}\big)\Big]
 \eeqn 
 Let us focus on the last sum in the rhs. Comparing the sum with an integral yields
 \beqn 
 \sum_{q=1}^{\Delta}\log^2 \big(\frac{\Delta}{q}\big)&\leq& \log^2(\Delta)+ \int_{1}^{\Delta}\log^2(\frac{\Delta}{t})dt = \log^2 (\Delta) + \Delta \int_{1}^{\Delta}\frac{\log^2(x)}{x^2}dx
 \\
 &\leq & \log^2(\Delta) +  \Delta \int_{1}^{\infty}\frac{\log^2(x)}{x^2}dx\leq c \Delta \ .
 \eeqn 
 Putting everything together, we obtain
 \[d^2_2(\theta,\mathbb{B}_0(k_0))\leq c_{\alpha,\beta}  \Delta \frac{ \log^2\big(1+ \frac{k_0}{\Delta}\big)}{\log\big(1+\frac{  k_0}{\sqrt{n}}\big)}\ , \]
 which contradicts  \eqref{eq:upper_adaptatif_l2} if $c'_{\alpha,\beta}$ in \eqref{eq:upper_adaptatif_l2} is chosen small enough.
 
 \bigskip

 \noindent 
 {\bf Case B.3}: $k_0> \sqrt{n}$ and $\Delta> k_0$. As in Case A.2, we divide the distance into two sums. 
 \[
  d^2_2(\theta,\mathbb{B}_0(k_0))= \sum_{q=1}^{k_0}\theta^2_{(k_0+q)}+ \sum_{q=k_0+1}^{\Delta}\theta^2_{(k_0+q)}
 \]
The first sum has already been handled in Case B.2 and is (up to constants) smaller than $k_0/\log[1+k_0/\sqrt{n}]$. Condition \eqref{eq:upper_adaptatif_linfini} ensures that 
 all coefficients $\theta_{(k_0+q)}$ with $q>k_0$ are (in absolute values and up to  constants) smaller than $1/\log[1+k_0/\sqrt{n}]$. As a consequence, 
 \[
  \sum_{q=k_0+1}^{\Delta}\theta^2_{(k_0+q)}\leq c_{\alpha,\beta}\sum_{q=k_0+1}^{\Delta} \big(\theta^2_{(k_0+q)}\wedge \frac{1}{s_{k_0}^2}\big)\ ,
 \]
which is (up to constants) smaller than $k_0/\log[1+k_0/\sqrt{n}]$ by Condition \eqref{eq:separation_log_l2}. This contradicts again  \eqref{eq:upper_adaptatif_l2}  if $c'_{\alpha,\beta}$ in \eqref{eq:upper_adaptatif_l2} is chosen small enough.

 \end{proof}

\section{Proof of Theorem \ref{thm:estimation}}

As in the previous section, it is assumed that $\sigma=1$. 
   These proofs follow closely the same steps as the analysis of $T^{HC}_{\alpha,k_0}$, $T^B_{\alpha,k_0}$, and $T^I_{\alpha,k_0}$ in Section \ref{sec:proof_KVUB}. We also use the same notation. Fix any $\theta\in \mathbb{R}^n$.
 \bigskip
 
 \noindent 
 {\bf Type I error (Proof of \eqref{eq:level_estimator})}. We consider separately $\widehat{k}^{HC}$, $\widehat{k}^{B}$ and $\widehat{k}^{I}$ and we will prove that, for each of them, the probability that it exceeds $\|\theta\|_0$ is smaller than $\alpha/3$. First, we consider $\widehat{k}^{HC}$. Arguing as in the proof of Proposition \ref{prp:T_HC_levelpower}, we have
 \[\P_{\theta}[N_{t^*}\geq \|\theta\|_0]\leq 2[n-\|\theta\|_0]\Phi(t_*)\leq 2n \exp\big[-t^{*2}/2\big]\leq \alpha/6\]
 For any positive integer $t$, $N_t-\|\theta\|_0$ is stochastically larger than a Binomial distribution with parameter $(n-\|\theta\|_0, 2\Phi(t))$. In view of the definition \ref{eq:rejection_Nt} of $u_{t,\alpha}^{HC}$, Bernstein's inequality yields
 \[
  \P_{\theta}\big[N_{t}\leq  \|\theta\|_0 + 2(n-\|\theta\|_0)\Phi(t)+ u_{t,\alpha/3}^{HC}\big]\geq 1 - \frac{\alpha \pi^2}{t^2}\ .
 \]
Taking an union bound over all $t\in \cT$, we derive that with probability larger than $1-\alpha/6$, 
\[
 \max_{t\in \cT}\frac{N_t- 2n\Phi(t)- u^{HC}_{t,\alpha/3}}{1-2\Phi(t)}\leq \|\theta\|_0\ .
\]
We have proved that $\P_{\theta}[\widehat{k}^{HC}>\|\theta\|_0]\leq \alpha/3$. 
 
Let us turn $\widehat{k}^{B}$. Lemma \ref{lem:concentration_Z} provides a deviation inequality for all statistics $Z(s)$. Together with the definition \eqref{eq:definition_Tb_alpha} of $u_{k_0,\alpha}^B$, this yields
\[
  \P_{\theta}\big[Z(s_{k_0})\leq \E_{\theta}[Z(s_{k_0})] + u_{k_0,\alpha_{k_0}}^B \big]\geq 1 - \frac{2\alpha \pi^2}{[1+\log_2(k_0/k_{\min})]^2},
\]
for all $k_0$ in the dyadic collection $\cK_0$. Besides the identity \eqref{eq:decomposition_expectation_Z} ensures that $\E_{\theta}[Z(s)]\leq \|\theta\|_0$ for any $s>0$. Taking an union bound over all $k_0\in \cK_0$, we obtain $\P_{\theta}[\widehat{k}_B\leq k_0]\geq 1-\alpha/3$.

Finally, we consider $\widehat{k}^I$. The deviation inequality for $V(r,w)$ (Lemma \ref{lem:concentration_Vq}) and the definition \eqref{eq:rejection_intermediary} of $u^I_{k_0,l,\alpha_{k_0}}$ ensures that, with probability larger $1-\alpha_{k_0}$, we have
\[
 V(r_{k_0,l},w_l)\leq \E_{\theta}[V(r_{k_0,l},w_l)]+ u^I_{k_0,l,\alpha_{k_0}},
\]
simultaneously for all $l\in \cL_{k_0}$. By Lemma \ref{lem:control_psi}, we have $\E_{\theta}[V(r_{k_0,l},w_l)]\leq \|\theta\|_0(1+\frac{l}{k_0})$. Since $\sum_{k_0\in \cK_0}\alpha_{k_0}\leq \alpha/3$, we conclude that $\P_{\theta}[\widehat{k}_I\leq \|\theta\|_0]\geq 1- \alpha/3$.

 \bigskip 
 
 \noindent
 {\bf Type II error (Proof of \eqref{eq:puissance_estimateur} and \eqref{eq:puissance_estimateur2})}. 
For any $t>0$, we denote $N_t^{\theta}$ the number of components of $\theta$ larger or equal to $t$ (in absolute value). Also, we write $t_{*,\alpha}$ for  $t_{*,\alpha}^{HC}$. 
 Arguing as for the type I error, we shall prove that with probability larger than $1-\beta$, all the statistics involved in $\widehat{k}_{HC}$, $\widehat{k}_{B}$, and $\widehat{k}_{I}$ are not much smaller than their expectation. First, an union bound tell us that, with probability larger than $1-\beta/6$,
 \[
 N_{t_{*,\alpha/3}}\geq N^{\theta}_{t_{*,\alpha/3}+t_{*,\beta/3}}\ .
 \]
Besides, for any $t>0$, $N_t$ is stochastically larger than a sum of a Binomial distribution with parameter $(N_{2t}^{\theta}, 1-\Phi(t))$ and Binomial distribution with parameter $(n-N_{2t}^{\theta},2\Phi(t))$. Since the variance of this sum is smaller than $2n\Phi(t)$, it follows from Bernstein's inequality together with an union bound that, with probability larger than $1-\beta/6$, we have
\[
 N_t\geq N_{2t}^{\theta}(1-\Phi(t))+ 2(n-N_{2t}^{\theta})\Phi(t) - u_{t,\beta/3}^{HC}\ ,
\]
simultaneously for all $t\in \cT$.  For any $k_0\in \cK_0$, denote $\beta_{k_0}:= 2\beta([1+\log_2(\tfrac{k_0}{k_{\min}})]^2 \pi^2)^{-1}$. Then, Lemmas \ref{lem:concentration_Z} and \ref{lem:concentration_Vq}, ensure that, with probability larger than $1-2\beta/3$,
\beqn
Z(s_{k_0})\geq \E_{\theta}[Z(s_{k_0})] - u_{k_0,\beta_{k_0}}^B \ ,\\
V(r_{k_0,l},w_{k_0})\geq \E_{\theta}[V(r_{k_0,l},w_{k_0})] - u_{k_0,l,\beta_{k_0}}^I \ ,
\eeqn
simultaneously for all $k_0\in \cK_0$ and all $l\in \cL_{k_0}$. Putting everything together, we conclude that with probability larger than $1-\beta$, we have $\widehat{k}\geq k^{\theta}_{HC}\vee  k^{\theta}_{B}\vee  k^{\theta}_{I}$, where these three deterministic quantities are defined by
\begin{eqnarray}\label{eq:def_K_HC}
  k^{\theta}_{HC}&:=& N^{\theta}_{t_{*,\alpha/3}+t_{*,\beta/3}}\bigvee \max_{t\in \cT}\frac{N_{2t}^{\theta}[1-3\Phi(t)]-u_{t,\alpha/3}^{HC}-u_{t,\beta/3}^{HC}}{1-2\Phi(t)}\\ \label{eq:def_K_B}
  k^{\theta}_{B}&:=& \max_{k_0}\E_{\theta}[Z(s_{k_0})]- (u_{k_0,\alpha_{k_0}}^B  + u_{k_0,\beta_{k_0}}^B )\\ \label{eq:def_K_I}
  k^{\theta}_{I}&:=& \max_{k_0\geq 20\sqrt{n}}\,\sup_{l\in\cL_{k_0}}\frac{\E_{\theta}[V(r_{k_0,l},w_l)] -(u^I_{k_0,l,\alpha_{k_0}} + u^I_{k_0,l,\beta_{k_0}})}{1+l/k_0}.
\end{eqnarray}
\bigskip

We study separately the consequence of the three inequalities $\widehat{k}\geq k^{\theta}_{HC}$,  $\widehat{k}\geq k^{\theta}_{B}$, and  $\widehat{k}\geq k^{\theta}_{I}$. First, we consider $k^{\theta}_{HC}$. Define $q_+:= \frac{16}{3}\log\big(\frac{t_{*,\alpha/3}^2\pi^2}{3(\alpha\wedge \beta)}\big)$ and fix any $q\in [n-\widehat{k}]$.

\noindent
Case 1: $q\leq q_+$. The condition $k^{\theta}_{HC}\leq \widehat{k}$ implies $N^{\theta}_{t_{*,\alpha/3}+t_{*,\beta/3}}\leq \widehat{k} < \widehat{k}+q$, which is equivalent to 
\[
 \big|\theta_{\widehat{k}+q}\big|\leq t_{*,\alpha/3}+t_{*,\beta/3}\leq c\sqrt{\log\big(\frac{4n}{\alpha\wedge \beta}\big)}\leq c_{\alpha,\beta}\Big[1+\sqrt{\log(1+\frac{\widehat{k}\vee \sqrt{n}}{q}}\Big]\ ,
\]
since $q\leq q_+$.
\medskip

\noindent 
Case 2: $q>q_+$. Let $t$ be the smallest number such that $8\Phi(t)< 1\vee \frac{q}{\widehat{k}}\vee \frac{q^2}{32n\log\big(\frac{t^2\pi^2}{3(\alpha\wedge \beta)}\big)}$. Then, we take $t'=\lceil t\rceil\wedge t_{*,\alpha/3}$. If $t' < t_{*,\alpha/3}$, we have $N^{\theta}_{2t'}< \widehat{k}+q$. Indeed, $N^{\theta}_{2t'}\geq  \widehat{k}+q$ would imply
\[
k^{\theta}_{HC} \geq \frac{(\widehat{k}+q)[1-3\Phi(t)]-u_{t,\alpha/3}^{HC}-u_{t,\beta/3}^{HC}}{1-2\Phi(t)}> \widehat{k} + \frac{2q}{3}- \frac{8}{3}u_{t,(\alpha\wedge \beta)/3}^{HC}\geq \widehat{k}\ ,
\]
where we used the definition of $t$ and that  $q>q_+$. This contradicts $\widehat{k}\geq k^{\theta}_{HC}$. We have proved that 
$\big|\theta_{(\widehat{k}+q)}\big|\leq c_{\alpha,\beta}\big[1+\sqrt{\log(1+\frac{\widehat{k}\vee \sqrt{n}}{q})}\big]$. If $t'= t_{*,\alpha/3}$, then we have 
$N_{t'+t_{*,\beta/3}}^{\theta}\leq \widehat{k} < \widehat{k}+q$ as in Case 1. Gathering the bounds for Cases 1 and 2, we have proved that, for all $q=1,\ldots , n-\widehat{k}$, 
\beq\label{eq:puissance_estim_HC}
\big|\theta_{(\widehat{k}+q)}\big|\leq c_{\alpha,\beta}\Big[1+\sqrt{\log\Big(1+\frac{\widehat{k}\vee \sqrt{n}}{q}\Big)}\Big]\ .
\eeq

\bigskip 

Turning to $k^{\theta}_B$, we define $\overline{k}_0$ as the smallest $k_0\in \cK_0$ such that $k_0\geq \widehat{k}/2$. Note that $\overline{k}_0$ always exists since $k_{\max}>n/2$.
The definition \eqref{eq:def_K_B} of $k^{\theta}_{B}$ implies that 
\begin{eqnarray}\nonumber
 \E_{\theta}[Z(s_{\overline{k}_0})]&\leq& \widehat{k} + u^{B}_{\overline{k}_0, \alpha_{k_0}}+ u^{B}_{\overline{k}_0, \beta_{k_0}}\\ \nonumber &\leq& \widehat{k} + c \frac{\overline{k}_0}{\sqrt{1+ \log(\overline{k}^2_0/n)} }\sqrt{\log\Big[\frac{[1+\log_2(\frac{\overline{k}_0}{k_{\min}})]^2\pi^2}{\alpha\wedge \beta}\Big]}\\
 &\leq & \widehat{k} + c_{\alpha,\beta} \overline{k}_0\leq c'_{\alpha,\beta}[\sqrt{n} \vee \widehat{k}] \ , \label{eq:etape0}
\end{eqnarray}
where we used in the second line the definition of $\alpha_{k_0}$ and of $u^{B}_{\overline{k}_0,\alpha_{k_0}}$ and $\overline{k}_0\geq k_{\min}\geq \sqrt{n}$ in the last line. From the definition 
\eqref{eq:decomposition_expectation_Z} of the expectation of $Z(s_{\overline{k}_0})$ and its lower bound \eqref{lower_W}, we derive that
\[\E_{\theta}[Z(s_{\overline{k}_0})]\geq (\widehat{k}+q)f[s_{\overline{k}_0} |\theta_{(\widehat{k}+q)}|]\geq q f[s_{\widehat{k}/2} |\theta_{(\widehat{k}+q)}|], \]
since $g$ is increasing. As a consequence, 
\[
 f[s_{\widehat{k}/2} |\theta_{(\widehat{k}+q)}|]\leq  c'_{\alpha,\beta}\frac{\sqrt{n} \vee \widehat{k}}{q}\ . 
\]
Relying on the definition~\eqref{lower_W} of $f$, we obtain 
\beq\label{eq:puissance_estim_B}
 |\theta_{(\widehat{k}+q)}|\leq c_{\alpha,\beta}\sqrt{\frac{\widehat{k}}{q\log\big(1+\frac{\widehat{k}}{\sqrt{n}}\big)}}\ , \quad \quad \text{ for all }q\geq c'_{\alpha,\beta}\big[\widehat{k}\vee \sqrt{n}\big]\ .
\eeq

\bigskip 

Finally, we investigate  $k^{\theta}_I$. Since  \eqref{eq:puissance_estim_HC} and \eqref{eq:puissance_estim_B} are alone sufficient to prove \eqref{eq:puissance_estimateur} for $\widehat{k}\leq 40\sqrt{n}$. We assume henceforth that $\widehat{k}\geq 40\sqrt{n}$. Let $\overline{k}_0$ be defined as previously. Note that $\overline{k}_0$ is larger than $20\sqrt{n}$. The definition \eqref{eq:def_K_I} of $k^{\theta}_{I}$ implies that, for all $l\in \cL_{\overline{k}_0}$, 
\beqn 
\E_{\theta}[V(r_{\overline{k}_0,l}w_l)]&\leq& \widehat{k}\big[1+ \frac{l}{\overline{k}_0}\big]+ u^I_{\overline{k}_0,l,\alpha_{\overline{k}_0}}+ u^I_{\overline{k}_0,l,\alpha_{\overline{k}_0}}\\
&\leq & \widehat{k}+ 2l+ c_{\alpha,\beta}\sqrt{ln^{1/2}\big[1 + \log\log(l/l_{\overline{k}_0})+ \log\log(\overline{k}_0/k_{\min})]}\\
&\leq & \widehat{k}+ c_{\alpha,\beta} l\ ,
\eeqn 
where we used the definition \eqref{eq:rejection_intermediary} of $u^{I}_{k_0,l,\alpha}$ in the second line and the inequalities $\overline{k}_0\geq k_{\min}\geq \sqrt{n}$, $l\geq \sqrt{n^{1/2}\overline{k}_0}$ in the third line. Lemma \ref{lem:control_psi} then ensures that
\beqn 
\E_{\theta}[V(r_{\overline{k}_0,l}w_l)]&\geq& (\widehat{k}+q)\Big[1 - \frac{l}{\overline{k}_0}-2\exp\Big(-\frac{w^2_l\theta^2_{(\widehat{k}+q)}}{2r^2_{\overline{k}_0,l}}\Big)\Big]\\
&\geq & \widehat{k}- 2l +\frac{3q}{4} - 4(\widehat{k}\vee q)\exp\Big(-\frac{w^2_l\theta^2_{(\widehat{k}+q)}}{2r^2_{\overline{k}_0,l}}\Big) \ , 
\eeqn 
since $l\leq \overline{k}_0/4$ by definition of $\cL_{\overline{k}_0}$. These two bounds imply that for all $q\geq 1$ and all $l\in \cL_{\overline{k}_0}$, we have
\[
 \theta^2_{(\widehat{k}+q)}\leq c\frac{\log\big(\frac{\overline{k}_0}{l}\big)}{\log\big(\frac{l}{\sqrt{n}}\big)} \log\Big(\frac{4(\widehat{k}\vee q)}{[\frac{3q}{4}-c'_{\alpha,\beta}l]_+}\Big)\leq c''\frac{\log\big(\frac{\widehat{k}}{l}\big)}{\log\big(\frac{\widehat{k}}{\sqrt{n}}\big)} \log\Big(\frac{4(\widehat{k}\vee q)}{[\frac{3q}{4}-c'_{\alpha,\beta}l]_+}\big) \ , 
\]
with the convention $\log(1/0)=\infty$. For any $q\geq 2c'_{\alpha,\beta}l_{\overline{k}_0}\geq c'_{\alpha,\beta}\sqrt{2\widehat{k}n^{1/2}}$ with $c'_{\alpha,\beta}$ as above, we obtain by taking $l_q=\max\{l \in \cL_{k_0},\ \text{such that } q \geq 2c'_{\alpha,\beta}\}$, that 
\beq\label{eq:puissance_estim_I}
\theta^2_{(\widehat{k}+q)}\leq c_{\alpha,\beta} \frac{\log^2\big(2 + \frac{\widehat{k}}{q}\big)}{\log\big(1+\frac{\widehat{k}}{\sqrt{n}}\big)}\ .
\eeq
Putting together \eqref{eq:puissance_estim_HC}, \eqref{eq:puissance_estim_B} and \eqref{eq:puissance_estim_I} and playing with the constants, we prove \eqref{eq:puissance_estimateur}.

\bigskip 

As argued in the proof of Corollary \ref{cor:power_combined2}, the second result \eqref{eq:puissance_estimateur2} is a consequence of \eqref{eq:puissance_estimateur} together with the upper bound.
\beq\label{eq:estimateur_power_l2}
 \sum_{q=1}^{n-\widehat{k}} \big[\theta^2_{(\widehat{k}+q)}\wedge \frac{1}{s^2_{\widehat{k}}}\big]\leq c_{\alpha,\beta}\frac{\widehat{k}}{\log\big[1+\frac{\widehat{k}}{\sqrt{n}}]}\ .
\eeq
Thus, we will skip the details for \eqref{eq:puissance_estimateur2} and only prove \eqref{eq:estimateur_power_l2}. Starting from \eqref{eq:etape0} and the expression \eqref{eq:decomposition_expectation_Z} of $\E_{\theta}[Z_{\overline{k}_0}]$. We have
\[
 \sum_{i=1}^{n} g\big[s_{\overline{k}_0} \theta_{(i)}\big]\leq c_{\alpha,\beta}[\sqrt{n} \vee \widehat{k}].
\]
By \eqref{lower_W}, the function $g$ satisfies $g(x)\geq c(x^2\wedge 1)$. Since $s^2_{\overline{k}_0}\geq s^2_{\widehat{k}}-\log(2)\geq cs^2_{\widehat{k}}$, it follows that 
\[
 \sum_{i=1}^{n} \big[\big(s^2_{\widehat{k}}\theta^2_{(i)}\big)\wedge 1 \big]\leq c_{\alpha,\beta}[\sqrt{n} \vee \widehat{k}]\ , 
\]
which implies \eqref{eq:estimateur_power_l2}.

\begin{proof}[Proof of Corollary \ref{cor:rate_estimation_optimality}]
The first negative result \eqref{eq:1_negative} is a consequence of the minimax lower bounds in Section \ref{sec:testKV}. The second negative \eqref{eq:2_negative} result is expressed in terms of the tail distribution of $\theta$ rather in terms of its $l_2$ distance to a sparsity ball. Nevertheless, one may readily adapt all the proofs of the testing minimax lower bounds to account for this modification. 
\end{proof}

\section{Proofs of the results with unknown variance}

\subsection{Proof of the lower bounds}

\subsubsection{Proof of Proposition \ref{prp:signal_detection_uv}}
This proposition is mostly a consequence of other results in this manuscript. When $\Delta\geq \sqrt{n}$, the minimax lower bound is a consequence of Theorem \ref{prp:lower} for known variance. The extension of the Higher criticism statistic to unknown variance as described in Section \ref{sec:ubuv} below achieves the matching upper bound as proved in Theorem \ref{thm:HCuv}. For $\Delta\geq \sqrt{n}$, the lower bound \eqref{eq:separation_distance_sduv_L} is a consequence of Theorem \ref{thm:lower_bound_mainuv}. To prove the minimax upper bound in \eqref{eq:separation_distance_sduv_L}, we rely on the statistic $S_4 = \frac{n\|Y\|_4^4}{\|Y\|_2^2}-3$ defined in \eqref{eq:definition_S4}. Under the null, Chebychev inequality enforces that $\|Y\|_4^4/\sigma^3= 3n + O_P(\sqrt{n})$ and that $\|Y\|_2^2/\sigma^2= n+O_P(\sqrt{n})$. As a consequence, $S_4= O_P(1/\sqrt{n})$. Under the alternative, one has 
\beqn 
\|Y\|_2^2/\sigma^2&=& \frac{\|\theta\|_2^2}{\sigma^2} + n + O_P(\sqrt{n} + \frac{\|\theta\|_2}{\sigma})\ ,\\
\|Y\|_2^4/\sigma^4&=& \frac{\|\theta\|_2^4}{\sigma^4} +  6\frac{\|\theta\|_2^2}{\sigma^2} + 3n +  O_P\big[\sqrt{n} + \frac{\|\theta\|_2}{\sigma} + \frac{\|\theta\|_4^2}{\sigma^2} +   \frac{\|\theta\|_6^3}{\sigma^3} \big]\ ,
\eeqn
so that 
\beqn 
S_4& =& \frac{(n\|\theta\|_4^4-3 \|\theta\|_2^4)/\sigma^4+ O_P\big[n^{3/2} + \frac{n\|\theta\|_2}{\sigma} + \frac{n\|\theta\|_4^2}{\sigma^2} +   \frac{n\|\theta\|_6^3}{\sigma^3}\big]}{\big(\frac{\|\theta\|_2^2	}{\sigma^2} + n\big)^2+ O_P\big[n^{3/2}+\frac{\|\theta\|_2^3}{\sigma^3}\big] }\\
&\geq & \frac{\eta n\|\theta\|_4^4/\sigma^4+ O_P\big[n^{3/2} + \frac{n\|\theta\|_2}{\sigma} + \frac{n\|\theta\|_4^2}{\sigma^2} +   \frac{n\|\theta\|_6^3}{\sigma^3}\big]}{\big(\frac{\|\theta\|_2^2}{\sigma^2} + n\big)^2+ O_P\big[n^{3/2}+\frac{\|\theta\|_2^3}{\sigma^3}\big] }\ ,
\eeqn 
where we used $\|\theta\|_2^4\leq \|\theta\|_0\|\theta\|_4^4\leq \Delta\|\theta\|_4^4$. Besides, for $\|\theta\|_4^4/\sigma^4\geq \sqrt{n}$, one has  $\|\theta\|_6^3/\sigma^3 \leq \sqrt{n}+ 2\|\theta\|_4^4/\sigma^4 n^{-1/8}$ (consider separately the components of $\theta$ smaller than one 1, between $1$ and $n^{1/8}$ and larger than $n^{1/8}$). As a consequence, if $\|\theta\|_4^4/\sigma^4$ is large enough is front of $\sqrt{n}$, then $S_4$ will be also large in front of $\sqrt{n}$ with high probability. Define a test $T_4$ rejecting for large values of $S_4$ in such a way that the size of $T_4$ is equal to $\gamma/2$. It follows from the above discussion that the type II error probability will be smaller than $\gamma/2$ for $\|\theta\|_4^4\geq c_{\gamma,\eta}\sigma^4\sqrt{n}$. Since Cauchy-Schwarz inequality enforces that $\|\theta\|_2^2\leq \sqrt{\Delta}\|\theta\|_4^2$, this implies that $\rho_{\gamma,\mathrm{var}}^{*2}[T_4;0,\Delta] \leq c'_{\gamma,\eta} \sigma_+^2 \sqrt{\Delta n^{1/2}}$, which concludes the proof.

\subsubsection{Proof of Theorem~\ref{thm:lower_bound_mainuv}}

By homogeneity, we assume that $\sigma_-\leq 1\leq \sigma_+\leq 2$ in this proof.

\paragraph{Case 1 : $k_0 = 0$.} Let us first consider the case $k_0 = 0$.
This proof follows the same general approach as that of Theorem \ref{prp:lower} for $k_0=0$. Fix $\Delta'= \Delta/2$. 
Define the probability measures $\mu_0=\delta_0$ and $\mu_1 = \frac{\Delta'}{2n}(\delta_{-M} + \delta_{M}) + \frac{(1-\Delta')}{n}\delta_0$, where $M^8 = \Upsilon \frac{n}{(\Delta')^2}$ where $\Upsilon\leq 1$ is a positive constant to be fixed.  In the sequel, we denote  $p = \Delta'/(2n)$ and  $v^2 = 2p M^2$. Finally, we define 
\[
 \mathbf{P}_0:= \mathbb{P}_{0,1} \ ,\quad \quad  \mathbf{P}_1:= \int \mathbb{P}_{\theta,(1-v^2)^{1/2}} \mu_1^{\otimes n}(d\theta)
\]
Note that, when $Y\sim \mathbf{P}_1$, the marginal variance $\var{Y_i}$ are all equal to one.

Let $\theta$ be sampled according to the product distribution $\mu_1^{\otimes n}$. Since $\Delta\geq \sqrt{n}$, Bernstein's inequality implies that $\mu_1^{\otimes n}[\|\theta\|_0\in [\Delta/4,\Delta]$ is close to one (and in particular is larger than 0.55). As in the proof of Theorem \ref{prp:lower} (Step 2), if we can prove that $\|\mathbf{P}_0-\mathbf{P}_1\|_{TV}\leq 0.05$ (for some $\Upsilon$ small enough), then this will enforce that the minimax separation distance $\rho^{*2}_{\gamma,\mathrm{var}}[0,\Delta]$ is larger than $c \Upsilon^{1/4}\sqrt{\Delta n^{1/2}}$.

Both $\mathbf{P}_0$ and $\mathbf{P}_1$ are product measures and can be decomposed as $\mathbf{P}_0=\pi_0^{\otimes n}$ and $\mathbf{P}_1=\pi_1^{\otimes n}$.
By Cauchy Schwarz and by independence, we relate the total variation distance with the $\chi^2$  distance
$$\|\pi_0^{\otimes n} - \pi_1^{\otimes n}\|_{TV} \leq d(\pi_0^{\otimes n}, \pi_1^{\otimes n}),~~~\text{with}~~~d(\pi_0^{\otimes n}, \pi_1^{\otimes n})^2 = \int \frac{(d\pi_1^{\otimes n})^2}{d\pi_0^{\otimes n}} - 1=\Big[\int \frac{d\pi_1^2}{d\pi_0}\Big]^n - 1 .$$
As a Consequence, it suffices to prove that $\int \frac{(d\pi_1)^2}{d\pi_0}\leq 1+ \frac{c}{n}$ for $c=  \log(1+0.05^2)$ to conclude that $\|\pi_0^{\otimes n} - \pi_1^{\otimes n}\|_{TV}\leq 0.05$.
Expanding the integral, we get
\beqn
\int \frac{(d\pi_1)^2}{d\pi_0}  &=& \int\int     \frac{1}{\sqrt{2\pi}(1-v^2)}e^{x^2/2}e^{-\frac{1}{2(1-v^2)}[(x-\theta_1)^2+(x-\theta_1)^2] }\mu_1(d\theta_1)\mu_1(d \theta_2) dx\\
&=& \int     \frac{1}{\sqrt{2\pi}(1-v^2)}e^{-\frac{x^2(1+v^2)}{2(1-v^2)}}\int e^{-\frac{(\theta_1^2+\theta_2^2)}{2(1-v^2)}}e^{\frac{x(\theta_1+\theta_2)}{1-v^2}} \mu_1(d\theta_1)\mu_1(d \theta_2) dx\\
&=& (1-v^4)^{-1/2}\Big[(1-2p^2)+ 4p(1-2p)e^{-\frac{v^2M^2}{2(1-v^4)}} + 2p^2e^{-\frac{M^2}{(1-v^2)}} + 2p^2e^{\frac{M^2}{(1+v^2)}} \Big]\\
&=& (1-4p^2M^4)^{-1/2}\Big[(1-2p^2)+ 4p(1-2p)e^{-\frac{pM^4}{1-4pM^4}} + 2p^2e^{-\frac{M^2}{(1-2pM^2)}} + 2p^2e^{\frac{M^2}{(1+2pM^2)}} \Big]\ ,
\eeqn
since $v^2=2pM^2$. Let $g_1$ and $g_2$ be the two functions defined by 
$$
 g_1(x):= (1-4p^2x^2)^{-1/2}  \ , \quad g_2(x):=(1-2p^2)+ 4p(1-2p)e^{-\frac{px^2}{1-4px^2}} + 2p^2e^{-\frac{x}{(1-2px)}} + 2p^2e^{\frac{x}{(1+2px)}}\ ,
$$
so that $\int \frac{(d\pi_1)^2}{d\pi_0}= g_1(M^2)g_2(M^2)$. Observe that $g_1$ and $g_2$ are symmetric and infinitely differentiable on $(-1/p,1/p)$. Recall that $p\leq 1/4$. A fourth-order Taylor Lagrange inequality leads to 
\[
 g_1(x)\leq 1 + 2p^2x^2 + c_1 p^2x^4 \ , \quad\quad g_2(x)\leq 1 - 2p^2 x^2 + c_2 p^2 x^4\ , \quad \forall x\in [-1,1]
\]
where $c_1$ and $c_2$ are positive numerical constants constants. Since  $M^8=  \Upsilon \frac{n}{(\Delta')^2}\leq 1$, we obtain that 
\[
 \int \frac{(d\pi_1)^2}{d\pi_0}\leq 1+ c_3 p^2M^8 = \Upsilon \frac{c_3}{4n} \ , 
\]
which is small enough if $\Upsilon$ is well-chosen. This concludes the proof.

\paragraph{Case 2 : $k_0 >0 $.} We follow the same lines as above except that we now take $\Delta'= k_0 + \Delta/2$. Since $\Delta\geq \sqrt{n}\geq  k_0$,  Bernstein's inequality implies that $\mu_1^{\otimes n}[\|\theta\|_0\in [k_0+\Delta/4,k_0+ \Delta]$ is close to one. Taking $\Upsilon$ small enough, we have $\|\mathbf{P}_0-\mathbf{P}_1\|_{TV}\leq 0.05$ as above. Thus, we conclude that 
\[
 \rho^{*2}_{\gamma,\mathrm{var}}[k_0,\Delta]\geq c \Delta M^2= c' \sqrt{n^{1/2} \frac{\Delta^2}{\Delta'}}\geq c''\sqrt{\Delta n^{1/2}}\ ,
\]
since $\Delta\geq k_0$.

\subsubsection{Proof of Proposition \ref{prp:n_div_3}}

 We follow the same steps at in the proof of Theorem \ref{thm:lower_bound_mainuv}, except that we now fix $\Delta'=n/3$ (and therefore $p=1/6$) and $M^{12}= \Upsilon/ n$ with some $\Upsilon\in (0,1)$. Since $\Delta\geq n/3 (1+\zeta)$ for some $\zeta>0$, 
 Bernstein's inequality enforces that  $\mu_1^{\otimes n}[\|\theta\|_0\in [\Delta/2,\Delta]$ is close to one when $n$ is large enough. As a consequence, it suffices to prove that, for a suitable choice of $\Upsilon$, $\|\mathbf{P}_0-\mathbf{P}_1\|_{TV}$ is small enough to enforce that $\rho^{*2}_{\gamma,\mathrm{var}}[0,\Delta]$ is larger than $c \Upsilon^{1/6} n^{5/6}$. As in the previous proof, this amount to proving that $\int \frac{(d\pi_1)^2}{d\pi_0}\leq 1+ \frac{c'}{n}$ for $c'=  \log(1+0.05^2)$.  As above this integral writes as $\int \frac{(d\pi_1)^2}{d\pi_0} = g(M^2)$ with 
 \[
  g(x):= (1-4p^2x^2)^{-1/2}\Big[(1-2p^2)+ 4p(1-2p)e^{-\frac{px^2}{1-4px^2}} + 2p^2e^{-\frac{x}{(1-2px)}} + 2p^2e^{\frac{x}{(1+2px)}} \Big]
 \]
In contrast to the general case, the choice $p=1/6$ has been precisely made to nullify the fourth-order expansion term of $g$. Since $g$ is symmetric and $g$ is infinitely differentiable is on $(-2,2)$, there exists a numerical constant $c>0$ such that $g(x)\leq 1 + c x^6$ for all $x\in [-1,1]$, this implies that $\int \frac{(d\pi_1)^2}{d\pi_0}\leq 1+ \frac{c\Upsilon}{n}$. Taking $\Upsilon$ small enough concludes the proof.

\subsubsection{Proof of Theorem~\ref{thm:lower_bound_mainuv2}}

Without loss of generality, we assume that $\sigma_+=1$,  $k_0 \geq c \sqrt{n}$ ($c>0$ is a large enough universal constant) and  that $k_1:= k_0+\Delta $ satisfies $n/2^{16} \geq k_1\geq 2^{16}k_0$.  Set $\tilde k_0 = k_0/2$,  $\tilde k_1 = k_1/2$, $p_0= \tilde k_0/n$ and $p_1=\tilde k_1/n$.
Let $h_0$ and $h_1$ be two probability measures whose expression will be given later.  We consider the probability measure 
\beq\label{eq:definition_mu_1_uv}
\mu_0:=(1-p_0) \delta_0 + p_0 h_0  \text{ and }\mu_1 := (1 -p_1)\delta_0 + p_1 h_1\ .
\eeq
and 
\[
 \mathbf{P}_0:=  \int \mathbb{P}_{\theta,(1+\sigma_0^2)^{1/2}} \mu_0^{\otimes n}(d\theta) \ ,\quad \quad  \mathbf{P}_1:= \int \mathbb{P}_{\theta,1} \mu_1^{\otimes n}(d\theta)
\]
Obviously, $\mathbf{P}_0$ and $\mathbf{P}_1$ are product measures and decompose as $\mathbf{P}_0=\pi_0^{\otimes n}$ and $\mathbf{P}_1= \pi_1^{\otimes n}$. Note that $\pi_0$ is a convolution of the normal distribution with variance $1+\sigma_0^2$ with $\mu_0$ and $\pi_0$ is a convolution of the normal distribution with variance $1$ with the measure $\mu_1$.

\medskip 

By Chebychev's inequality, we have 
\beq\label{eq:el1}
\mu_0^{\otimes n}\big[\|\theta\|_0 > k_0\big]\leq \frac{2}{k_0} \leq 0.1 \ , \quad\quad  \mu_1^{\otimes n}\big[\|\theta\|_1 > k_0\big] \leq 0.1\ ,
\eeq
for $n$ large enough. Also, the following lemma states that, with high probability, the vector $\theta$ sampled from $\mu_1^{\otimes n}$ is far from $\bbB_0[k_0]$. 

\begin{lem}\label{lem:distance_H0_UV}
For $h_1$ defined as in \eqref{eq:def_f1_hat} below and for $n$ large enough, we have 
 \[
  \mu_1^{\otimes n}\Big[d^2_2(\theta,\bbB_0[k_0]) < c \frac{\sqrt{k_0\Delta}}{\log(k_0/\sqrt{n})}\Big] \geq 0.9\ , 
 \]
 where $c$ is some positive universal constant. 
\end{lem}
Now consider any test $T$. Write $d_n=  c \frac{\sqrt{k_0\Delta}}{\log(k_0/\sqrt{n})}$ where $c$ is the constant occurring in the above lemma. As in the proof of Theorem \ref{prp:lower}, we have 
\beqn 
R_{\mathrm{Var}}(T;k_0,\Delta,d^{1/2}_n)&\geq& \sup_{\theta \in \bbB_0[k_0]}\P_{\theta,(1+\sigma_0^2)^{1/2}}[T=1] +  \sup_{\theta \in \bbB_0[k_1,k_0,d^{1/2}_n]}\P_{\theta,1}[T=0] \\
 &\geq& \int \P_{\theta,(1+\sigma_0^2)^{1/2}}[T=1]\overline{\mu}_0^{\otimes n}(d\theta) - \overline{\mu}_0^{\otimes n}[\|\theta\|_{0}> k_0] \\ &&  + \int \P_{\theta,1}[T=0]\overline{\mu}_1^{\otimes n}(d\theta)- \overline{\mu}_1^{\otimes n}\big[|\|\theta\|_{0}>k_1\big]- \overline{\mu}_1^{\otimes n}\big[d^2_2(\theta,\bbB_0[k_0]) \geq d_n\big]\\
 &\geq & \mathbf{P}_0[T=1] + \mathbf{P}_{1}[T=0] - 0.3 = 0.7  + \mathbf{P}_{1}[T=0]- \mathbf{P}_0[T=0]\\
 &\geq & 0.7 - \|\pi_0^{\otimes n} - \pi_1^{\otimes n}\|_{TV}\ . 
 \eeqn 
As a consequence, the result of Theorem \ref{thm:lower_bound_mainuv2} holds as long as we are able to construct prior measures $h_0$ and $h_1$ such that $\|\pi_0^{\otimes n} - \pi_1^{\otimes n}\|_{TV}\leq 0.2$. By Cauchy-Schwarz inequality, we have
\beqn
\|\pi_0^{\otimes n} - \pi_1^{\otimes n} \|_{TV}^2 &\leq& \int \frac{d\pi_0^{\otimes n}}{d\pi_1^{\otimes n}}d\pi_0^{\otimes n} - 1 = \Big[\int \frac{d\pi_0}{d\pi_1}d\pi_0\Big]^n - 1 = \Big[1+ \int \frac{d\pi_0 - d\pi_1}{d\pi_1}d\pi_0\Big]^n - 1 \nonumber\\
&= &\Big[1 + \int \frac{(d\pi_1 - d\pi_0)^2}{d\pi_1}\Big]^n - 1 . 
\eeqn 
As a consequence, it suffices to prove that 
\beq  \label{eq:objective_A}
A:= \int \frac{(d\pi_1 - d\pi_0)^2}{d\pi_1}\leq \frac{\log(1+0.2^2)}{n}\ . 
\eeq

\paragraph{Step 1 : Construction of the probability measures $h_0$ and $h_1$.}~\\
The purpose of this paragraph is to choose $h_0$ and $h_1$ in such a way that the characteristic function $\widehat{\pi}_0$ and $\widehat{\pi}_1$ of $\pi_0$ and $\pi_1$ match on the widest interval possible. Let us call $\widehat{h}_0$ and $\widehat{h}_1$ the characteristic function of $h_0$ and $h_1$. 

It follows from the definition \eqref{eq:definition_mu_1_uv} of $\mu_0$ and $\mu_1$ that 
 $\widehat{\mu}_0(t)=p_0 \widehat{h}_0(t) + (1-p_0) $ and $\widehat{\mu}_1(t) = (1 -p_1)+ p_1 \widehat{h}_1(t)$. Since $\pi_0$ (resp. $\pi_1$) are convolution production of $\mu_0$ (resp.$\mu_1$) with centered Gaussian measure with variance $(1+\sigma_0^2)$ (resp. $\sigma_1$). We have 
\beq\label{eq:def_pi_hat}
\widehat{\pi}_0(t) = \widehat{\mu}_0(t) \exp(-t^2(1+\sigma_0^2)/2),\quad \mathrm{and}\quad \widehat{\pi}_1(t) = \widehat{\mu}_1(t) \exp(-t^2/2)\ .
\eeq
To match $\widehat{\pi}_0(t)$ and $\widehat{\pi}_1(t)$, we therefore require that
\beq \label{eq:matching}
 1-p_0 + p_0 \widehat{h}_0(t)= e^{\sigma_0^2t^2/2}\big[1-p_1 + p_1 \widehat{h}_1(t)\big]\ . 
\eeq
We start with some notation. 
Define  $t^* = c^* \sqrt{\log(\tilde k_0^2/n)}$ with $c^* := 18$ and 
\begin{eqnarray}\label{eq:def_sigma1}
\sigma_1^2&:=& \Big(\frac{p_0}{p_1}\Big)^{1/2} \frac{1}{8t^{*2}}\ ,\quad \quad \quad  \kappa:= 4p_1 \sigma_1^2 t^*=  \frac{1}{2t^*} \sqrt{\frac{p_0}{p_1}}\ ,\\ 
(1-\lambda)&:=& \frac{p_0}{2p_1\kappa t^*}=  \big(\frac{p_0}{p_1}\big)^{1/2}\ , \quad \quad 
\sigma_0^2 := p_1\lambda\sigma_1^2 = \frac{\sqrt{p_1p_0}}{8 t^{*2}}\Big[1 - \big(\frac{p_0}{p_1}\big)^{1/2}\Big]\ .\nonumber
\end{eqnarray}
We first fix $\widehat{h}_1$ and then choose $\widehat{h}_0$ in such a way that \eqref{eq:matching} is satisfied on $[-t^*,t^*]$. 
\beq\label{eq:def_f1_hat}
\widehat{h}_1(t):= \lambda e^{-\sigma_1^2t^2/2}+ (1-\lambda) e^{- \kappa  |t|}.
\eeq
In other words,  $h_1$ is a mixture of a Gaussian measure with variance $\sigma_1^2$ and of a Cauchy measure with  parameter $\kappa$. For any $t \in [-t^*,t^*]$, we define 
\begin{equation}\label{eq:definition_f_0} 
\widehat{h}_0(t):= -\frac{1-p_0}{p_0 } + \frac{e^{\sigma_0^2t^2 /2}}{p_0} \Big[1 -p_1 + p_1 \widehat{h}_1(t)\Big]\ , 
\end{equation}
to satisfy \eqref{eq:matching}. To conclude, it remains to prove that  $\widehat{h}_0$ is the restriction to $[-t^*,t^*]$ of the characteristic function of some probability measure (that will correspond to $h_0$).
\begin{lem}\label{lem:convexity}
If a symmetric function $g$ with $g(0)=1$ is convex and decreasing to $0$ on $\mathbb{R}^+$, then $g$ is the characteristic function of some probability measure. 
\end{lem}
In view of the above lemma, it suffices to prove $\widehat{h}_0$ can be extended to satisfy the above property.  The parameters $\sigma_0$, $\sigma_1$, $\kappa$ and $\lambda$ have been carefully chosen to ensure the following property.

\begin{lem}\label{lem:hatf0}
Assume that $k_1 \geq 6k_0$ and $k_1 \leq n/14$. Then $\widehat{h}_0$ is positive, convex and decreasing on $[0,t^*]$.
\end{lem}
For any $t \geq t^*$, we set $\widehat{h}_0(t) = ( \widehat{h}_0(t^*) + \widehat{h}_0'(t^*) (t-t^*) )_+ =: (a + b(t-t^*))_+$ ,
and for $t\in (-\infty,-t^*)$ we simply take $\widehat{h}_0(t)= \widehat{h}_0(-t)$. In view of this extension, $\widehat{h}_0$ is continuous at $t^*$ and its  slope is $\widehat{h}_0'(t^*)$. Hence,  $\widehat{h}_0$ is a convex and decreasing  on $\mathbb R^+$ and converges to 0 at $+\infty$. Since in addition $\widehat{h}_0$ is positive and $\widehat{h}_0(0)=1$, Lemma \ref{lem:convexity} ensures that $\widehat{h}_0$ is the characteristic function of a probability measures, denoted $h_0$ in the following.

Since $\widehat{h}_0$ is decreasing, positive and convex on $\mathbb{R}^+$, it follows that $a= \widehat{h}(t^*)\in (0,1)$, and that $|b| \leq |\widehat{h}'(0)| = (p_1/p_0) (1 - \lambda) \kappa \leq 1/(2t^*) \leq 9$.

\paragraph{Step 2 : Upper bound of $A$ in terms of derivatives of Fourier transforms.}
To simplify the notation, we write $\pi_0(x)$ (resp. $\pi_1(x)$) for the density corresponding to the probability measure $\pi_0$ and $\pi_1$. Recall that we aim  \eqref{eq:objective_A} to upper bound 
\[A = \int \frac{G(x)^2}{\pi_1(x)}dx\quad \text{ where }\, G(x) := \pi_1(x) - \pi_0(x)\ .\]
Since $\pi_1$  is a mixture distribution with three components, one of which is a normal with variance $1$, we know that 
$\pi_1(x) \geq\tfrac{(1-p_1)}{\sqrt{2\pi }} e^{-x^2/2}$, which implies since $p_1\leq 1/2$ that  
\[A \leq 2 \sqrt{2\pi } \int G^2(x) e^{x^2/2} dx\]
For any function defined on $\mathbb{R}$, denote $\|f\|_2$ its $l_2$ norm. Denote $P_k$ the polynom function $x\mapsto x^k$. 
Then, we take the Taylor expansion of the function $t\mapsto e^t$ to obtain
\begin{eqnarray}
A &\leq& 6 \int G^2(x) \big(\sum_{k=0}^{\infty} \frac{x^{2k}}{2^k k!}\big) dx\nonumber\\
&= &6  \sum_k \frac{1}{2^k k!} \|P_kG\|_2^2\nonumber\\
&\leq& 6  \sum_k \frac{1}{2^k k!} \frac{\|\widehat{G}^{(k)}\|_2^2}{(2\pi)^2}\leq   \sum_k \frac{1}{2^k  k!} \|\widehat{G}^{(k)}\|_2^2\ ,\label{eq:Adist}
\end{eqnarray}
by Plancherel formula and since $\widehat{x^k G} = i^k \widehat{G}^{(k)}/\sqrt{2\pi}$ (recall that $G$ is infinitely differentiable everywhere except at $-t^*$ and $t^*$).

\paragraph{Step 3 : Decomposition of  $\|\widehat{G}^{(k)}\|_2^2$. }
Our choice of $\widehat{\mu}_0$ and $\widehat{\mu}_1$ in Step 1 enforces that $\widehat{G}=\widehat{\pi}_1-\widehat{\pi}_0$ satisfies
\[\widehat G(t) = 0~~~\forall t \in [-t^*,t^*]\ .\]
We have for any $t$ such that $|t| \geq t^*$
\begin{align*}
\widehat{G}(t) &= e^{-t^2/2} \Big[p_1 \lambda e^{-\sigma_1^2t^2/2} + p_1 (1-\lambda) e^{-\kappa|t|} + (1-p_1)\\
&- (1 - p_0) e^{-\sigma_0^2t^2/2} - p_0\lfloor a + b(t-t^*)\rfloor_+ \Big]\\
&:= e^{-t^2/2}  V(t) = \sqrt{2\pi}\phi(t) V(t),
\end{align*}
For any $t$, let us write $V(t) = V_1(t) + V_2(t) + V_3(t)+V_4(t)$, where
\begin{align*}
&V_1(t) = p_1 \lambda e^{-\sigma_1^2t^2/2} - e^{-\sigma_0^2t^2/2} +(1-p_1 \lambda),\\
&V_2(t) = p_0 e^{-\sigma_0^2t^2/2} - p_0,\\
&V_3(t) =  p_1 (1-\lambda) e^{-\kappa|t|} - p_1 (1-\lambda)\\
&V_4(t) = -p_0\lfloor a + b(t-t^*)\rfloor_+.
\end{align*}
As a consequence, we have the decomposition
$$\frac{\widehat{G}^{(k)}(t)}{\sqrt{2\pi}} = \sum_{i=1}^4(\phi  V_i)^{(k)}(t) \mathbf 1\{|t| \geq t^*\}\ ,$$
which enforces 
\begin{align}
\|\widehat{G}^{(k)}\|_2^2 &\leq 64 \pi \sum_{i=1}^4\|(\phi  V_i)^{(k)}\mathbf 1\{t \geq t^*\}\|_2^2,\label{eq:Gdist}
\end{align}

We consider two subcases depending on the values of $k$: for small $k$ ($k\leq 5\log(\tilde{k}_0/\sqrt{n})$), we only need a loose upper bound of $(\phi  V_i)^{(k)}$ but we heavily rely on the fact that this derivative is null for $|t|\leq t^*$. For larger $k$, the computations need to be handled more carefully.

\paragraph{Step 4: Control of $\|\widehat{G}^{(k)}\|_2^2$ for $k\leq  5\log(\tilde{k}_0/\sqrt{n})$.} The  binomial formula enforces that, for $i=1,\ldots,4$, $(\phi  V_i)^{(k)} = \sum_{d=0}^k \binom{k}{d} \phi^{(k-d)} V_i^{(d)}$, implying that 
\beq\label{eq:upper_phi_Vi_k}
\|(\phi  V_i)^{(k)} \mathbf 1\{t \geq t^*\}\|_2^2 \leq 2^{2k} \sup_{d=0}^k\|\phi^{(k-d)} V_i^{(d)} \mathbf 1\{t \geq t^*\}\|_2^2\ .
\eeq 
Define $\overline{V}(t)= 10p_0 (t^4\vee 1)[e^{t/16}+ e^{t^2/16}]$.
\begin{lem}\label{lem:upper_Vi}
 For all nonnegative integers $d$, all $t>0$ and all   $i=1,\ldots, 4$, one has 
\beqn 
V_i^{(d)}(t)\leq \overline{V}^{(d)}(t)\ . 
\eeqn
\end{lem}
\begin{proof}[Proof of Lemma \ref{lem:upper_Vi}]
Writing down the power expansion of $V_1$, we observe that the two first terms cancel out (recall that $p_1\lambda\sigma_1^2=  \sigma_0^2$). As a consequence, 
the smallest order term is of order $p_1 \lambda \sigma_1^4 t^4\leq p_0 t^4$. Besides all the terms of order $t^{2q+4}$ are smaller (in absolute value) than $(1/16)^{q}/q!$ because both $\sigma_0$ and $\sigma_1$ are small enough. This implies $V_1^{(d)}(t)\leq \overline{V}^{(d)}(t)$.
 The results for $V_2$, $V_3$ and $V_4$ follow similarly.
\end{proof}
Define the function  $\phi_+: t\mapsto e^{t^2/2}/\sqrt{2\pi}$. For any nonnegative integer $k$, there exists a polynom $R_k$ of degree less or equal to $k$ such that $\phi^{(k)}_{+}(t)= R_k(t)\phi_+(t)$. By a straightforward induction on $k$, we observe that $|\phi^{(k)}(t)|\leq R_k(t) \phi(t)$. The same recursion allows us to prove that $|\overline{V}^{(k)}(t)|\leq c \big(\max_{q=k,..,(k-3)_+} R_q(t)\big) |\overline{V}(t)|$, where $c$ is a numerical constant. Also, we have $R_kR_q(t)\leq R_{k+q}(t)$. Coming back to \eqref{eq:upper_phi_Vi_k}, we obtain 
\[
\|(\phi  V_i)^{(k)} \mathbf 1\{t \geq t^*\}\|_2^2 \leq c 2^{2k} \sup_{d=0}^4\|R_{k-d} \phi \overline{V} \mathbf 1\{t \geq t^*\}\|_2^2\ . 
\]
Write $R_k(t)= \sum_{j=0}^{k} r_{j,k}t^j$. Again, a straightforward induction leads to $0\leq r_{j,k}\leq \binom{k}{(k+j)/2} k^{(k-j)/2}\leq 2^k k^{(k-j)/2}$. Recall that $t^*\geq 1$. By the triangular inequality, we obtain 
\begin{eqnarray}
\|(\phi  V_i)^{(k)} \mathbf 1\{t \geq t^*\}\|_2^2 &\leq& Cp_0^2 2^{4k}\sup_{j=0}^{k} k^{k-j+1}  \int_{t^*}^{\infty }t^{2j+4 } e^{-7t^2/8}dt \nonumber \\
&\leq& Cp_0^2 2^{4k}\sup_{j=0}^{k} k^{k-j+1}  \int_{t^*}^{\infty }t^{2j+4 }  e^{-7t^2/8}dt \ . 
\label{eq:upper_phi_Vi_k2} 
\end{eqnarray}
Let us now bound this integral
\beqn 
\int_{t^*}^{\infty }t^{2j+4 } e^{-7t^2/8}dt&\leq& e^{-3t^{*2}/8}\int_{\mathbb{R}}t^{2j+4 } e^{-t^2/2}dt =e^{-3t^{*2}/8} \frac{2^{j+1}}{\sqrt{\pi}}\Gamma(j+5/2)\\ 
&\leq&  e^{-3t^{*2}/8} 2^{j+1}(j+3/2)^{j}  \ .
\eeqn 
 Coming back to \eqref{eq:upper_phi_Vi_k2}, we obtain
\[
\|(\phi  V_i)^{(k)} \mathbf 1\{t \geq t^*\}\|_2^2 
\leq cp_0^2 2^{4k} k^{3} k^{k}e^{-3t^2_*/8}  \ . 
\]
Thanks to \eqref{eq:Gdist}, we conclude that, for $k\leq \log(\tilde{k}_0/\sqrt{n})= t^{*2}/(2c^{*2})$,
\begin{eqnarray}  \nonumber
\frac{\|\widehat{G}^{(k)}\|_2^2}{2^k k!} &\leq& cp_0^2 t^{*3} (8e)^k k^{5} e^{-3t^2_*/8}\\
&\leq &  c\frac{p_0^2}{k^2}  e^{-t^{*2}/4} \nonumber \\
&\leq &  \frac{6\log(1+0.2^2)}{\pi^2 nk^2} \ , 
\label{eq:upper_phi_Vi_k_petit}
\end{eqnarray}
where we used that $t^*=c^*\sqrt{\log(\tilde{k}^2_0/n)}$ with $c^*\geq 18$ and that $p_0= k_0/(2n)\geq c n^{-1/2}$ for a constant $c$ large enough.

\paragraph{Step 5: Control of $\|\widehat{G}^{(k)}\|_2^2$ for $k> 5\log(\tilde{k}_0/\sqrt{n})$.}
For such $k$, we may neglect the threshold $t\geq t^*$ but we need to be more careful about the computation of $(\phi  V_1)^{(k)}$. 
\beq\label{eq:transform_t_derivative}
\|(\phi  V_1)^{(k)}\mathbf{1}_{|t|\geq t^*}\|_2^2\leq \|(\phi  V_1)^{(k)}\|_2^2 =\|P_k \widehat{(\phi  V_1)}\|_2^2\  ,
\eeq
where $P_k: t\mapsto t^k$. 
Since $\phi  V_1$ is a linear combination of normal distributions with different variances, we have 
\begin{align*}
\widehat{(\phi  V_1)}(t) &=  \frac{p_1 \lambda}{1+\sigma_1^2} \exp(-\frac{t^2}{2(1+\sigma_1^2)})- \frac{1}{1+\sigma_0^2}\exp(-\frac{t^2}{2(1+\sigma_0^2)}) +(1-p_1)\exp(-\frac{t^2}{2})\\
&= e^{-t^2/2} \Big[\frac{p_1 \lambda}{1+\sigma_1^2} \exp\big(t^2\frac{\sigma_1^2}{2(1+\sigma_1^2)}\big)
- \frac{1}{1+\sigma_0^2}\exp\big(t^2\frac{\sigma_0^2}{2(1+\sigma_0^2)}\big) +(1-p_1\lambda)\Big].
\end{align*}
A comparison of the power expansion ensures that, for any $x$, $|e^{x^2/2}-1-x|\leq \tfrac{x^2}{2}e^{x^2/2}$. Since $\sigma_0\leq \sigma_1$ and $p_1\lambda\sigma_1^4\leq p_0$, we obtain
\beqn 
|\widehat{(\phi  V_1)}(t)|&\leq& e^{-t^2/2}\Big[\Big|\frac{p_1\lambda}{1+\sigma_1^2} -\frac{1}{1+\sigma_0^2} +1-p_1\lambda + \frac{p_1\lambda\sigma_1^2}{2(1+\sigma_1^2)^2} -\frac{\sigma_0^2}{2(1+\sigma_0^2)^2}\Big|  + p_0\exp\Big[t^2\frac{\sigma_1^2}{2(1+\sigma_1^2)}\Big]\Big]\\
&\leq& e^{-t^2/2}\Big[\sigma_0^2\Big( \frac{1}{1+\sigma_0^2}-\frac{1}{1+\sigma_1^2} +  \frac{1}{2(1+\sigma_0^2)^2} -\frac{1}{2(1+\sigma_1^2)^2}\Big)  + p_0\exp\Big[t^2\frac{\sigma_1^2}{2(1+\sigma_1^2)}\Big]\Big]\\
&\leq & e^{-t^2/2}\Big[3\sigma^2_0\sigma_1^2  + 2p_0\exp\Big[t^2\frac{\sigma_1^2}{2(1+\sigma_0^2)}\Big]\Big]\\
&\leq & 4p_0\exp\Big[-\frac{t^2}{2(1+\sigma_0^2)}\Big]\ .
\eeqn 
Coming back to \eqref{eq:transform_t_derivative}, we conclude that 
\begin{eqnarray} \nonumber
\|(\phi  V_1)^{(k)}\mathbf{1}_{t\geq t^*}\|_2^2 &\leq& c \epsilon^2_0 \int_{\mathbb{R}} t^{2k}\exp\big[-\frac{t^2}{(1+\sigma_0^2)}\big]dt\\
 &\leq & c p_0^2 \big(\frac{1+\sigma_0^2}{2}\big)^{2k} 2^k k! = cp_0^2 \big(1+\sigma_0^2\big)^{2k} k!\ .\label{eq:phi_V1_k_upper}
\end{eqnarray}
Similarly, $\phi V_2$ is a difference of two normal distributions with different variances. Arguing as for $V_1$, we obtain
\beq \label{eq:phi_V2_k_upper}
 \|(\phi  V_2)^{(k)}\mathbf{1}_{t\geq t^*}\|_2^2\leq cp_0^2 \big(1+\sigma_0^2\big)^{2k} k! \ .
\eeq
Turning to $V_3$, we cannot directly apply \eqref{eq:transform_t_derivative} to the product $\phi V_3$. For $t\geq t*$, one has $\phi V_3(t)= p_1(1-\lambda) \big(e^{\kappa^2/2}e^{-(t-\kappa)^2/2}- e^{-t^2/2}\big)$. Let $W$ be the function 
defined on $\mathbb{R}$ by this last expression. 
\[
 \int_{t^*}^{\infty}\big[(\phi V_3)^{(k)}\big]^2(t) dt \leq \int_{-\infty}^{\infty} (W^{(k)}(t))^2dt=\int_{-\infty}^{\infty} |t^k\widehat{W}(t)|^2dt\ .   \
\]
Let us compute the Fourier transform of $W$. 
\beqn 
\big|\widehat{W}(t)\big|& =& p_1 (1-\lambda)e^{-t^2/2}\big| e^{\kappa^2/2+ i\kappa t}- 1\big|\\
&\leq & p_1 (1-\lambda)e^{-t^2/2}\big[|t|\kappa e^{\kappa^2/2}+ \big|e^{\kappa^2/2}-1\big|+ \frac{\kappa^2}{2}e^{\kappa^2/2}t^2   \big]\\
&\leq & 6 p_1 (1-\lambda)e^{-t^2/2}\kappa[1+|t|+t^2]\leq 6p_0e^{-t^2/2}[1+|t|+t^2]\ .
\eeqn 
Hence, we conclude that 
\begin{eqnarray} 
\|(\phi  V_3)^{(k)}\mathbf{1}_{t\geq t^*}\|_2^2 &\leq& c \epsilon^2_0 \int_{\mathbb{R}} (t^{2k}+ t^{2k+2}+ t^{2k+4})e^{-t^2}dt \leq  c p_0^2   (k+2)! \ .\label{eq:phi_V3_k_upper}
\end{eqnarray}
Finally, we consider $V_4$. Observe that 
\[
 |(\phi V_4)^{(k)}(t)|\leq  p_0 \big[a |\phi^{(k)}(t)|+ |b| |( P_1\phi)^{(k)}(t)|\big]\leq c p_0\big(|\phi^{(k)}(t)|+ |\phi^{(k+1)}t|\big)
\]
We then conclude 
\begin{eqnarray} 
\|(\phi  V_4)^{(k)}\mathbf{1}_{t\geq t^*}\|_2^2 &\leq& c \epsilon^2_0\Big[\int t^{2k} \phi^2(t)dt + \int t^{2(k+1)}\phi^2(t) dt\Big] \leq  c' \epsilon^2_0 (k+1)!\ .\label{eq:phi_V4_k_upper}
\end{eqnarray}
Gathering \eqref{eq:phi_V1_k_upper}, \eqref{eq:phi_V2_k_upper}, \eqref{eq:phi_V3_k_upper}, and \eqref{eq:phi_V4_k_upper}, we get
\begin{eqnarray}  \nonumber
\frac{\|\widehat{G}^{(k)}\|_2^2}{2^k k!} &\leq& cp_0^2 \frac{k^2+ (1+\sigma_0^2)^k}{2^k}\\ 
&\leq & c\frac{p_0^2}{k^2} 0.55^k \nonumber\\
&\leq &  \frac{6\log(1+0.2^2)}{\pi^2 nk^2} \ , \label{eq:upper_phi_Vi_k_grand}
\end{eqnarray}
since $\sigma_0^2\leq 0.1$ and $k\geq 5\log(\sqrt{n} p_0/2)$ and $p_0$ is small enough.

\paragraph{Step 6: Conclusion}

Coming back to \eqref{eq:Adist}, we  obtain by 
\eqref{eq:upper_phi_Vi_k_petit} and \eqref{eq:upper_phi_Vi_k_grand} that 
\[
A \leq c \sum_{k=0}^{\infty} \frac{\log(1+0.2^2)}{\pi^2 nk^2}\leq \frac{6\log(1+0.2^2)}{ n}\ .
\]
We have proved inequality~\eqref{eq:objective_A} and this concludes the proof.

\begin{proof}[Proof of Lemma \ref{lem:distance_H0_UV}]
We write $N^{\theta}_{\sigma_1}$ for the number of coordinates of $\theta$ that are larger than $\sigma_1$ in absolute value. One of the components of  the mixture distribution  $h_1$ is a centered normal distribution with variance $\sigma_1^2$ and has weight $\lambda_1 \tilde{k}_1/n\geq k_1/4n$. Hence, $N^{\theta}_{\sigma_1}$ is stochastically larger than a Binomial distribution with parameter $(n, k_1/(8n))$. By Chebychev's inequality, we have $\mu_1^{\otimes n}[N^{\theta}_{\sigma_1}< k_1/16]\leq 0.1$ for $n$ large enough. Since we assumed at the beginning of the proof of Theorem \ref{thm:lower_bound_mainuv2} that $k_1\geq 32k_0$, this implies that, with probability larger than $0.9$, 
\[
 d^2_2(\theta,\bbB_0[k_0]) \geq \frac{k_1 \sigma_1^2}{32}\geq c\frac{\sqrt{k_0\Delta}}{\log(k_0/\sqrt{n})}\ , 
\]
by definition \eqref{eq:def_sigma1} of $\sigma_1$.
\end{proof}

\begin{proof}[Proof of Lemma~\ref{lem:hatf0}]

Note first that by definition \eqref{eq:def_sigma1} of the parameters  the following conditions hold :
\beq\label{eq:COND1}
 \sigma_0^2= p_1 \lambda \sigma_1^2\ , \quad \quad 
p_1(1-\lambda) \kappa  t^*= p_0/2 \ .
\eeq

\noindent 
{\bf Step 1 : positivity of $\widehat{h}_0$ on $[0,t^*]$.} We first check that $\widehat{h}_0(t)\geq 0$ for any $t \in [-t^*, t^*]$, ie.\ ~we check that $e^{\sigma_0^2t^2 /2} \big[1 -p_1 + p_1 \widehat{h}_1(t)\big]\geq 1- p_0$. By definition \eqref{eq:def_sigma1}, we have, for $t \in [0,t^*]$, that $\sigma_0^2 t^2/2 \leq 1/2$ and that $\sigma_1^2t^2/2\leq 1/2$ and $\kappa  |t| \leq 1/2$. We get
\beqn
e^{\sigma_0^2t^2 /2} \big[1 -p_1 + p_1 \widehat{h}_1(t)\big]
&\geq& \big[ 1+  \sigma_0^2 t^2/2  \big]\big[1- \lambda p_1\sigma_1^2t^2/2 - p_1(1-\lambda) \kappa  t\big]\\
&& = 1 - \sigma_0^4t^4/4 - p_1(1-\lambda) \kappa  t - \sigma_0^2 p_1(1-\lambda) \kappa t^3/2 \quad \quad \text{(by \eqref{eq:COND1})} \\
&\geq& 1 - \sigma_0^4t^4/4 - 3p_1(1-\lambda) \kappa  t/2\ .
\eeqn 
Relying on \eqref{eq:COND1} and \eqref{eq:def_sigma1}, we have, for $t \in [0,t^*]$,
\beqn
e^{\sigma_0^2t^2 /2} \big[1 -p_1 + p_1 \widehat{h}_1(t)\big]
&\geq& 1-3p_0/4 - \frac{p_0p_1}{2^8},
\eeqn
which is positive since $p_0=k_0/(2n)$ is small enough.

\noindent
{\bf Step 2 : negativity of $\widehat{h}'_0(t)$ on $[0,t^*]$}. We have
\beqn 
p_0 \widehat{h}'_0(t)&=&  \Big[ \sigma_0^2t  \big(1 -p_1 + p_1 \widehat{h}_1(t)\big) + p_1 \widehat{h}'_1(t)
\Big]       e^{\sigma_0^2t^2 /2}, \\
\widehat{h}_1'(t)& =&  -  \lambda \sigma_1^2 te^{-\sigma_1^2t^2/2} -  (1-\lambda) \kappa e^{-\kappa |t|}\ ,
\eeqn
so that, for any $t  \in [0,t^*]$, we have 
\beqn 
p_0      e^{-\sigma_0^2t^2 /2} \widehat{h}'_0(t)&=&  -  p_1 (1-\lambda) \kappa   e^{-\kappa|t|} + t \big[-p_1 \lambda \sigma_1^2e^{-\sigma_1^2 t^2/2}+ \sigma^2_0  + \sigma_0^2p_1 \big(\widehat{h}_1(t) -1 \big)\big]\\
&<& -  p_1 (1-\lambda) \kappa e^{-\kappa t^*} /2 + t \big[-p_1 \lambda \sigma_1^2(1-\sigma_1^2 t^2)+ \sigma^2_0 \big] \quad \text{(since $\widehat{h}_1(t)\leq 1$)}\\ 
&   < &  -  p_1 (1-\lambda) \kappa  /2 + \sigma_1^2\sigma_0^2 t^3 \quad \quad \text{(by~\eqref{eq:COND1} and since $\kappa t^*\leq 1/2$)}\\
&& = \frac{1}{t^*}\Big[-\frac{p_0}{4} + \sigma_1^2\sigma_0^2 t^{*4}\Big]   =  \frac{p_0}{4t^*}\Big[ 1- \frac{1}{16}\big(1 - \big(\frac{p_0}{p_1}\big)^{1/2}\big)\big]\ , 
\eeqn
where we used again \eqref{eq:COND1} and the definition \eqref{eq:def_sigma1} of $\sigma_0$ and $\sigma_1$. Since $p_0/p_1\leq 1/2$, this last expression is nonpositive.

\bigskip

\noindent 
{\bf Step 3 : positivity of $\widehat{h}''_0(t)$ on $[0,t^*]$}. Deriving two times $\widehat{h}_0$, we get 
\beqn 
p_0 \widehat{h}^{''}_0(t)e^{-\sigma_0^2t^2 /2}&= &  \sigma_0^2 (1-p_1 + p_1 \widehat{h}_1(t))\big[\sigma_0^2 t^2 +1 \big] + 2p_1 \sigma_0^2t \widehat{h}'_1(t)  + p_1 \widehat{h}_1''(t)\ .
\eeqn
Let us bound $\widehat{h}_1$ and its derivatives for $t\in [0,t^*]$. Since, for $x\geq 0$, we have $1\geq e^{-x}\geq 1 -x$, it follows that 
\beqn 
\widehat{h}_1(t)&\geq & 1 -  \lambda \sigma_1^2\frac{t^2}{2} - p_1(1-\lambda) \kappa \ , \\
\widehat{h}'_1(t)&\geq &  -  \lambda \sigma_1^2t  - (1-\lambda) \kappa \ ,  \\
\widehat{h}_1''(t) &=&    \lambda \sigma_1^2 [\sigma_1^2t^2 -1]e^{-\sigma_1^2t^2/2} +   (1-\lambda) \kappa^2 e^{-\kappa t }
\\
& \geq& - \lambda\sigma^2_1 + \lambda  \sigma_1^4t^2- \lambda \sigma_1^6\frac{t^4}{2} + (1-\lambda)\kappa^2 (1-\kappa t) \\  
&\geq & - \lambda\sigma^2_1 + \lambda  \sigma_1^4\frac{t^2}{2} + (1-\lambda)\kappa ^2 - (1-\lambda)\kappa ^3t\ ,
\eeqn 
since $\sigma_1^2 t^{*2}\leq 1/4$ and $\kappa t^*\leq 1/2$. Gathering these bounds, we get 
\beqn 
p_0 \widehat{h}^{''}_0(t)e^{-\sigma_0^2t^2 /2} &\geq  & \sigma_0^2 \Big[1 - p_1 \lambda \sigma_1^2 \frac{t^2}{2} - p_1 (1-\lambda)\kappa t  \Big]\big(1+ \sigma_0^2 t^2\big) - 2p_1 \lambda \sigma^2_0\sigma_1^2t^2  -  2p_1  \sigma_0^2  (1-\lambda)\kappa t  \\ 
&& +  p_1 (1-\lambda) \kappa^2- p_1 (1-\lambda)\kappa ^3t -p_1  \lambda  \sigma_1^2+ p_1  \lambda \sigma_1^4\frac{t^2}{2}\\
&\geq & - p_1 (1-\lambda)\kappa  t [3\sigma_0^2+ \kappa^2]  - 3p_1 \lambda \sigma_0^2  \sigma_1^2 t^2 +  p_1(1-\lambda) \kappa^2 + p_1  \lambda \sigma_1^4 \frac{t^2}{2} \\
&& \hspace{6cm} \quad \quad \text{(since $p_1\lambda \sigma_1^2=\sigma_0^2$ and $\sigma_0^2t^{*2}\leq 1/2$)}\\
&\geq &p_1 (1-\lambda)\kappa \Big[- 3\sigma_0^2t^*  -\kappa^2 t^* +   \kappa\Big]\ , \\
\eeqn 
since $6\sigma_0^2\leq \sigma_1^2$ by~\eqref{eq:COND1} and  as we may suppose that $p_1= (k_0+\Delta/2)/n\leq 1/6$. By \eqref{eq:def_sigma1} $\kappa=4\sigma_0^2t^*/\lambda \geq 4\sigma_0^2t^*$ and $\kappa t^*= 0.5\sqrt{p_0/p_1}$ that we may suppose to be smaller than $1/4$. We have proved that $\widehat{h}''_0(t)\geq 0$ for all $t\in [0,t^*]$.

\end{proof}

\subsection{Proofs of the upper bounds}

By homogeneity, we assume henceforth that $\sigma_+=1$, which is equivalent to considering $Y'=Y/\sigma_+$ whose noise variance belongs to $[\sigma^2_-/\sigma^2_+,1]$.

\subsubsection{Proof of Theorem~\ref{thm:HCuv}}

For the sake of simplicity, we denote $t_{*}$ for $t_{*,\alpha}^{HC,\mathrm{var}}$. For any integer $q$ and any $x>0$, we denote
\beq\label{eq:definition_M}
M_{q,x}^{\theta}:= \sum_{i=1}^n |\theta_i|^q \mathbf{1}\{|\theta_i|<x\}\ .
\eeq
Let us write $\widehat{\sigma}^2$ for $\widehat{\sigma}^2(v)$ in order to simplify notation. Also denote
 \begin{eqnarray} \label{eq:upper_loss_sigma_def}
  d_{\sigma}^{-}&:=& 8\frac{k_0}{n\log(1+\frac{k_0}{\sqrt{n}})}\ , \quad \quad   d_{\sigma}^{+}:= \frac{M_{2,1/v}^{\theta}}{n}   + \frac{6N_{1/v}^{\theta}}{v^2 n}\ .
 \end{eqnarray}

The key step of this proof is to control the difference between the tail probabilities $\Phi[t/\sigma]$ and the estimated probabilities $\Phi[t/\widehat{\sigma}]$. To do this, we shall rely on the two following lemmas. The first one control the estimation error of $\sigma$ whereas the second one quantifies the error propagation for the tail probabilities. 
 
\begin{lem}\label{lem:sig}
Consider any vector $\theta$ satisfying $48 \|\theta\|_0\leq n$. For any $x>0$, the estimator $\widehat{\sigma}^2$ satisfies
\beq \label{eq:upper_loss_sigma}
 - d_\sigma^{-}\sqrt{x}\leq \widehat{\sigma}^2- \sigma^2 \leq d_\sigma^+ +d_\sigma^{-}\sqrt{x}\ ,
 \eeq
with probability larger than $1-2e^{-x}$. 
\end{lem}

\begin{lem}\label{lem:diff_phi}
Let $a>0$ and $b>0$ be such that  $-a \leq \widehat{\sigma}^2-\sigma^2\leq b$. We have
\begin{eqnarray}
\label{eq:lower_diff_Phi} 
\Phi\big(\frac{t}{\widehat{\sigma}}\big) -\Phi\big(\frac{t}{\sigma}\big)&\geq&  -\frac{ta}{\sigma^3} \phi\big(\frac{t}{\sigma}\big)\ ,\quad \quad   \text{ if }a\leq \sigma^2/2\ .\\
\label{eq:upper_diff_Phi}
\Phi\big(\frac{t}{\widehat{\sigma}}\big) -\Phi\big(\frac{t}{\sigma}\big) &\leq & \frac{tb }{2\sigma^3} \phi\big(\frac{t}{\sigma}\big) + \frac{t^3 b^2}{8\sigma^7} \phi\big[\frac{t}{\sigma}(1- \frac{b}{2\sigma^2})\big]\ ,\quad \quad  \text{ if }b\leq \sigma^2/4\ .  
\end{eqnarray}

\end{lem}

\noindent
{\bf Level of the test}. Consider any $\theta\in \mathbb{B}_0[k_0]$.  For any $t>0$, $N_t-k_0$ is stochastically bounded by a Binomial distribution with parameters $n-k_0$ and $2\Phi\big[t/\sigma\big]$. 
 Since $\Phi(t/\sigma) \leq \exp(-t^2/(2\sigma^2))$, 
 we obtain by a simple union bound that since $\sigma \leq \sigma^+=1$
\[
 \P_{\theta}[N_{t_{*}}\geq k_0+1]\leq 2(n-k_0)\exp\big(-\frac{t_{*}^2}{2\sigma^2}\big) \leq \alpha /3\ .
\]

For the statistic $N_{t}$ with smaller $t$, we first need to control the estimated variance $\widehat{\sigma}^2$. 
By Lemma \ref{lem:sig}, we have 
\beq\label{eq:lower_sigma_level}
\sigma^2 -\widehat{\sigma}^2\leq d_{\sigma}^{-}\sqrt{\log(6/\alpha)}\ ,
\eeq
on an event of  probability larger than $1-\alpha/3$. We assume henceforth that this event is true. 
In view of the definition \eqref{eq:upper_loss_sigma_def} of $d_{\sigma}^{-}$, we have $d_{\sigma}^{-}\sqrt{\log(6/\alpha)}\leq \sigma^2/2$ when $n$ is large in front of $\alpha$.

Consider any $t\in \mathbb{N}^*$. Bernstein's inequality ensures that
\[
 N_t \leq k_0 + 2(n-k_0)\Phi\big[\frac{t}{\sigma}\big]+ 2\sqrt{ n \Phi(\frac{t}{\sigma}) \log\big(\frac{\pi^2 t^2}{2 \alpha}\big) }+ \frac{2}{3}\log\big(\frac{\pi^2 t^2}{2 \alpha}\big) 
\]
outside an event of probability smaller than $(2\alpha)/(\pi^2 t^2)$. Gathering the bound \eqref{eq:lower_sigma_level} together with Lemma \ref{lem:diff_phi}, we deduce that 
\[
 \Phi\big(\frac{t}{\widehat{\sigma}}\big) -\Phi\big(\frac{t}{\sigma}\big)\geq  -\frac{td_{\sigma}^{-}\sqrt{\log(6/\alpha)}}{\sigma_{-}^3} \phi\big(t\big)
\]
and therefore 
\[
N_t \leq k_0 + 2(n-k_0)\Phi\big[\frac{t}{\widehat{\sigma}}\big] + u_{k_0,\alpha}^{HC,\mathrm{var}}\ . 
\]
Taking an union bound over all $t\in \mathbb{N}^*$, we conclude that the type I error probability of $T^{HC,\mathrm{var}}_{\alpha,k_0}$ is smaller or equal to $\alpha$.

\bigskip

\noindent
{\bf Power of the test}. Consider any $\theta$ satisfying Condition \eqref{eq:assu:cond2def} and fix $q\in [n-k_0]$. We will prove that, if $|\theta_{(k_0+q)}|$ is large enough so that Condition \eqref{the:cond2} is satisfied, the type II error probability of the test is smaller than $\beta$. We consider separately small and large values of $q$.
Define $q_+:= L_{\alpha,\beta}\big[1+ \log(C) +  \log(\frac{1}{\sigma_-}\big)+\log\big(k_0\vee \sqrt{n})\big]$, where the constant $L_{\alpha,\beta}$ will be fixed at the end of the proof.

\medskip

\noindent 
{\bf Case 1}: $q\leq q_+$. We focus on the statistic $N_{t_{*}}$ and we have :
 \[
 \P_{\theta}[T^{HC}_{\alpha,k_0} = 0] \leq \P_{\theta}[N_{t_{*}}\leq k_0].
 \]
 Restricting ourselves to the $k_0+1$ largest values of $\theta$, we get
 \[
 \P_{\theta}[T^{HC}_{\alpha,k_0} = 0] \leq \sum_{i=1}^{k_0+1} \Phi[|\theta|_{(i)}-t_{*}]\leq (k_0+1)\Phi\Big[\frac{|\theta_{(k_0+q)}|-t_{*}}{\sigma}\Big]\ .
 \]
 Since $\Phi(x)\leq e^{-x^2/2}$ for any $x>0$,the type II error probability is smaller than $\beta$ as soon as $\theta_{(k_0+q)}\geq t_{*}+ \sigma \sqrt{2\log((k_0+1)/\beta)}$. 
  For $q\leq q_+$, this condition is ensured by  \eqref{the:cond2}.

\bigskip

\noindent 
{\bf Case 2}: $q> q_+ $. Recall that for $t>0$, $N_t^{\theta}$ refers to the number of components of $\theta$ larger or equal to $t$ (in absolute value). 
Let $t$ be a positive integer larger than $4\sigma_+$ whose value will be fixed later (see \eqref{eq:def_t_0q} below) . We shall prove that, as long as $|\theta_{(k_0+q)}|\geq 2t$, the statistic $N_t$ takes large values so that the type II error probability of the test is smaller than $\beta$.

Let us first control the difference $\Phi(\frac{t}{\widehat{\sigma}})-\Phi(\frac{t}{\sigma})$ using Lemmas \ref{lem:sig} and \ref{lem:diff_phi}. Since Condition \eqref{eq:assu:cond2def}
ensures that $N_{1/v}^{\theta}\leq  C (k_0\vee \sqrt{n})$ and $v^2= 2\log(1+\frac{k_0}{\sqrt{n}})\vee 1$, it follows from the definition of $d_{\sigma}^{+}$ that 
\[
 d_{\sigma}^+\leq  3\sigma_+^2C\frac{k_0}{n\log(1+\frac{k_0}{\sqrt{n}})}+ \frac{M_{2,1/v}^{\theta}}{n} 
\]
By Cauchy-Schwarz inequality and Condition \eqref{eq:assu:cond2def}, $(M^{\theta}_{2,1/v})^2\leq n M_{4,1/v}^{\theta}\leq  nC v^{-4} (k_0\vee \sqrt{n})$. As a consequence, we have 
 \beq\label{eq:upper2_d_+}
d_{\sigma}^+ + d_{\sigma}^-\sqrt{\log\big(\frac{4}{\beta}\big)}\leq  c_{\beta} \frac{k_0}{n\log(1+\frac{k_0}{\sqrt{n}})} + \frac{M_{2,1/v}^{\theta}}{n}\leq c'_{\beta}  C\sqrt{\frac{k_0}{n\log(1+\frac{k_0}{\sqrt{n}})}}\leq \sigma^2/4\ , 
 \eeq
 for $n$ large enough in front of $\beta$, $\sigma_-$ and $C$.
 Besides, Lemma \ref{lem:sig} ensures that $\widehat{\sigma}-\sigma^2\leq d_{\sigma}^+ + d_{\sigma}^-\sqrt{\log\big(\frac{4}{\beta}\big)}$ with probability larger than $1-\beta/2$.
Under this event, we have, by Lemma \ref{lem:diff_phi}, that 
\beq\label{eq:upper_sigma_hat}
\Phi\big(\frac{t}{\widehat{\sigma}}\big) -\Phi\big(\frac{t}{\sigma}\big) \leq \frac{tM_{2,1/v}^{\theta} }{2\sigma^3} \phi\big(\frac{t}{\sigma}\big) + c''_{\beta}\frac{t^3}{\sigma_-^7} C^2\frac{k_0}{n\log(1+\frac{k_0}{\sqrt{n}})}
  \phi\big(\frac{t}{2\sigma}\big)\ . 
\eeq 
To control $N_t$, we divide the components of $\theta$ into three groups: the $(k_0+q)$ largest  (in absolute value) components of $\theta$ which are, by assumption, all larger than  than $2t$, those smaller than $1/v$, and the remaining components. Thus, the statistic
$N_t$ is stochastically lower bounded the random variable $S$, where $S$ is the sum of a Binomial random variable with parameters ($k_0+q,1- \Phi(t/\sigma))$, a Binomial random variables with parameters $(N^{\theta}_{1/v}- (k_0+q),2\Phi(t/\sigma))$, and $\sum_{i: |\theta_i|< 1/v}\mathbf{1}_{|Y_i|\geq t}$. By Bernstein's inequality, we have 
\beq\label{eq:deviation_N_t}
\P\Big[N_t \leq \E[S] - \sqrt{2\Var[S]\log\big(\frac{2}{\beta}\big)} - \frac{2}{3}\log\big(\frac{2}{\beta}\big)\Big]\leq \frac{\beta}{2}\ .
\eeq
We first control $\E[S]$. In the definition of $S$, only the expectation of $\sum_{i: |\theta_i|< 1/v}\mathbf{1}_{|Y_i|\geq t}$ is difficult to handle. 

\begin{lem}\label{lem:phidiff}
For any $t\geq 4\sigma$ and  $0 \leq x \leq t/2$, it holds that 
\beqn
\Phi(\frac{t - x}{\sigma}) + \Phi(\frac{t + x}{\sigma}) - 2\Phi(\frac{t}{\sigma}) &\geq &\frac{x^2t}{\sigma^3} \Phi\big(\frac{t}{\sigma}\big)\\
\Phi(\frac{t-x}{\sigma}) + \Phi(\frac{t+x}{\sigma}) -  2\Phi(\frac{t}{\sigma}) &
\leq& 
 \frac{x^2t}{\sigma^3}  \Phi\big(\frac{t}{2\sigma}\big)\ .
\eeqn
\end{lem}
Since $t\geq 4\geq 2/v$, it follows from Lemma \ref{lem:phidiff} above that 
\beqn
\E[S]-2(n-k_0)\Phi\big(\frac{t}{\sigma}\big)&\geq& k_0+q  -(k_0+3q)\Phi\big(\frac{t}{\sigma}\big)  + \sum_{i: |\theta_i|\leq 1/v}\theta_i^2  \frac{t}{\sigma^3}\Phi\big(\frac{t}{\sigma}\big)\\
&\geq & k_0+\frac{q}{2}  -k_0\Phi\big(\frac{t}{\sigma}\big)  +  \frac{tM_{2,1/v}^{\theta}}{\sigma^3}\Phi\big(\frac{t}{\sigma}\big)\ ,
\eeqn 
 Together with \eqref{eq:upper_sigma_hat}, we get
\beq\label{eq:lower_E_S}
\E[S]-2(n-k_0)\Phi\big(\frac{t}{\widehat{\sigma}}\big)\geq k_0+\frac{q}{2}  -k_0\Phi\big(\frac{t}{\sigma}\big)   - 2c''_{\beta}\frac{t^3\sigma_+^4}{\sigma_-^7} C^2\frac{k_0}{\log(1+\frac{k_0}{\sqrt{n}})}\Phi\big(\frac{t}{2\sigma}\big)\ .
\eeq 
Since the variance of a Binomial random variable with parameters $n$ and $p$ is upper bound by $n(p\wedge (1-p))$, we derive from Lemma \ref{lem:phidiff} that
\begin{eqnarray}
 \nonumber
\Var[S]&\leq& 2n\Phi\big(\frac{t}{\sigma}\big)+ \frac{tM_{2,1/v}^{\theta}}{\sigma_{-}^3}\Phi(\frac{t}{2\sigma})\\
\label{eq:upper_V_S}
 &\leq & 2n\Phi\big(\frac{t}{\sigma}\big)+ c \frac{t}{\sigma_{-}^3} \sqrt{\frac{Cnk_0}{\log\big(1+\frac{k_0}{\sqrt{n}})}}\Phi(\frac{t}{2\sigma})\ ,
\end{eqnarray}
where we used again Condition \eqref{eq:assu:cond2def} in the second line. In view of \eqref{eq:rejection_HC_ad}, \eqref{eq:deviation_N_t}, \eqref{eq:lower_E_S}, and \eqref{eq:upper_V_S}, the type II error probability of $T^{HC,\beta}_{\alpha,k_0}$ is smaller than $\beta$ as soon
\beqn
q&\geq& A_1+ A_2 \ , \text{ where}\\
A_1&:=& c_{\alpha,\beta} C^2t^3  \frac{1}{\sigma_-^7}  \big[k_0\vee \sqrt{n}\big]\Phi^{1/2}\big(\frac{t}{2}\big) \\
A_2&:=& \frac{4}{3}\log\big(\frac{ \pi^2 t^2 }{\alpha\beta}\big)\ .
\eeqn
Since $\Phi(x)\leq e^{-x^2/2}$, we  $A_1\leq q/2$ as soon as we choose 
\beq\label{eq:def_t_0q}
 t\geq t_q^{0,\mathrm{var}}:=  c'_{\alpha,\beta}\Big[\sqrt{\log(C)}+ \sqrt{\log\big(\frac{1}{\sigma_{-}}\big)}+ \sqrt{\log\Big(2+ \frac{k_0\vee \sqrt{n}}{q}\Big)}\Big]\ ,
 \eeq
for some constant $c'_{\alpha,\beta}$ large enough.
Fixing $t= \lceil t_q^{0,\mathrm{var}}\rceil$, we also have $A_2\leq q/2$ for $q> q_+$ which we can  assume if we take the constant $L_{\alpha,\beta}$ large enough in the definition of $q_+$.

\begin{proof}[Proof of Lemma~\ref{lem:sig}]

Since the cosinus function if bounded, Hoeffding's inequality ensures that 
\beq\label{eq:control_cos_v}
\P_{\theta}\Big[\big|\overline{\varphi}_n(v) - \overline{\varphi}(v)\big|\geq  \sqrt{\frac{2x}{n}}\Big]\leq 2e^{-x}\ ,
\eeq
for any $x>0$. Recall that the characteristic function writes as $\overline{\varphi}(v)=e^{-v^2\sigma^2/2} \sum_{i=1}^n \frac{\cos(v\theta_i)}{n}$.
Using the Taylor expansion of $\cos(x)$, we derive that $1\geq \cos(x)\geq 1- x^2/2 + x^{4}/48$ for any $x\in (-1,1)$. Considering separately the components $v \theta_i$ that are smaller or larger than one (in absolute value), we get
\[
1 - \frac{v^2M^{\theta}_{2,1/v}}{2n} +  \frac{v^4M_{4,1/v}^{\theta}}{48 n} -  \frac{2N^{\theta}_{1/v}}{n} \leq \sum_{i=1}^n \frac{\cos(v\theta_i)}{n}   \leq  1 \ ,
\]
Since $v^2M^{\theta}_{2,1/v}\leq \|\theta\|_0$ and $N^{\theta}_{1/v}\leq \|\theta\|_0$, the condition $48\|\theta\|_0\leq n$
implies that the expression in lhs is larger than $1/2$. For $x\in (0,1/2)$, $\log(1-x)\geq -x -2x^2$, implying that
\begin{eqnarray} \nonumber
  \log\big[\sum_{i=1}^n \frac{\cos(v\theta_i)}{n} \big] &\geq & - \frac{v^2M_{2,1/v}^{\theta}}{2n} - \frac{2N_{1/v}^{\theta}}{n} +  \frac{v^4M_{4,1/v}^{\theta}}{48 n} - \frac{1}{2n^2}\Big[ v^2M_{2,1/v}^{\theta} + 4\frac{N_{1/v}^{\theta}}{n} \Big]^2    \\
   &\geq&  - \frac{v^2M_{2,1/v}^{\theta}}{2n} - \frac{2N_{1/v}^{\theta}}{n} -  16 \frac{(N_{1/v}^{\theta})^2}{n^2}+  \frac{v^4M_{4,1/v}^{\theta}}{48 n} -  \frac{v^4(M_{2,1/v}^{\theta})^2}{n^2}     \nonumber \\
   &\geq & - \frac{v^2M_{2,1/v}^{\theta}}{2n} - \frac{6N_{1/v}^{\theta}}{n}  + \frac{v^4M_{4,1/v}^{\theta}}{n}\big[\frac{1}{48}- \frac{\|\theta\|_0}{n}\big] 
   \nonumber\\
   &\geq & - \frac{v^2M_{2,1/v}^{\theta}}{2n} - \frac{6N_{1/v}^{\theta}}{n} 
  \ , \label{eq:upper_cos_v_det}
\end{eqnarray}
where we used Cauchy-Schwarz inequality and $4N_{1/v}^{\theta}\leq 4\|\theta\|_0\leq n$ in the third line and $48\|\theta\|_0\leq n$ in the last line. It follows from \eqref{eq:control_cos_v} that, with probability larger than $1-2x$,
\begin{eqnarray}\nonumber
 \Big|\log\big[\overline{\varphi}_n(v)\big]+ \frac{v^2\sigma^2}{2} - \log\Big[\sum_{i=1}^n \frac{\cos(v\theta_i)}{n}\Big]\Big|&\leq  &\log\Big[1+ \frac{e^{v^2\sigma^2/2}}{\sum_{i=1}^n \cos(v\theta_i)/n}\sqrt{\frac{2x}{n}}\Big]\\ \label{eq:upper_deviation_loss}
 &\leq & 2\sqrt{\frac{2x}{n}}e^{v^2\sigma^2/2} \ ,
\end{eqnarray} 
since $\sum_{i=1}^n \cos(v\theta_i)\geq n/2$.
Together with \eqref{eq:upper_cos_v_det},  we conclude that 
\begin{align*}
- 2\sqrt{\frac{2x}{n}}e^{v^2\sigma^2/2} - \frac{v^2M_{2,1/v}^{\theta}}{2n} - \frac{6N_{1/v}^{\theta}}{n}   \leq \log\big(\sum_{i=1}^n \frac{\cos(vY_i)}{n}\big) + \frac{v^2\sigma^2}{2}  \leq  2\sqrt{\frac{2x}{n}}e^{v^2\sigma^2/2} \ .
\end{align*}
Since $2v^{-2}e^{v^2\sigma^2/2}\leq  e \frac{k_0}{n\log\big(1+ \frac{k_0}{\sqrt{n}}\big)}$, we have proved \eqref{eq:upper_loss_sigma}.

\end{proof}

\begin{proof}[Proof of lemma \ref{lem:diff_phi}]

fix $t>0$ and denote $\delta_t := \frac{t}{\widehat{\sigma}} - \frac{t}{\sigma}= \frac{t}{\sigma}\big[(1+\frac{\widehat{\sigma}^2-\sigma^2}{\sigma^2})^{-1/2} - 1]$. Applying the Taylor formula to the function $x\mapsto \Phi(x)$, we get 
\[
\Phi\big(\frac{t}{\widehat{\sigma}}\big) -\Phi\big(\frac{t}{\sigma}\big)\leq  - \delta_t \phi\big(\frac{t}{\sigma}\big) + \frac{\delta_t^2t}{2\sigma} \phi\big[\frac{t}{\sigma} + \delta_t\big]\ ,
\]
if $\delta_t<0$, whereas this difference is bounded by $0$ when $\delta_t\geq 0$. We now need to bound $\delta_t$ in terms $\widehat{\sigma}-\sigma$.
By convexity, we have $(1+x)^{-1/2}\geq 1-x/2$ for any $x>-1$. It then follows that $\delta_t\geq - \frac{tb}{2\sigma^3}$.
\beqn 
\Phi\big(\frac{t}{\widehat{\sigma}}\big) -\Phi\big(\frac{t}{\sigma}\big)&\leq&  \frac{tb }{2\sigma^3} \phi\big(\frac{t}{\sigma}\big) + \frac{t^3b^2}{8\sigma^7} \phi\big[\frac{t}{\sigma} + \delta_t\big]\\
&\leq &  \frac{tb }{2\sigma^3} \phi\big(\frac{t}{\sigma}\big) + \frac{t^3b^2}{8\sigma^7} \phi\big[\frac{t}{\sigma}\big(1- \frac{b}{2\sigma^2}\big)\big]\ .
\eeqn 
Turning to the lower bound \eqref{eq:lower_diff_Phi}, we use the convexity of the function $x\mapsto \Phi(x)$ in the second line. 
\[
\Phi\big(\frac{t}{\widehat{\sigma}}\big) -\Phi\big(\frac{t}{\sigma}\big) \geq  - \delta_t \phi\big(\frac{t}{\sigma}\big) \ ,
\]
For any $x\in [-1/2,0]$, we have $(1+x)^{-1/2}\leq 1- x$. Taking $x=\min(\tfrac{\widehat{\sigma}^2-\sigma^2}{\sigma^2},0)\geq -\frac{a}{\sigma^2}\geq - 1/2$, we obtain 
\[
\Phi\big(\frac{t}{\widehat{\sigma}}\big) -\Phi\big(\frac{t}{\sigma}\big)\geq  -\frac{ta}{\sigma^3} \phi\big(\frac{t}{\sigma}\big) \ .
\]

\end{proof}

\begin{proof}[Proof of Lemma~\ref{lem:phidiff}]
Fix $t\geq 4\sigma$ and consider the function 
$$h :x \in [0,t] \rightarrow \Phi(\frac{t-x}{\sigma}) + \Phi(\frac{t+x}{\sigma}) -  2\Phi(\frac{t}{\sigma}).$$
It holds that $h'(0) = 0$ and 
\[h''(x) = \frac{1}{\sigma^3}\Big[(x+t)\phi\big(\frac{x+t}{\sigma}\big) - (x-t)\phi\big(\frac{x-t}{\sigma}\big)\Big]\ .\]
Next, we show that  $h''$ is increasing on $[0,t/2]$. We have 
\beqn 
h'''(x)& =& \frac{1}{\sigma^3}\Big[\big(1 - \frac{(x+t)^2}{\sigma^2}\big)\phi\big(\frac{x+t}{\sigma}\big)  + \big(\frac{(x-t)^2}{\sigma^2} -1\big)\phi\big(\frac{x-t}{\sigma}\big)\Big]\\
&= &\frac{1}{\sigma^3}\Big[k\big[\big(\frac{t-x}{\sigma}\big)^2\big] - k\big[\big(\frac{t+x}{\sigma}\big)^2\big]\Big]\ , 
\eeqn 
where $k:x\mapsto (x-1)e^{-x/2}$. Observe that the function $k$ is decreasing on $[3,\infty]$. For $x\leq t/2$ and $t\geq 4\sigma$, $(t-x)/\sigma\geq \sqrt{3}$
and $h'''(x)$ is therefore positive. Relying on $h(0)=h'(0)=0$ as well as $h''(t/2)\geq h''(u)\geq h''(0)$ for any $u\in [0,x]$, we obtain by Taylor's theorem that
\beqn 
\Phi(\frac{t-x}{\sigma}) + \Phi(\frac{t+x}{\sigma}) -  2\Phi(\frac{t}{\sigma}) 
&\geq& h''(0) \frac{x^2}{2}=
 \frac{tx^2}{\sigma^3}  \phi\big(\frac{t}{\sigma}\big)\ ,\\
\Phi(\frac{t-x}{\sigma}) + \Phi(\frac{t+x}{\sigma}) -  2\Phi(\frac{t}{\sigma}) 
&\leq& h''(t/2) \frac{x^2}{2}\leq 
 \frac{tx^2}{\sigma^3}  \phi\big(\frac{t}{2\sigma}\big)\ .
 \eeqn
This concludes the proof of the lemma. 
\end{proof}


\subsubsection{Proof of Theorem \ref{thm:bulk_unknown}}
For the sake of simplicity, we simply write $s$ for $s_{k_0}^{\mathrm{var}}$ in this section. In order for the statistic $Z^{\mathrm{var}}(s)$ to be properly defined, the process  $\overline{\varphi}_n(s\xi)>0$ for $\xi\in [0,1]$ has to be positive. This will turn out to be true when the following  event holds.
\beq\label{eq:definition_A}
\cA:= \Big\{\max_{|u| \leq \sqrt{2\log(n)}} \big|\overline{\varphi}_n(u)-\overline{\varphi}(u)\big| \leq 14 \sqrt{\frac{\log(n)}{n}} \Big\}\ 
\eeq
holds.

\begin{lem}\label{lem:unform_control_characteristic_function}
For any $a>1$ and any $\theta\in\mathbb{R}^n$, we have
\beq\label{eq:control_uniform_characteristic}
 \P_{\theta}\Big[\sup_{u \in [0,\sqrt{2\log(n)}]}\big|\overline{\varphi}_n(u) - \overline{\varphi}(u)\big|\leq 7 \sqrt{a\frac{\log(n)}{n} }\Big]\leq e^{-n/2}+  2n^{1-a}\big(1 + \frac{\|\theta\|_1}{n}\big)\ .
\eeq
As a consequence $\P_{\theta}(\cA^c)\leq e^{-n/2}+  \frac{1}{n^{3}}\big(1 + \frac{\|\theta\|_1}{n}\big)$.
\end{lem}

The following proposition  characterizes the  deviations of the statistic $Z^{\mathrm{var}}(s)$. Denote $N_{1/s}^{\theta}:=  |\{i:\ |\theta_i|> s^{-1}\}|$ the number of coordinates larger than $1/s$.
\begin{prp}\label{prp:power_Z_unknown variance}
There exist numerical constants $c_1$, $c_2$, $c_3$ and $c_4$ such that the following holds. Assume that $n\geq c_1$, $\|\theta\|_0\leq n/c_2$. 
For any $x\geq 2$, the statistic $Z^{\mathrm{var}}(s)$ satisfies 
 \beqn
Z^{\mathrm{var}}(s) &\leq&  1.09|\theta|_{0}+   16 \frac{\|\theta\|_0^2}{n} + 4e^{s^2\sigma^2/2}\sqrt{nx} \\
Z^{\mathrm{var}}(s)&\geq &  c_3 N_{1/s}^{\theta} +  c_4 \sum_{i=1}^n (s\theta_i)^4 \mathbf{1}_{| s \theta_i|\leq 1}    - 4 e^{s^2\sigma^2/2}\sqrt{nx}  \ ,
\eeqn 
on the intersection of $\cA$ and an event of probability larger than $1-2e^{-x}$.
\end{prp}

 Theorem \ref{thm:bulk_unknown} is a straightforward consequence of Proposition \ref{prp:power_Z_unknown variance} and Lemma \ref{lem:unform_control_characteristic_function}. The first upper bound in the above proposition ensures that the type I error is smaller than $\alpha+\cP_{\theta}[\cA^c]$. With probability larger than $1-\beta-\P_{\theta}[\cA^c]$, the statistic $Z^{\mathrm{var}}(s)$ is larger than 
\[c_3 N_{1/s}^{\theta} +  c_4 \sum_{i=1}^n (s\theta_i)^4 \mathbf{1}_{|s\theta_i|\leq 1}    - 4\sqrt{e} (\sqrt{k_0n^{1/2}}\vee \sqrt{n} )\sqrt{\log(2/\beta)} \]
and the test rejects the null hypothesis as soon as this expression is larger than 
\[
 1.09k_{0}+  16 \frac{k^2_0}{n} + 4 \sqrt{e}(\sqrt{k_0n^{1/2}}\vee \sqrt{n} ) \sqrt{\log(2/\alpha)}
\]
This is the case if either $N_{1/s}^{\theta}$ or $\sum_{i=1}^n (s\theta_i)^4 \mathbf{1}_{|s\theta_i|\leq 1}$ is large enough, which is precisely ensured by Condition \eqref{eq:separation_test_intermediary_unknown_varianceBULK}.

 \begin{proof}[Proof of Proposition \ref{prp:power_Z_unknown variance}]

 Assume that $\|\theta\|_0\leq n/40$.
Under the event $\cA$ (defined in \eqref{eq:definition_A}), the  empirical characteristic function satisfies
\beqn
\max_{|u| \leq s} \big|e^{u^2\sigma^2/2}\overline{\varphi}_n(u)-1\big|&\leq& \max_{|u| \leq s} \big|e^{u^2\sigma^2/2}\overline{\varphi}(u)-1\big|+ e^{s^2\sigma^2/2}\max_{|u| \leq \sqrt{2\log(n)}}\big|\overline{\varphi}_n(u)-\overline{\varphi}(u)\big|\\
&\leq & \frac{1}{n}\sup_{u\leq s}|\sum_{i=1}^n (\cos(u\theta_i)-1)|  + e^{s^2\sigma^2/2} 14\sqrt{\frac{\log(n)}{n}} \\
&\leq & \frac{2\|\theta\|_0}{n}+ e^{s^2\sigma^2/2} 14 \sqrt{\frac{\log(n)}{n}}\leq 1/10\ ,
\eeqn
for $n$ large enough since $s^2\sigma^2 \leq 1+\log(n)/2$. As a consequence, the empirical characteristic functions  $\overline{\varphi}_n(u)$ is positive on $[0,s]$ and the statistic $Z^{\mathrm{var}}(s)$ is properly defined.

By definition of $P_B$, $\int_0^1 P_B(\xi)\xi^2d\xi=0$. Hence,
\beqn
Z^{\mathrm{var}}(s)&:=&  n\int_0^1 P_B(\xi) \log\big[\big(\overline{\varphi}_n(s\xi)\big)\big]d\xi\\
& = &  - n \sigma^2 \frac{s^2}{2} \int_0^1 P_B(\xi)\xi^2 d\xi  +  n \int_0^1 P_B(\xi) \log\big[e^{s^2\sigma^2\xi^2/2}\overline{\varphi}_n(s\xi)\big]d\xi\\
& = &   n \int_0^1 P_B(\xi) \log\big[e^{s^2\sigma^2\xi^2/2}\overline{\varphi}_n(s\xi)\big]d\xi .
\eeqn
To control the behavior of the statistic, we linearize the logarithm. For any $x\in [0.9,1.1]$, it holds that $|\log(1+x)-x|\leq 2x^2/3$. Hence, under the event $\cA$, the statistic $Z^{\mathrm{var}}(s)$ satisfies
 \beqn 
\Big|Z^{\mathrm{var}}(s) - n\int_{0}^{1} P_B(\xi) \big[e^{s^2\sigma^2\xi^2/2}\overline{\varphi}_n(s\xi)-1 \big]d\xi \Big| \leq \frac{2}{3} n \int_{0}^{1} |P_B(\xi)| \big[e^{s^2\xi^2\sigma^2/2}\overline{\varphi}_n(s\xi)-1 \big]^2 d\xi.
\eeqn
In the above bound, we decompose the deterministic and random quantities as follows
\beqn
A_{1,1}&:=&  \int_0^1 P_B(\xi)  \big[e^{s^2\xi^2\sigma^2/2}\overline{\varphi}(s\xi)- 1\big]d\xi\\
A_{1,2}&:=&  \int_0^1 P_B(\xi)  e^{s^2\xi^2\sigma^2/2}\big[\overline{\varphi}_n(s\xi)- \overline{\varphi}(s\xi) \big]d\xi\\
A_{2,1}&:=&  \int_0^1|P_B(\xi)| \big[e^{s^2\xi^2\sigma^2/2}\overline{\varphi}(s\xi)- 1\big]^2d\xi\\
A_{2,2}&:=&  \int_0^1|P_B(\xi)| e^{s^2\xi^2\sigma^2}\big[\overline{\varphi}_n(s\xi)- \overline{\varphi}(s\xi) \big]^2d\xi
\eeqn
so that 
\beq \label{eq:decomposition_Z_A}
\Big|Z^{\mathrm{var}}(s)/n - A_{1,1}-A_{1,2}\Big|\leq  2A_{2,1} +2 A_{2,2}.
\eeq
 
 In the remainder of the proof, we control each of these four quantities. 
 
 \medskip

 \noindent
 {\it Control of $A_{1,1}$}. Relying on the definition of $P_B(\xi)=4\xi-3$, we explicitly compute the trigonometric integral 
\[
A_{1,1} = \sum_{i=1}^n \int_{0}^1 P_B(\xi)\big[ \cos(s\xi\theta_i)-1 \big]d\xi= \sum_{i=1}^n  \Big[1 +  \frac{\sin(s \theta_i)}{s \theta_i}+ 4\frac{\cos(s \theta_i)-1}{(s\theta_i)^2} \Big]\]
Define the symmetric function $g$ by $g(0)=0$ and $g(x):= 1+ \frac{\sin(x)}{x}+ 4\frac{\cos(x)-1}{x^2}$ for $x\neq 0$.
\begin{lem}\label{lem:function_f}
The function $g$ is supported in $[0,1.09)$ and satisfies
\beq\label{eq:function_f_bound}
g(x) \geq \left\{
\begin{array}{cc}
 \frac{11}{7!} x^4 & \text{ if } |x|\leq 1,\\
 g(1) & \text{ if } |x|>1
\end{array}
\right.
\eeq 
\end{lem}
Hence, we conclude that 
\beq
A_{1,1}\leq1.09 \frac{\|\theta\|_0}{n} \ , \quad \text{ and }\quad 
A_{1,1}\geq  g(1) \frac{N_{1/s}^{\theta}}{n} + \frac{11}{7!}\sum_{i=1}^n (s\theta_i)^4 \mathbf{1}_{|s\theta_i|\leq 1}. \label{eq:upper_A11}
\eeq

\bigskip 

\noindent 
{\it  Control of $A_{2,1}$}. In this second order deterministic term, we also separately handle small and large coordinates of $\theta$. For any $\xi \in [0,1]$, it holds that
\beqn 
\big[e^{s^2\xi^2\sigma^2/2}\overline{\varphi}(s\xi)- 1\big]^2&=& n^{-2}\Big[\sum_{i=1}^n \big(\cos(s\theta_i\xi) -1\big) \Big]^2\\
&\leq& 2 \frac{(N_{1/s}^{\theta})^2}{n^2}+\frac{2}{n^2}\Big[\sum_{i=1}^n \mathbf{1}_{|\theta_i|\leq s^{-1}}\big(\cos(s\theta_i\xi)-1 \big)\Big]^2\\
&\leq & 2 \frac{(N_{1/s}^{\theta})^2}{n^2}+\frac{s^4\xi^4}{2n^2}\big[\sum_{i=1}^n \mathbf{1}_{|s\theta_i|\leq 1} \theta_i^2  \big]^2\quad\quad \text{(since $\cos(x)\geq 1-x^2/2$)}\\
&\leq & 2 \frac{(N_{1/s}^{\theta})^2}{n^2}+\frac{\|\theta\|_0}{2n} s^4\xi^4\frac{\sum_{i=1}^n \mathbf{1}_{|s\theta_i|\leq 1}  \theta_i^4}{n}.
\eeqn 
Since $\int_{0}^{1} |P_B(\xi)|d\xi\leq 2$ and $\int_{0}^{1} |P_B(\xi)|\xi^4d\xi\leq 3$, we arrive at
\beq\label{eq:upper_A21}
A_{2,1}\leq 4\frac{N_{1/s}^{\theta}\|\theta\|_0}{n^2} + \frac{3\|\theta\|_0}{2n} \frac{\sum_{i=1}^n \mathbf{1}_{|s\theta_i|\leq 1}  (s\theta_i)^4}{n}\ .
\eeq
Simply bounding $|\sum_{i=1}^n \cos\big(s\theta_i\xi\big) -1 |$ by $2\|\theta\|_0$, we also have $A_{2,1}\leq 8 [\|\theta\|_0/n]^2$.
Together with  \eqref{eq:upper_A11}, this yields
\beq \label{eq:upper_A_deterministic_bulk}
nA_{1,1}+ 2nA_{2,1}\leq 1.09 \|\theta\|_0+ 16\frac{\|\theta\|^2_0}{n}. 
\eeq
Turning to a lower bound of $A_{1,1}-2A_{2,1}$, we observe that the expressions in \eqref{eq:upper_A11} and \eqref{eq:upper_A21} counterbalance
\begin{eqnarray}
 nA_{1,1}- 2nA_{2,1}&\geq& N_{1/s}^{\theta} \big[g(1) - \frac{8\|\theta\|_0}{n}\big]  + \big[\frac{11}{7!}- \frac{3\|\theta\|_0}{n}\big]\sum_{i=1}^n (t\theta_i)^4 \mathbf{1}_{|s\theta_i|\leq 1} \nonumber \\
 \label{eq:lower_A_deterministic_bulk}
 &\geq&  N_{1/s}^{\theta}g(1)/2 +  \frac{5}{7!}\sum_{i=1}^n (t\theta_i)^4 \mathbf{1}_{|s\theta_i|\leq 1} \ , 
\end{eqnarray}
assuming that $\|\theta\|_0/n \leq \tfrac{g(1)}{16} \wedge \tfrac{2}{7!}$.

\bigskip 

\noindent 
{\it Control of $A_{1,2}$}. Let $X\sim \cN(x,\sigma^2)$. The random variable 
 $\int_{0}^{1} P_B(\xi)e^{s^2\xi^2\sigma^2/2}\cos(s\xi X)d\xi$ is smaller in absolute value than $e^{s^2 \sigma^2/2}\int_{0}^1|P_B(\xi)|d\xi\leq 2e^{s^2 \sigma^2/2}$. Hence, Hoeffding's inequality yields 
 \[
  \P\Big[|A_{1,2}|\geq  2e^{s^2 \sigma^2/2} \sqrt{\frac{2x}{n}} \Big]\leq 2e^{-x}\ ,
 \]
for any $x>0$.

\bigskip 

\noindent 
{\it Control of $A_{2,2}$}. The event $\cA$ ensures uniform bound on the difference $\overline{\varphi}_n(u)- \overline{\varphi}(u)$. As a consequence, 
\[
|A_{2,2}|\leq 14^2 e^{s^2\sigma^2}\frac{\log(n)}{n}  \int_{0}^{1} |P_B(\xi)|d\xi\leq 2\cdot 14^2 e^{s^2\sigma^2}\frac{\log(n)}{n} 
\]
Since $s^2\sigma^2\leq 1+ \log(n)/2$, this term is small in front of the first order term $A_{1,2}$ for $n$ large enough, that is 
$|A_{2,2}|\leq e^{s^2 \sigma^2/2}/\sqrt{n}$.
We conclude that, for any $x \geq 1$,  $|A_{1,2}|+2|A_{2,2}|$ is smaller than $4e^{s^2 \sigma^2/2} \sqrt{\tfrac{x}{n}}$ on the intersection of $\cA$ and an event of probability larger than $1-2e^{-x}$.
Together with \eqref{eq:decomposition_Z_A}, \eqref{eq:lower_A_deterministic_bulk} and \eqref{eq:upper_A_deterministic_bulk}, this concludes the proof.

\end{proof}

\begin{proof}[Proof of Lemma \ref{lem:unform_control_characteristic_function}]
Denote $u_{*}:=\sqrt{2\log(n)}$. Let $K$ be an integer whose value will be fixed later. By Hoeffding's inequality, we have, for any $u>0$ and $x>0$,
\[
 |\overline{\varphi}_n(u) - \overline{\varphi}(u)| \leq \sqrt{2\frac{x}{n}}\ .
\]
Fix $x>0$. Applying an union bound, we obtain, that, with probability larger than $1-2Ke^{-x}$,
\beq\label{eq:concent1}
\sup_{j=1,\ldots, K}\big|\overline{\varphi}_n(\frac{ju_*}{K}) - \overline{\varphi}(\frac{ju_*}{K})\big| \leq \sqrt{2\frac{x}{n}}\ .
\eeq
Since the function $x\mapsto \cos(x)$ is $1$-Lipschitz, we have
\[
 |\overline{\varphi}_n(u)  - \overline{\varphi}_n(u')|\leq \frac{|u-u'|}{n}\sum_{i=1}^n |Y_i|\leq \frac{|u-u'|}{n}\big(\|\theta\|_1 +\sum_{i=1}^n|\epsilon_i|\big)\ ,
\]
for any $u\neq u'$.
By the Gaussian concentration theorem, we have $\sum_{i=1}^n|\epsilon_i|\leq 2\sigma n\leq 2n$ with probability larger than $1-\exp(-n/2)$. Taking the expectation in the above inequality also leads to
\[
 |\overline{\varphi}(u)  - \overline{\varphi}(u')|\leq \frac{|u-u'|}{n}\sum_{i=1}^n \E_{\theta}[|Y_i|]\leq |u-u'|\big(1+ \frac{\|\theta\|_1}{n}\big)
\]
For any $u\in [0,u_*]$, there exists $j$ such that $|u-ju_*/K|\leq u_*/K$. With probability larger than $1-2Ke^{-x}-e^{-n/2}$, we therefore have
\[
 \sup_{u\in[0,u_*]}|\overline{\varphi}_n(u) - \overline{\varphi}(u)|\leq \sqrt{2\frac{x}{n}}+ \frac{u*}{K}\big(2\frac{\|\theta\|_1}{n}+3\big)\ .
\]
Setting $K= n [1+ \frac{\|\theta\|_1}{n}]$ and $x= a\log(n)$ for any $a>1$ yields the first result. Then, fixing $a=4$ yields the second result.

\end{proof}

\begin{proof}[Proof of Lemma \ref{lem:function_f}]
Fist we consider the behavior of $g(x)$ for $|x|\geq 2\pi$. Since $\cos^2(x)+\sin^2(x)=1$, 
\[ |g(x)-1 + 4/x^2|= \frac{|x\sin(x)+ 4\cos(x)|}{x^2}\leq \frac{\sqrt{x^2+16}}{x^2}\leq \frac{\sqrt{4\pi^2+16}}{4\pi^2}\]
 As a consequence, $g(x)\geq 0.7$ for $|x|>2\pi$. Besides, studying the behavior of the function $(-4+\sqrt{x^2+16})/x^2$ for $|x|\geq 2\pi$, we also conclude that $g(x)\leq 1.09$ for $|x|\geq 2\pi$.
 
 Then, we prove that $g$ is non-decreasing on $[0,2\pi]$. To do this, we study the sign of $h(x):= x^3g'(x)=x^2\cos(x)-5x\sin(x)+ 8(1-\cos(x))$. Since $h''(x)= x[\sin(x)-x\cos(x)]$, we observe by considering the sign of the derivative of $(h''(x)/x)$ that $h''(x)$ is first increasing from $h''(0)=0$ and then decreasing to $h''(2\pi)<0$. Thus, $h'(x)$ is therefore also increasing from $h'(0)$ and then decreasing to $h'(2\pi)<0$. Since $h(0)=0$ and $h(2\pi)>0$, this implies that $g$ is increasing on $[0,2\pi]$.

 As consequence of the two above results, we conclude that $\inf_{x>1} g(x)\geq g(1)\wedge 0.7=g(1)$.
 
 \medskip

 For $|x|$ smaller than $1$, we come back to the definition of $g(x)=\int_0^1 P_B(t)[\cos(tx)-1]dt$. By Taylor's inequality, we get $|\cos(tx)-1+(t^2x^2)/2 + (t^4x^4)/4!|\leq t^6x^6/6!$. Together with the identity $\int_0^1 P(t)t^2dt=0$, this yields 
 \[
  |g(x) - \frac{x^4}{4!}\int_0^1P(t)t^4dt\big|\leq x^6 \int_0^1|P_B(t)| \frac{t^6}{6!}dt \leq \frac{3x^6}{7!}\ , 
 \]
which allows us to conclude since $x^6\leq x^4$.

\end{proof}


\subsubsection{Proof of Theorem \ref{cor:stat_intermediaire_unknown_variance}}

For the sake of clarity with $r_l$ for $r_{k_0,l}$ in the remainder.  First observe that for all $l\in \cL_{k_0}$, $r_l\geq 4$ which implies 
\beq\label{eq:param_s_bound}
2.97 <\kappa_l \leq 3,\quad 0.99< \zeta_l \leq 1\ , \quad \text{ and } \quad \gamma_l\in (0.49,0.51)\ .
\eeq

The following proposition characterizes the deviations of the statistics $V^{\mathrm{var}}(r_l,w_l)$. 

\begin{prp}\label{prp:Vq_unknown_variance}
There exist two positive constants $c$ and $c$ such that the following hold. Assume that $n\geq c$ and consider any vector $\theta$ satisfying $\|\theta\|_0\leq c'n$.
For any $x\geq 1$ and any $l\in \cL_{k_0}$, the statistic $V^{\mathrm{var}}(r_l,w_l)$ satisfies 
 \beqn
V^{\mathrm{var}}(r_l,w_l) &\leq&  |\theta|_{0}\big[1+ \delta_l\big]+  32 \frac{\|\theta\|_0^2}{n} + 8 e^{s^2_l\sigma^2/2} \sqrt{nx} \\
V^{\mathrm{var}}(r_l,w_l)&\geq &   N_{r_l^2/w_l}^{\theta} - N_{1/w_l}^{\theta} \delta_l(1+r_l)  - 64 \frac{(N_{1/w_l}^{\theta})^2}{n}    - 8 e^{s^2_l\sigma^2/2} \sqrt{nx}  \ ,
\eeqn 
on the intersection of the event $\cA$ (defined in \eqref{eq:definition_A}) and an event of probability larger than $1-2e^{-x}$.
\end{prp}

We first prove how Theorem \ref{cor:stat_intermediaire_unknown_variance} derives from the above proposition.

\begin{proof}[Proof of Theorem \ref{cor:stat_intermediaire_unknown_variance}]
  The control of the type I error probability is a straightforward consequence of Lemma \ref{lem:unform_control_characteristic_function} and  Proposition \ref{prp:Vq_unknown_variance} together with an union bound over all $l\in \cL_{k_0}$ with weights $\tfrac{3\alpha }{\pi^2} [1+\log_2(l/l_0)]^{-2}$.

  Let us turn to the type II error.  Denote $a_0:= 1/s_{k_0}^{\mathrm{var}}$, where we recall that $s_{k_0}^{\mathrm{var}}=\sqrt{\log(ek_0/\sqrt{n})}$. For all $l\in \cL_{k_0}$, we have $l\leq k_0$ implying that $N_{1/w_l}^{\theta}\leq N_{a_0}^{\theta}$.
Consider any parameter $\theta$ satisfying $N_{a_0}^{\theta}\leq C k_0$ for some $C>2$.  
For any $l\in \cL_{k_0}$, Proposition \ref{prp:Vq_unknown_variance} ensures that 
\[
 V^{\mathrm{var}}(r_l,w_l)\geq    N_{r_l^2/w_l}^{\theta} -  C k_0 \delta_l(1+r_l)  - 64 C^2\frac{k_0^2}{n}    - 8 e^{s^2_l\sigma^2/2} \sqrt{n\log(2/\beta)}
\]
with probability larger $1-\beta-\cP_{\theta}[\cA^c]$. Since $e^{s^2_l\sigma^2/2}\leq l^{1/2}n^{-1/4}$, it follows from the definition \eqref{eq:rejection_intermediary_unknown} of $T^{I,\mathrm{var}}_{\alpha,k_0}$ that the type II error probability  smaller than $\beta+\P_{\theta}[A^c]$, if there exists $l\in \cL_{k_0}$ such that 
\beq\label{eq:condition_rejet}
 N_{r_l^2/w_l}^{\theta} - k_0 \geq      k_0\delta_l[1 +C(1+r_l)] +32(1+2C^2)\frac{k_0^2}{n}+ 16\sqrt{ln^{1/2}\log\Big( \frac{\pi^2 [1+\log_2(l/l_0)]^2}{3\alpha\wedge \beta }\Big)}\ .
\eeq
Since $l\geq l_0\geq \sqrt{n^{1/2}k_0}$, the last expression in the rhs is smaller than $c_{\alpha,\beta}l$. 
By definition \eqref{eq:param_s} of $\delta_l$ and since $r_k\geq 4$, it holds that $\delta_l k_0 \leq 4r_l\phi(r_l)k_0\leq 16 l (\tfrac{l}{k_0})^{7}\sqrt{\log(k_0/l)}$. As a consequence,  Condition  \eqref{eq:condition_rejet} simplifies as 
\[
 N_{r_l^2/w_l}^{\theta} - k_0 \geq    c_{\alpha,\beta} Cl + c' C^2\frac{k_0^2}{n}\ .
\]
which is equivalent  to 
\beq\label{eq:condition_rejet2}
 |\theta_{(k_0+q)}|\geq \frac{r_l^2}{w_l}\, \quad \text{ for some $q$ and $l\in \cL_{k_0}$ s.t.}\quad  q\geq c_{\alpha,\beta} Cl+ c' C^2\frac{k_0^2}{n}\ .
\eeq 
To conclude, it suffices to prove that, with suitable constants, Condition \eqref{eq:separation_test_intermediary_unknown_variance_simplified} enforces \eqref{eq:condition_rejet2}. Assume that $\theta$ satisfies Condition \eqref{eq:separation_test_intermediary_unknown_variance_simplified} for some $q$. Define $l(q):=\max \{l\in \cL_{k_0},\ \text{such that} \,  q\geq 2c_{\alpha,\beta} Cl \}$. Since $q$ is large in front of $C\sqrt{k_0n^{1/2}}$, this implies that $l(q)\geq l_0$. As a consequence,  for some constant $c''_{\alpha,\beta}$, it holds that
\[
 l(q) \geq c''_{\alpha,\beta} \frac{k_0 \wedge q  }{ C}\ ,
\]
and therefore
\[
 \frac{r_{l(q)}^2}{s_{l(q)}}\leq 16 \frac{\log\big(\frac{k_0}{l(q)}\big)}{\sqrt{\log\big(\frac{l_0}{\sqrt{n}}\big)}}\leq c \frac{\log\big(C/c''_{\alpha,\beta}
\big)+ \log(1 \vee \frac{k_0}{q})}{\sqrt{\log\big(\frac{k_0}{\sqrt{n}}\big)}}\leq \overline{c}_{\alpha,\beta} \log( C) \frac{1+ \log(1 + \frac{k_0}{q})}{\sqrt{\log\big(1+\frac{k_0}{\sqrt{n}}\big)}}\ .
\]
If the constant $c''_{\alpha,\beta}$ in \eqref{eq:separation_test_intermediary_unknown_variance_simplified} is set to $\overline{c}_{\alpha,\beta}$,  then $ N_{r_l^2/w_l}^{\theta} \geq k_0+q$, which implies that 
\eqref{eq:condition_rejet2} is satisfied for $l=l(q)$. Thus, the type II error probability is smaller than $\beta+\P_{\theta}[\cA^c]$.
\end{proof}

\begin{proof}[Proof of Proposition \ref{prp:Vq_unknown_variance}]

Assume that $\|\theta\|_0\leq n/40$.  Observe that for all $l\in \cL_{k_0}$, $w_l^2\leq \log(n)/2$.  Under the event $\cA$ (defined in \eqref{eq:definition_A}), the  empirical characteristic function satisfies
\beqn
\max_{|u| \leq \sqrt{\log(n)/2}} \big|e^{(u\sigma)^2/2}\overline{\varphi}_n(u)-1\big|&\leq& \max_{|u| \leq \sqrt{\log(n)/2}} \big|e^{u^2\sigma^2/2}\overline{\varphi}(u)-1\big|+ e^{\log(n)\sigma^2/4}\max_{|u| \leq \sqrt{\log(n)/2}}\big|\overline{\varphi}_n(u)-\overline{\varphi}(u)\big|\\
&\leq & \frac{1}{n}\sup_{u\leq \sqrt{\log(n)/2}}|\sum_{i=1}^n (\cos(u\theta_i)-1)|  +  14\frac{\sqrt{\log(n)}}{n^{1/4}} \\
&\leq & \frac{2\|\theta\|_0}{n}+ 14 \frac{\sqrt{\log(n)}}{n^{1/4}}\leq 1/10\ ,
\eeqn
for $n$ large enough. As a consequence, the empirical characteristic function  $\overline{\varphi}_n(u)$ is positive on $[0,w_l]$ for $l\in \cL_{k_0}$ and the statistics $V^{\mathrm{var}}(r_l,w_l)$ are properly defined.

Fix some $l\in \cL_{k_0}$. As the polynomial $P_l$ has been chosen in such a way that $\int_{-r_l}^{r_l} P_l(\xi)\phi(\xi)\xi^2d\xi=0$, we have
\beqn
V^{\mathrm{var}}(r_l,w_l)&:=&  n \int_{-r_l}^{r_l} P_l(\xi) \phi(\xi)\log\big[\overline{\varphi}_n\big(\frac{w_l}{r_l}\xi\big)\big] d\xi \\
& = &  - n\frac{w_l^2}{2r_l^2}  \sigma^2   \int_{-r_l}^{r_l} P_l(\xi) \xi^2\phi(\xi) d\xi   +  n \int_{-r_l}^{r_l} P_l(\xi)\phi(\xi) \log\big[\exp\big(\frac{w_l^2}{2r_l^2}\sigma^2\xi^2\big)\overline{\varphi}_n(w_l\xi/r_l)\big]d\xi\\
&= &  n \int_{-r_l}^{r_l} P_l(\xi) \log\big[\exp\big(\frac{w_l^2\sigma^2\xi^2}{2r_l^2}\big)\overline{\varphi}_n(w_l\xi/r_l)\big]d\xi\ ,
\eeqn 
 As for the statistic $Z^{\mathrm{var}}(s)$, we then linearize the logarithm. For any $t\in [0.9,1.1]$, $|\log(1+t)-t|\leq 2t^2/3$. Hence, under the event $\cA$,  $V^{\mathrm{var}}(r_l,w_l)$ satisfies
 \beqn 
\lefteqn{\Big|V^{\mathrm{var}}(r_l,w_l) -  \int_{-r_l}^{r_l} P_l(\xi) \phi(\xi) \left[\exp\big(\frac{w_l^2\sigma^2\xi^2}{2r_l^2}\big)\overline{\varphi}_n(\frac{w_l}{r_l}\xi)-1 \right]d\xi \Big|}&&\\& \leq& \frac{2n}{3} \int_{-r_l}^{r_l} |P_l(\xi)| \phi(\xi) \left[\exp\big(\frac{w_l^2\sigma^2\xi^2}{2r_l^2}\big)\overline{\varphi}_n(\frac{w_l}{r_l}\xi)-1 \right]^2 d\xi\ .
\eeqn
In the above bound, we decompose the deterministic and random quantities as follows
\begin{eqnarray*}
A_{1,1}&:=&   \int_{-r_l}^{r_l} P_l(\xi) \phi(\xi) \left[\exp\big(\frac{w_l^2\sigma^2\xi^2}{2r_l^2}\big)\overline{\varphi}(\frac{w_l}{r_l}\xi)- 1\right]d\xi,\\
A_{1,2}&:=&   \int_{-r_l}^{r_l} P_l(\xi) \phi(\xi) \exp\big(\frac{w_l^2\sigma^2\xi^2}{2r_l^2}\big)\big[\overline{\varphi}_n(\frac{w_l}{r_l}\xi)- \overline{\varphi}(\frac{w_l}{r_l}\xi) \big]d\xi,\\
A_{2,1}&:=&    \int_{-r_l}^{r_l} |P_l(\xi)| \phi(\xi) \left[\exp\big(\frac{w_l^2\sigma^2\xi^2}{2r_l^2}\big)\overline{\varphi}(\frac{w_l}{r_l}\xi)- 1\right]^2d\xi,\\
A_{2,2}&:=&    \int_{-r_l}^{r_l} |P_l(\xi)| \phi(\xi) \exp\big(\frac{w_l^2\sigma^2\xi^2}{2r_l^2}\big)\big[\overline{\varphi}_n(\frac{w_l}{r_l}\xi)- \overline{\varphi}(\frac{w_l}{r_l}\xi) \big]^2d\xi,
\end{eqnarray*}
so that
\beq \label{eq:decomposition_V_l}
\Big|V^{\mathrm{var}}(r_l,w_l)/n - A_{1,1}-A_{1,2}\Big|\leq  2A_{2,1} +2  A_{2,2}
\eeq
In the sequel, we control these four quantities. 
\bigskip 

\noindent 
{\it Control of $A_{1,1}$}.
We first focus on the deterministic quantity $A_{1,1}$.  Define the function $\Psi^{\mathrm{var}}_l$ by 
\beq\label{eq:def_psi_l_log} 
\Psi^{\mathrm{var}}_l(x):=  \int_{-r_l}^{r_l} P_l(\xi)  \phi(\xi) \cos\big(\frac{ w_l}{r_l} x \xi \big)d\xi\ , 
\eeq
so that $A_{1,1}= n^{-1}\sum_{i=1}^n [\Psi^{\mathrm{var}}_l(\theta_i)-\Psi^{\mathrm{var}}_l(0)]$. The following lemma provides bounds for function $\Psi^{\mathrm{var}}_l$.

\begin{lem}\label{lem:computation_psi_l}
The function $\Psi^{\mathrm{var}}_l$ satisfies 
\beq \label{eq:lower_psi_l}
\Big|\Psi^{\mathrm{var}}_l(x)- \Psi^{\mathrm{var}}_l(0) - 1  + \sqrt{2\pi}\phi\big(\frac{w_lx}{r_l}\big)\big[1+ \frac{\zeta_l}{\kappa_l -\zeta_l}\big( \frac{w_lx}{r_l}\big)^2\big] \Big|\leq \delta_l \ , 
\eeq  
for any $x\in \mathbb{R}$. This implies that
\beq\label{eq:upper_psi_l_2}
\min_{x\in \mathbb{R}} \Psi^{\mathrm{var}}_l(x)- \Psi^{\mathrm{var}}_l(0)\geq - 2\delta_l\ , \quad  \, \quad \min_{x\geq r_l^2/w_l}\Psi^{\mathrm{var}}_l(x)- \Psi^{\mathrm{var}}_l(0)\geq 1 - \delta_l(1+r_l)\ ,
\eeq
Finally, for all  $x\in [-1/w_l; 1/w_l]$, 
\beq\label{eq:lower_psi_l_small_x2}
\Psi^{\mathrm{var}}_l(x) - \Psi^{\mathrm{var}}_l(0) \geq \frac{\gamma_l}{6}\big(\frac{w_l x}{r_l}\big)^4.
\eeq
\end{lem}
 Recall that $\gamma_l\geq 1/3$ by \eqref{eq:param_s_bound}.
As a consequence, we obtain the following the bound for $A_{1,1}$
\begin{eqnarray}
A_{1,1}&\geq & \frac{N_{r_l^2/w_l}^{\theta}}{n}  - \frac{N_{1/w_l}^{\theta}}{n} (1+r_l) \delta_l + \frac{ w_l^4}{18nr_l^4}\sum_{i=1}^n \theta_i^4 \mathbf{1}_{|\theta_i|\leq w_l^{-1}}   \label{eq:lower_T11}\ ,\\
A_{1,1}&\leq&\frac{\|\theta\|_0}{n}\big[1+ \delta_l\big]\ . \label{eq:upper_T11}
\end{eqnarray}

\bigskip 

\noindent 
{\it Control of $A_{2,1}$}. As for $A_{1,1}$ we consider separately the coordinates larger than $1/w_l$ and the coordinates smaller than $1/w_l$.
\beqn 
\left[\exp\big(\frac{w_l^2\sigma^2\xi^2}{2r_l^2}\big)\overline{\varphi}(\frac{w_l}{r_l}\xi)- 1\right]^2&=& n^{-2}\big[\sum_{i=1}^n \big(\cos\big(\frac{w_l}{r_l}\theta_i\xi\big) -1\big) \big]^2\\ &\leq& 8 \frac{(N_{1/w_l}^{\theta})^2}{n^2}+\frac{2}{n^2}\big[\sum_{i=1}^n \mathbf{1}_{|\theta_i|\leq w_l^{-1}}\big(\cos(\frac{w_l\theta_i\xi}{r_l})-1\big) \big]^2\\
&\leq & 8 \frac{(N_{1/w_l}^{\theta})^2}{n^2}+\frac{w_l^4\xi^4}{2r_l^4n^2}\big[\sum_{i=1}^n \mathbf{1}_{|\theta_i|\leq w_l^{-1}} \theta_i^2  \big]^2\quad\quad \text{since }\cos(t)\geq 1-t^2/2\\
&\leq & 8 \frac{(N_{1/w_l}^{\theta})^2}{n^2}+\frac{\|\theta\|_0}{2n}\cdot \frac{w_l^4\xi^4}{r_l^4}\cdot\frac{\sum_{i=1}^n \mathbf{1}_{|\theta_i|\leq w_l^{-1}}  \theta_i^4}{n}\ .
\eeqn 
Relying on the bounds \eqref{eq:param_s_bound} for $\zeta_l$, $\gamma_l$ and $\kappa_l$, we derive that
$\int_{-r_l}^{r_l} |P_l(\xi)|\phi(\xi) d\xi\leq \gamma_l\int_{\bbR}(\zeta_l \xi^2 + \kappa_l)\phi(\xi)d\xi\leq 4$ and $\int_{-r_l}^{r_l} |P_l(\xi)|\xi^4\phi(\xi)d\xi\leq \gamma_l \int_{\bbR}(\zeta_l \xi^6 + \kappa_l\xi^4)\phi(\xi)d\xi\leq (15\zeta_l + 4 \kappa_l)\gamma_l \leq 27$ , we arrive at
\beq\label{eq:upper_T21}
A_{2,1}\leq 32\frac{(N_{1/w_l}^{\theta})^2}{n^2} + \frac{27\|\theta\|_0}{2n}\cdot \frac{w_l^4}{r_l^4}\cdot\frac{\sum_{i=1}^n \mathbf{1}_{|\theta_i|\leq w_l^{-1}}  \theta_i^4}{n}\ .
\eeq
In the first line of the above derivation, we may also simply bound $|\sum_{i=1}^n \cos\big(\frac{w_l}{r_l}\theta_i\xi\big) -1 |$ by $2\|\theta\|_0$ to obtain $A_{2,1}\leq 16\|\theta\|_0^2/n^2$.
Together with  \eqref{eq:lower_T11}, this yields
\beq \label{eq:upper_A_deterministic}
A_{1,1}+ 2A_{2,1}\leq \frac{\|\theta\|_0}{n}\big[1+ \delta_l\big] + 32\frac{\|\theta\|_0^2}{n^2}\ . 
\eeq
Turning to the lower bound of $A_{1,1}-2A_{2,1}$, we observe that the terms in $\theta_i^4$ in  \eqref{eq:upper_T21} counterbalanced by those in  \eqref{eq:upper_T11}
\begin{eqnarray}
 A_{1,1}- 2A_{2,1}&\geq& \frac{N_{r_l^2/w_l}^{\theta}}{n}   - \frac{N_{1/w_l}^{\theta}}{n}  \delta_l(1+r_l)  - 64 \frac{(N_{1/w_l}^{\theta})^2}{n^2}\ , \label{eq:lower_A_deterministic}
\end{eqnarray}
assuming that $\|\theta\|_0/n$ is small enough.

\bigskip

\noindent 
{\it Control of $A_{1,2}$}. Let $X\sim \cN(x,\sigma^2)$. The random variable 
 $\int_{-r_l}^{r_l} P_l(\xi) \phi(\xi) \exp\big(\frac{w_l^2\sigma^2\xi^2}{2r_l^2}\big)\cos(\frac{w_l}{r_l}\xi X)d\xi$ is smaller in absolute value than $e^{w_l^2 \sigma^2/2}\int_{\mathbb{R}}|P_l(\xi)|\phi(\xi)d\xi\leq 4e^{w_l^2 \sigma^2/2}$. As a consequence, Hoeffding's inequality yields 
 \[
  \P\Big[|A_{1,2}|\geq  4 e^{w_l^2 \sigma^2/2} \sqrt{\frac{2x}{n}} \Big]\leq 2e^{-x}\ ,
 \]
for any $x>0$.

\bigskip 

\noindent 
{\it Control of $A_{2,2}$}. We use the event $\cA$ (Eq.\eqref{eq:definition_A}), to uniformly bound the difference $\overline{\varphi}_n(u)- \overline{\varphi}(u)$.
\beq
|A_{2,2}|\leq 14^2  e^{w_l^2\sigma^2}\frac{\log(n)}{n}    \int_{-r_l}^{r_l} |P_l(\xi)|\phi(\xi)d\xi\leq c e^{w_l^2\sigma^2}\frac{\log(n)}{n}\ .
\eeq
Since $w_l^2\sigma^2\leq \log(n)/2$, this term is negligible is small in front of the first order term $A_{1,2}$ for $n$ large enough, that is 
\[|A_{2,2}|\leq \frac{e^{w_l^2 \sigma^2/2}}{2\sqrt{n}}\ .\]
We conclude that, for any $x \geq 1$,  $|A_{1,2}|+2|A_{2,2}|$ is smaller than $8e^{w_l^2 \sigma^2/2} \sqrt{x/n}$ on the intersection of $\cA$ and an event of probability larger than $1-2e^{-x}$. Together with \eqref{eq:decomposition_V_l}, \eqref{eq:upper_A_deterministic} and \eqref{eq:lower_A_deterministic}, this concludes the proof.

\end{proof}

\begin{proof}[Proof of Lemma \ref{lem:computation_psi_l}]
 For the sake of simplicity, we simply write $r$, $w$, $\gamma$, and $\delta$ for $r_l$, $w_l$, $\gamma_l$, and $\delta_l$ in the remainder of this proof. 
As for the function $\Psi_l$ corresponding to the statistic with known variance, we decompose the integral in $\Psi^{\mathrm{var}}_l(x)$ to obtain the Fourier transform of a standard normal distribution
\beqn 
\gamma^{-1}(\Psi^{\mathrm{var}}_l(x)-\Psi^{\mathrm{var}}_l(0))&=& \gamma^{-1}\int_{\mathbb{R}}P_l(\xi)\phi(\xi)\big(\cos(\frac{w}{r}x\xi)-1)d\xi - 2 \gamma^{-1}\int_{r}^{\infty} P_l(\xi) \big(\cos(\frac{w}{r}\xi)-1\big)d\xi\\
& =& \sqrt{2\pi}\phi\big(\frac{wx}{r}\big)\Big[\zeta-\kappa  - \zeta\big( \frac{wx}{r}\big)^2\Big]+\kappa -\zeta - 2 \gamma^{-1}\int_{r}^{\infty} P_l(\xi)\phi(\xi)\big( \cos(\frac{w}{r}\xi)-1\big)d\xi\ ,
\eeqn 
where we used the integration by part in the second line. Let us now upper bound the second expression in the rhs. 
\beqn 
\gamma^{-1}\Big|\int_{r}^{\infty} P_l(\xi)\phi(\xi) \big( \cos(\frac{w}{r}\xi)-1\big)d\xi\Big|&\leq& 2|\kappa| \int_{r}^{+\infty} \phi(\xi) d\xi + 2|\zeta|\int_r^{\infty} \phi(\xi)\xi^2 d\xi \\
&\leq & 2(|\kappa|+ |\zeta|)\frac{\phi(r)}{r}+ 2|\zeta|r\phi(r)\\
&\leq& \frac{8\phi(r)}{r}+ 2r\phi(r)\ ,
\eeqn
where we used again the integration by part and \eqref{eq:param_s_bound}. Gathering the two above inequalities yields
\[
\Big|\Psi^{\mathrm{var}}_l(x)- \Psi^{\mathrm{var}}_l(0) - 1  + \sqrt{2\pi}\phi\big(\frac{wx}{r}\big)\big[1+ \frac{\zeta}{\kappa -\zeta }\big( \frac{wx}{r}\big)^2\big] \Big|\leq \frac{4}{\kappa -\zeta} \big( r+ 4r^{-1} \big) \phi(r) =\delta \ , 
\]
We have proved  \eqref{eq:lower_psi_l}. Consider the function $h:u\mapsto \sqrt{2\pi}\phi(u)[1+ \frac{\zeta}{\kappa -\zeta }u^2]$ defined on $\mathbb{R}^+$. Studying the sign of its derivative, we observe that it is maximized at $u_*^2=  \frac{3\zeta-\kappa}{\zeta}= \frac{2r^3\phi(r)}{\zeta}\leq 1/2$ since $r\geq 4$. As a consequence of \eqref{eq:param_s_bound}, we obtain
\[
 h(u) \leq  h(u_*)\leq \big[1 - \frac{u_*^2}{2} + \frac{u_*^4}{8}\big][1+ \frac{\zeta u_*^2 }{\kappa -\zeta }]\leq 1+   \frac{u_*^2}{2}\big( \frac{2\zeta}{\kappa-\zeta} - 1 \big) + \frac{u_*^4}{4} \leq 1 + \frac{3}{4}u_*^4\leq 1+ 4r^6\phi^2(r)\ ,
\]
where $4r^6\phi^2(r)\leq \delta$ since $r\geq 4$. Plugging this bound into \eqref{eq:lower_psi_l} yields the first part of \eqref{eq:upper_psi_l_2}. For $x\geq r^2/w$, we have, since $r\geq 4$,
\[
 \Psi^{\mathrm{var}}_l(x)-\Psi^{\mathrm{var}}_l(0)\geq 1  - \delta - \sqrt{2\pi}\phi(r)[1+ r^2\gamma ]\geq 1 - \delta (1+r)\ ,
\]
implying the second part of \eqref{eq:upper_psi_l_2}.
\medskip 

It remains to control $\Psi^{\mathrm{var}}_l(x) - \Psi^{\mathrm{var}}_l(0)$ for $x\in [-1/w,1/w]$. Denoting $a=wx/r$, we have $|a|\leq 1/r\leq 1/4$.
Taylor's inequality yields $\big|\cos(t) - 1+ \tfrac{t^2}{2}-  \tfrac{t^4}{4!}\big|\leq \tfrac{t^6}{6!}$, for any $|t|\leq 1$. Plugging this bound in the definition of $\Psi^{\mathrm{var}}_l(x)$, we get 
\[
\Big|\Psi^{\mathrm{var}}_l(x)- \Psi^{\mathrm{var}}_l(0)+ \int_{-r}^{r}P_l(\xi)\phi(\xi) a^2\frac{\xi^2}{2}d\xi - \int_{-r}^{r}P_l(\xi)\phi(\xi) a^4\frac{\xi^4}{4!}d\xi \Big|\leq \int_{-r}^{r}\phi(\xi)|P_l(\xi)| a^6\frac{\xi^6}{6!}d\xi
\]
Recall that $P_l$ has been defined in such a way that $\int_{-r}^{r}P_l(\xi)\xi^2d\xi=0$. It then follows that 
\beqn
\Psi^{\mathrm{var}}_l(x)- \Psi^{\mathrm{var}}_l(0)&\geq&  \gamma \frac{a^4}{4!}\Big[ \big(\zeta - \kappa \frac{a^2}{30}\big)  \int_{-r}^{r}\phi(\xi)\xi^6 d\xi  - \kappa \int_{\mathbb{R}}\phi(\xi)\xi^4d\xi - \zeta\frac{a^2}{30}\int_{\mathbb{R}}\phi(\xi)\xi^8d\xi\Big]\\ 
&\geq & \gamma \frac{a^4}{6}\ ,
\eeqn
where we have used that $r>4$, $|a|\leq 1/4$ and  the bounds \eqref{eq:param_s_bound}. We have proved \eqref{eq:lower_psi_l_small_x2}.

\end{proof}

\subsubsection{Proof of Corollary~\ref{cor:power_combined2ubv2}}

We first state the following analysis of the test $T^{C,\mathrm{var}}_{\alpha,k_0}$. 

\begin{cor}\label{cor:power_combined2ubv}
Fix any $\xi\in (0,1)$.
 There exist  positive constants $c$, $c'$, $c''_{\alpha,\beta,\xi}$ and $c'''_{\alpha,\beta,\xi}$ such that the following holds. Consider any $k_0\leq n^{1-\xi}$ and $n\geq c$. 
 Then, for any $\theta\in \bbB_{0}[k_0]$, one has 
 \[
  \P_{\theta,\sigma}[T^{C,\mathrm{var}}_{\alpha,k_0}=1]\leq \alpha + \frac{2\|\theta\|_1}{n^4\sigma_+}+ \frac{2}{n^3}\ . 
 \]
Moreover,
 $\P_{\theta,\sigma}[T^{C,\mathrm{var}}_{\alpha,k_0}=1]\geq 1-\beta-  \frac{2\|\theta\|_1}{n^4\sigma_+}- \frac{2}{n^3}$ for any vector $\theta$ satisfying $\|\theta\|_0\leq c'n$ and
 \beq\label{eq:upper_adaptatif_linfiniuv2}
 |\theta_{(k_0+q)}| \geq c''_{\alpha,\beta,\xi}\sigma_+\psi_{k_0,q}^{\mathrm{var}}\ , \text{ for some } q\in [1,c'n-k_0]\ .
 \eeq
  Also,  $\P_{\theta,\sigma}[T^{C,\mathrm{var}}_{\alpha,k_0}=1]\geq 1-\beta- \frac{2\|\theta\|_1}{n^4\sigma_+}- \frac{2}{n^3}$ for any vector $\theta$ satisfying  
 \beq\label{eq:upper_adaptatif_l2uv2}
  \theta\in \mathbb{B}_0(k_0+\Delta)\quad  \text{ and }\quad  d^2[\theta,\mathbb{B}_0(k_0)] \geq c'''_{\alpha,\beta,\xi} \sigma_+^2\Delta(\psi_{k_0,\Delta}^{\mathrm{var}})^2\ ,\, \text{for some $\Delta\in [1,c'n-k_0]$.}
 \eeq
\end{cor}

In the sequel, $\P_U$ stands for the probability with respect to $U$.  As we did for $Y$, we denote $\cS[U,\theta]$ denote the coordinates $i$ such that $|\theta_i|>(U+1)\sigma_+ n^2$. Also, we write $\widetilde{Y}(\cS[U, \theta]):=(Y_i), i\in ([n]\setminus \cS[U, \theta])$ and $\widetilde{\theta}(\cS[U, \theta]):=(\theta_i), i\in ([n]\setminus \cS[U, \theta])$. Note first that
$\|\widetilde{\theta}(\cS[U, \theta])\|_1 \leq 2\sigma_+ n^3$. Let us call $\overline{T}^{C,U}_{\alpha,k_0 - |\cS(U,\theta)|}$ the oracle test which is  applied to the size $ n - |\cS(U,\theta)|$ vector $\widetilde{Y}(\cS(U,\theta))$ when $k_0 \geq |\cS(U,\theta)|$. Conditionally on $U = u$ : 
\begin{itemize}
\item If $\theta\in \bbB_{0}[k_0]$ then  $|\cS[u, \theta]| \leq k_0$,  then $\widetilde{\theta}(\cS[u, \theta])\in \bbB_0[k_0-|\cS[U,\theta|]$. We know from Corollary~\ref{cor:power_combined2ubv} that
\begin{equation}\label{eq:levlim}
  \P_{\theta,\sigma}[\overline{T}^{C,u}_{\alpha,k_0 - |\cS[u, \theta]|}=1]\leq \alpha + 2\frac{2n^3 +n}{n^4}  \leq \alpha +  \frac{6}{n}\  . 
\end{equation}
\item If $|\cS[U, \theta]|> k_0$, then the test reject the null with probability one. Consider the case where $|\cS[U, \theta]|\leq  k_0$. If 
$\theta$ satisfies  \eqref{eq:upper_adaptatif_linfiniuv}, then $\widetilde{\theta}(\cS[u, \theta])$ satisfies the counterpart of  Condition \eqref{eq:upper_adaptatif_linfiniuv2} for a test of sample size $n-|\cS[U,\theta|$. Hence, it follows  from Corollary~\ref{cor:power_combined2ubv} that
\begin{equation}\label{eq:powlim}
  \P_{\theta,\sigma}[\overline{T}^{C,u}_{\alpha,k_0 - |\cS[u, \theta]|}=0]\leq \beta + \frac{6}{n}\ . 
 \end{equation}
Similarly, the test rejects with high probability when Condition \eqref{eq:upper_adaptatif_l2uv} is satisfied.
\end{itemize}
Integrating these bounds with respect to $\P_U$, we conclude that the type I error probability of the oracle test is smaller than $\alpha+6/n$. Besides, for any $\theta$ satisfying either \eqref{eq:upper_adaptatif_linfiniuv} or \eqref{eq:upper_adaptatif_l2uv}, the probability of rejection is larger than $1-\beta-6/n$.

It remains to prove that the trimmed test $\overline{T}^{C,\mathrm{var}}_{\alpha,k_0}$ agrees with the oracle test $\overline{T}^{C,U}_{\alpha,k_0 - |\cS(U,\theta)|}$ except on an event of small probability. 

\begin{lem}\label{lem:U}
Fix any $\theta\in \mathbb{R}^n$. Define the events $\cE$ and $\cE'$  by 
\begin{eqnarray*}
\cE&:=&\{\|Y-\theta\|_{\infty}\leq 2\sigma_+ \sqrt{\log(n)}\}\ ,\\
\cE'&: = &\{(U+1)\sigma_+n^2 \not\in\bigcup_{i \leq n} [\theta_i - 2\sigma_+ \sqrt{\log(n)}; \theta_i+2\sigma_+ \sqrt{\log(n)}]\}\ .
 \end{eqnarray*}
Then, $\P_{\theta,\sigma}[\cE]\geq 1-1/n$ and $P_{U,\epsilon}(\cE') \geq 1- 4\sqrt{\log(n)}/n$.
\end{lem}

\begin{proof}[Proof of Lemma~\ref{lem:U}]
It follows from the Gaussian concentration inequality together with an union bound that $\P_{\theta,\sigma}[\cE]\geq 1-1/n$. Turning to $\cE'$, we observe the probability of the event 
\[
 \Big\{(U+1)\sigma_+n^2 \in  [\theta_i - 2\sigma_+ \sqrt{\log(n)}\}; \theta_i+2\sigma_+ \sqrt{\log(n)}\Big\}
\]
is less than $4\sqrt{\log(n)}/n^2$. Taking an union bound over all $i$, we conclude that $\cP[\cE']\geq 1-4\log(n)/n$. 
\end{proof}
Note that $\cE\cap \cE' \subset \{\cS[ U; Y]= \cS[U; \theta]\}$. As a consequence, outside an event of probability less than $5\sqrt{\log(n)}/n$, the oracle test and the trimmed test agree. This concludes the proof.

\begin{proof}[Proof of Corollary~\ref{cor:power_combined2ubv}]
This corollary is a direct consequence of Theorems~\ref{thm:HCuv}, \ref{thm:bulk_unknown} and \ref{cor:stat_intermediaire_unknown_variance}. The constants $C$ in Theorems \ref{thm:HCuv} and \ref{cor:stat_intermediaire_unknown_variance} are chosen large enough so that when  Conditions~\eqref{eq:assu:cond2def} or Conditions \eqref{eq:separation_test_intermediary_unknown_variance_simplified2} are not satisfied, then  Condition \eqref{eq:separation_test_intermediary_unknown_varianceBULK} in Theorem \ref{thm:bulk_unknown} is met.
We focus on  \eqref{eq:upper_adaptatif_linfiniuv}, the result \eqref{eq:upper_adaptatif_l2uv} being proved similarly.

\noindent
{\bf Case $k_0\leq \sqrt{n}$ and $\Delta\leq \sqrt{n}$.} If Condition~\eqref{eq:assu:cond2def} holds, then the bound follows from Theorem \ref{thm:HCuv}. If Condition~\eqref{eq:assu:cond2def} does not hold, then the test $T_{\alpha/3,k_0}^{B}$ rejects the null hypothesis with high probability by Theorem~\ref{thm:bulk_unknown}.

\noindent
{\bf Case $k_0\leq \sqrt{n}$ and $\Delta\geq \sqrt{n}$.} Theorem~\ref{thm:bulk_unknown} leads to the desired bound.

\noindent
{\bf Case $k_0\geq \sqrt{n}$ and $\Delta\leq \sqrt{k_0n^{1/2}}\vee \frac{k_0^2}{n}$.}  If Condition~\eqref{eq:assu:cond2def} is not satisfied, then $T_{\alpha/3,k_0}^{B}$ rejects the null hypothesis with high probability. Otherwise, Theorem \ref{thm:HCuv} ensures that the Higher-Criticism test rejects the null hypothesis with high probability if $\theta^2_{(k_0+\Delta)}$ is large in front of $\log(1+k_0/\Delta)$. For $\Delta \in (\sqrt{k_0n^{1/2}},\frac{k_0^2}{n})$ we have 
\[\log(1+k_0/\Delta)\leq c_{\xi}\frac{\log^2\big(1+\frac{k_0}{\Delta}\big)}{\log\big(1+\frac{k_0}{\sqrt{n}}\big)}\ , \quad \quad\text{ since }k_0\leq n^{1-\xi}\ . 
\]

\noindent
{\bf Case $k_0\geq \sqrt{n}$ and $k_0 \geq \Delta\geq \sqrt{k_0n^{1/2}}\vee \frac{k_0^2}{n}$.} Theorem~\ref{cor:stat_intermediaire_unknown_variance} leads to the desired bound if  Condition~\eqref{eq:separation_test_intermediary_unknown_variance_simplified2} is satisfied. Otherwise, Theorem~\ref{thm:bulk_unknown} enforces the desired result.

\noindent
{\bf Case $k_0\geq \sqrt{n}$ and $ \Delta\geq k_0$.} This is again a consequence of Theorem~\ref{thm:bulk_unknown}.

\end{proof}

\end{document}